\documentclass{amsart}

\usepackage{hyperref}

\usepackage{esint}
\usepackage{amssymb}
\usepackage{tikz}
\usepackage{graphicx}
\usepackage{amsrefs,enumitem}

\usepackage{fancyhdr}
\usepackage{time}

\newtheorem{theorem}{Theorem}
\newtheorem{lemma}[theorem]{Lemma}
\newtheorem{corollary}[theorem]{Corollary}
\newtheorem{proposition}[theorem]{Proposition}
\theoremstyle{definition}

\newtheorem{remark}[theorem]{Remark}

\newtheorem{definition}[theorem]{Definition}

\newcommand{\eref}[1]{(\ref{e.#1})}
\newcommand{\tref}[1]{Theorem \ref{t.#1}}
\newcommand{\lref}[1]{Lemma \ref{l.#1}}
\newcommand{\pref}[1]{Proposition \ref{p.#1}}
\newcommand{\cref}[1]{Corollary \ref{c.#1}}
\newcommand{\fref}[1]{Figure \ref{f.#1}}
\newcommand{\sref}[1]{Section \ref{s.#1}}
\newcommand{\partref}[1]{\ref{part.#1}}
\newcommand{\dref}[1]{Definition \ref{d.#1}}

\numberwithin{theorem}{section}
\numberwithin{equation}{section}

\newcommand{\Z}{\mathbb{Z}}
\newcommand{\R}{\mathbb{R}}

\newcommand{\C}{\mathbb{C}}

\newcommand{\grad}{\nabla}

\def\XXint#1#2#3{{\setbox0=\hbox{$#1{#2#3}{\int}$ }
\vcenter{\hbox{$#2#3$ }}\kern-.6\wd0}}

\newcommand{\ep}{\varepsilon}
\newcommand{\osc}{\mathop{\textup{osc}}}

\newcommand{\id}{\operatorname{1}}

\begin{document}

\title[The Alt-Caffarelli energy functional in inhomogeneous media]{Limit Shapes of local minimizers for the Alt-Caffarelli energy functional in inhomogeneous media}

\author{William M Feldman}
\email{feldman@math.uchicago.edu}
\address{Department of Mathematics, The University of Chicago, Chicago, IL 60637, USA}
\keywords{contact angle hysteresis, pinning, free boundaries, homogenization}

\maketitle

\begin{abstract}
This paper considers the Alt-Caffarelli free boundary problem in a periodic medium.  This is a convenient model for several interesting phenomena appearing in the study of contact lines on rough surfaces, pinning, hysteresis and the formation of facets. We show the existence of an interval of effective pinned slopes at each direction $e \in S^{d-1}$.  In $d=2$ we characterize the regularity properties of the pinning interval in terms of the normal direction, including possible discontinuities at rational directions.  These results require a careful study of the families of plane-like solutions available with a given slope. Using the same techniques we also obtain strong, in some cases optimal, bounds on the class of limit shapes of local minimizers in $d=2$, and preliminary results in $d \geq 3$.  
\end{abstract}

\tableofcontents

\section{Introduction}

Consider the free boundary problem in a heterogeneous medium,
\begin{equation}\label{e.fb0}
  \begin{cases}
    \Delta u = 0 & \mbox{in } \{ u > 0 \} \\
    |\grad u| = Q(x/\ep) & \mbox{on } \partial \{ u > 0 \}.
  \end{cases}
\end{equation}
This is the Euler-Lagrange equation associated with the Alt-Caffarelli-type energy functional,
\begin{equation}\label{e.energyep}
 E_\ep(u) = \int |\grad u|^2 + Q(x/\ep)^2{\bf 1}_{\{u>0\}} \ dx.
 \end{equation}
 Our main physical motivation for studying this problem is the connection with capillarity problems on a rough surface, in that case the dimension of interest is $d=2$. Dimension $d=3$ is also of interest in connection with problems involving flows in porous media.
 
The global energy minimizers, generally speaking, converge as $\ep \to 0$ to the global minimizer of
\begin{equation}\label{e.gammalimit}
 E_0(u) = \int |\grad u|^2 + \langle Q^2 \rangle {\bf 1}_{\{u>0\}} \ dx.
 \end{equation}
We are interested instead in the limiting shape of local minimizers or critical points.  In that case, formally speaking, the scaling limit is a free boundary problem of the form
\begin{equation}\label{e.fb1}
  \begin{cases}
    \Delta u = 0 & \mbox{in } \{ u > 0 \} \\
    |\grad u| \in [Q_*(n_x),Q^*(n_x)] & \mbox{on } \partial \{ u > 0 \}
  \end{cases}
  \end{equation}
  where $n_x$ is the inward unit normal to $\{u>0\}$ at $x$.   The interval of stable slopes, or pinning interval, $[Q_*(n),Q^*(n)]$ defined for each $n \in S^{d-1}$ is determined by a cell problem.  We call \eref{fb1} the pinning free boundary problem, or the pinning problem.  
  
  In this paper we will show that, in $d=2$, solutions of \eref{fb1} correspond to local minimizers of \eref{energyep} for $\ep>0$ small.  There are some important additional restrictions on the result which will be explained below.  In the process we study the fine properties of $Q_*,Q^*$, directions of continuity and discontinuity.  These properties give qualitative information on the structure of the free boundary.  In a forthcoming companion article we show how discontinuities in $Q^*,Q_*$ are responsible for formation of facets in the free boundary under a monotone quasi-static motion \cite{FeldmanQS}.  It was already discovered by Caffarelli and Lee~\cite{CaffarelliLee}, and explored further by the author and Smart \cite{FeldmanSmart}, that, in a convex setting, discontinuities in $Q^*$ result in facets in the minimal supersolution of \eref{fb1}.
  
One of the most interesting aspects of this problem is that macroscopic hysteresis arises from inhomogeneities in a microscopic system which is reversible.  This is well known in the physics literature, and has been explored in some aspects in the mathematical literature \cite{DeSimoneGrunewaldOtto,AlbertiDeSimone,CaffarelliMellet2,Kim,KimMellet}.  
  
  The interval of stable slopes, or pinning interval, $[Q_*(e),Q^*(e)]$ defined for each $e \in S^{d-1}$ is determined by the following cell problem. We say $p \in \R^d \setminus \{0\}$ is a stable or pinned slope if there exists a solution on $\R^d$ of
  \begin{equation}\label{e.cell0}
\left\{
\begin{array}{ll}
 \Delta u = 0 & \hbox{ in } \ \{u >0\} \vspace{1.5mm}\\
|\grad u| = Q(x) & \hbox{ on } \ \partial \{u >0\} \vspace{1.5mm}\\
\sup_{\R^d} |u(x) - (p \cdot x)_+| < +\infty.
\end{array}
\right.
\end{equation}
  Our first main result is on the qualitative properties of $Q^*,Q_*$ as functions of the normal direction.   
  \begin{theorem}\label{t.mainQ}
  The following properties holds for the pinning interval endpoints:
    \begin{enumerate}[label=(\roman*)]
  \item\label{part.mainQexist} Let $e \in S^{d-1}$ there exist $Q_*(e) \leq \langle Q^2 \rangle^{1/2} \leq  Q^*(e)$, respectively lower and upper semicontinuous in $e$, such that, there exists a global solution of \eref{cell0} with slope $p = \alpha e$ if and only if $\alpha \in [Q_*(e),Q^*(e)]$.  
  \item For any $\alpha \in (Q_*(e),Q^*(e)) \cup  \langle Q^2 \rangle^{1/2}$ there exist solutions of \eref{cell0} which are local energy minimizers.
  \item \label{part.mainQcont} When $d=2$, $Q^*,Q_*$ are continuous at irrational directions $e \in S^{1} \setminus \R\Z^2$.
  \item \label{part.mainQcont1} When $d=2$, directional limits of $Q^*,Q_*$ exist at rational directions $e \in S^{1}  \cap \R\Z^2$.
  \item\label{part.mainQdisc} Given any $k$-dimensional rational subspace, $1\leq k \leq d-1$, there exists $Q$ such that $Q^*,Q_*$ are discontinuous on that subspace.
  \item\label{part.mainQint} There exist $Q$ such that the pinning interval is nontrivial at every direction, $\inf_{S^{d-1}}(Q^*-Q_*) \geq \delta >0$.
  \end{enumerate}
  
  \end{theorem}
  Qualitative properties of $Q^*,Q_*$ are important to study, both for our homogenization result, and to understand the structure of the free boundary for solutions to \eref{fb1}.  As explained above, there is a direct connection between the formation of facets in the free boundary and the discontinuities in $Q_*,Q^*$.

  In a previous work with Smart \cite{FeldmanSmart} we considered the scaling limit of a free boundary problem on the lattice $\Z^d$ analogous to \eref{fb0}.  In that case we were able to find an explicit formula for $I(p) = [Q_*(p),Q^*(p)]$.  There $I(p)$ has jump discontinuities along every rational subspace of co-dimensions $1 \leq k \leq d-1$.  Still $I(p)$ satisfies a continuity property, easiest explained in $d=2$, left and right limits of $I(p)$ exist at every $p$.  Our expectation is that, generically, a similar structure is present here.  \tref{mainQ} gives examples supporting the presence of discontinuities, and proves the sharp continuity result in $d=2$.
  
  The key in the proof for parts~\ref{part.mainQcont} and \ref{part.mainQcont1} of \tref{mainQ} is the construction of certain foliations of $\R^2$ by the free boundaries of global plane-like solutions.  These foliations allow to construct approximate solutions at nearby directions by sewing together solutions along the foliation.  One of the major difficulties we face, and it is fundamental to the problem, is that these are not truly foliations. At irrational directions there may be gaps in the foliations, we are able to show that the gaps are localized in a certain sense which still allows for the sewing procedure.  At rational directions the foliations keep an orientation which only allows to construct approximate plane-like solutions on one side.  As we will see below this issue can potentially lead to additional facets at rational directions, which we do not yet fully understand.

  Now we discuss the limit \eref{fb0} to \eref{fb1} for general, not asymptotically linear, solutions.  This limit is slightly unusual from the perspective of homogenization theory in that there is no uniqueness for the limiting equation \eref{fb1}.  Nonetheless it is precisely this non-uniqueness that explains the multitude of local minimizers for the rough coefficient energy $E_\ep$.  
  
  Our main result has two parts. The first part is that limits of solutions to \eref{fb0} solve \eref{fb1}, this type of statement is usually all that is needed for typical elliptic homogenization problems.  In fact it has already been considered by Caffarelli and Lee~\cite{CaffarelliLee}, and it is also a corollary of the result of Kim~\cite{Kim} on a related dynamic problem.  We include the statement for completeness not for novelty.  
  \begin{theorem}[Caffarelli-Lee~\cite{CaffarelliLee}, Kim~\cite{Kim}]\label{t.kcl}
Let $U \subset \R^d$ open. Suppose that $u^\ep$ is a bounded sequence of solutions to \eref{fb0} in $U$.  Then $u^\ep$ are uniformly Lipschitz and if, along a subsequence, $u^\ep \to u$ locally uniformly in $U$, then $u$ solves \eref{fb1} in the viscosity sense.
  \end{theorem}
  Note that the full statement of \tref{kcl} that we make here is not proven in \cite{CaffarelliLee}, however almost all of the main ideas of the proof can be found there.  Of course it is possible, with only the information of \tref{kcl}, that the class of limits of $u^\ep$ satisfy some stronger condition than just \eref{fb1}.  
  
  The second part of the homogenization result, which is completely new in this paper, is to show that for an arbitrary solution $u$ of \eref{fb1} there exists a sequence of solutions $u^\ep$ of \eref{fb0} converging to $u$.  In analogy with the language of $\Gamma$-convergence we call this the existence of a recovery sequence for $u$.  Furthermore, we would like this sequence $u^\ep$ to be local minimizers of the energy functional \eref{energyep}.  Actually we do not prove such a general result.  We give a sufficient condition here, we leave to future work to answer the question of whether such a condition is necessary.  
  
  We need to augment the information provided by the upper and lower endpoints of the pinning interval with additional microscopic information.  We call this the continuous part of the pinning interval
  \begin{equation}\label{e.qcontd1}
   [Q_{*,cont}(e), Q^*_{cont}(e)] \subset [Q_*(e),Q^*(e)]. 
   \end{equation}
  The definition is rather technical so we drop some of the details, the full exposition can be found in \sref{contpart}.  Define $Q_{*,cont}(e)$ to be the smallest slope $\alpha$ such that, for sufficiently small $\delta>0$, and any smooth test function $\varphi$ with $ |\grad \varphi  - \alpha e| \leq \delta$, there exists a recovery sequence of subsolutions $\varphi^\ep$ solving \eref{fb0} and $\varphi^\ep \to \varphi_+$ as $\ep \to 0$.  Then $Q^*_{cont}$ is defined similarly in terms of recovery sequences for smooth supersolutions with approximately constant gradient.

  \begin{figure}
  \begin{tabular}{ll}
  \begin{tabular}{l}
  \includegraphics[width=.4\textwidth]{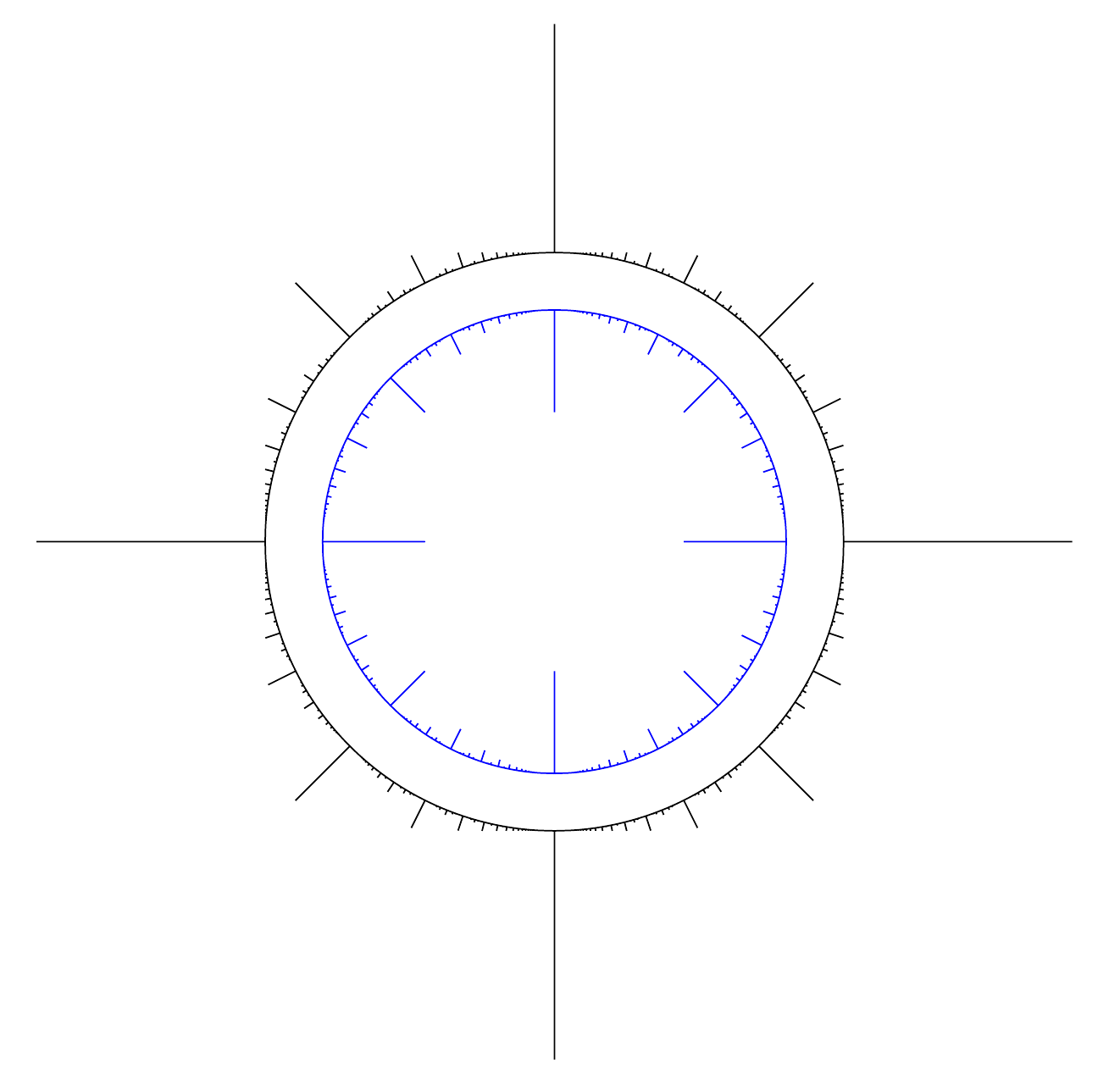}
  \end{tabular} &
  \begin{tabular}{l}
    \includegraphics[width=.4\textwidth]{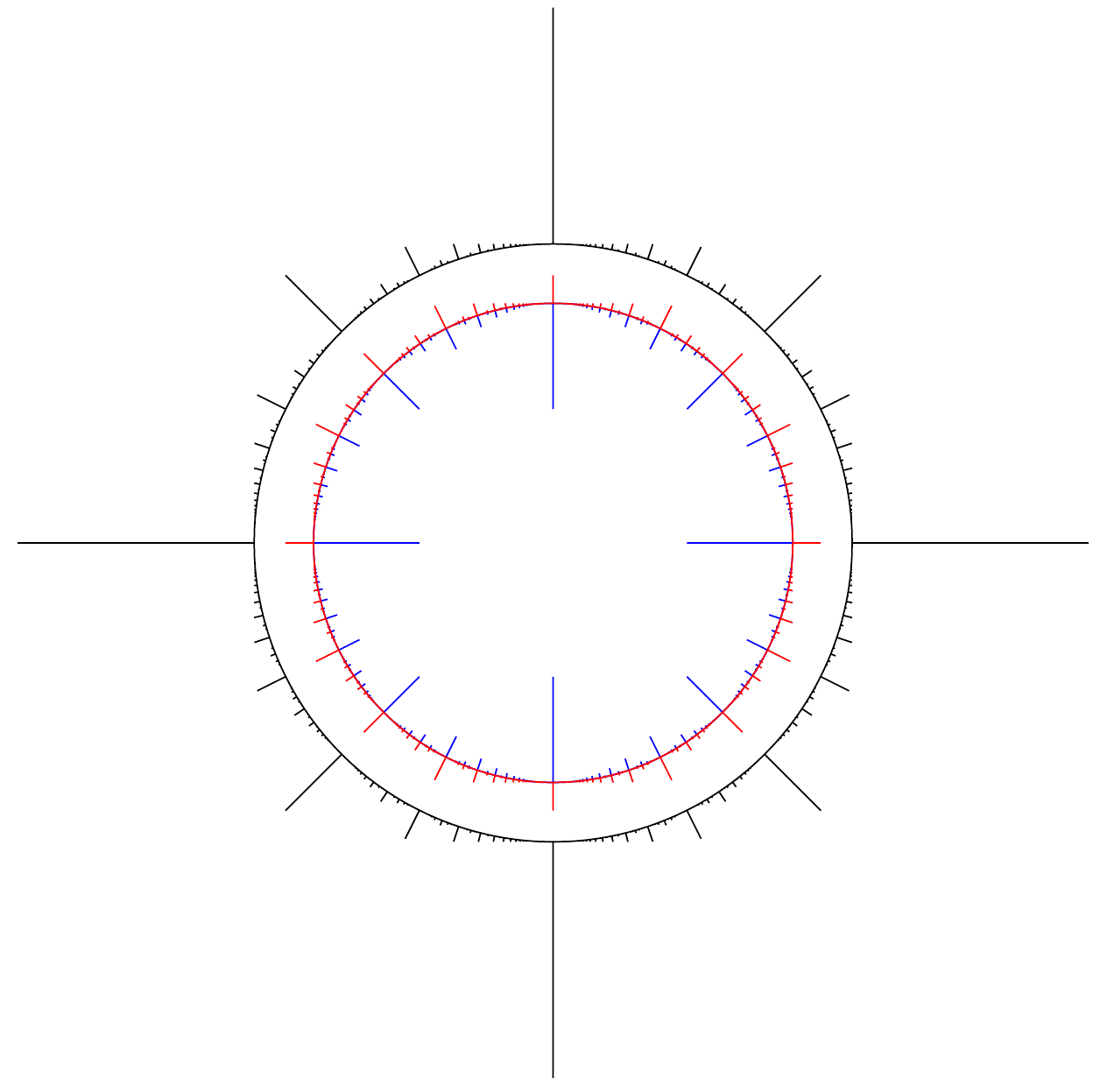}
  \end{tabular}
  \end{tabular}
  \caption{On the left is a schematic drawing of $Q_*$ and $Q^*$ as functions on $S^{1}$, in blue and black respectively.  On the right is additionally included the graph of $Q_{*,cont}$ in red.  Except for the radial symmetry at irrational directions, the picture represents the bounds proved in \tref{mainQ} and \tref{mainQcont}.}
  \label{f.intro2d}
\end{figure}

 It will follow easily from the definitions that $Q_{*,cont}$ and $Q^*_{cont}$ have the reversed upper/lower-semicontinuity properties from $Q_*$ and $Q^*$, and 
  \begin{equation}\label{e.Qconteq}
   \limsup_{e' \to e} Q_*(e') \leq Q_{*,cont}(e)  \ \hbox{ and } \  Q^*_{cont}(e) \leq \liminf_{e' \to e} Q^*(e').
   \end{equation}
  Our conjecture is that equality holds in \eref{Qconteq}, however we do not have evidence in either direction at the moment. Assuming equality holds, with minor nondegeneracy caveats, we could construct recovery sequences for arbitrary solutions of \eref{fb1} in $d=2$.  Although we do not prove the full conjecture, we make significant steps in that direction, in particular we prove that equality holds in \eref{Qconteq} at all irrational directions in $d=2$, and only fails by a small amount for rational directions with large modulus.  This is stated precisely below.
   
   Before stating our results we explain what role $I_{cont}$ plays.  In terms of $I_{cont}$ we specify the subclass of solutions to the pinning problem \eref{fb1} for which we can construct a recovery sequence.  We will call this new problem the augmented pinning problem.  Consider a convex setting, let $U \subset \R^d$ an open domain with $\R^d \setminus U$ convex and compact.  Say that $u$ is a solution of the augmented pinning problem if $\{u>0\}$ is convex and
  \begin{equation}\label{e.fbconvex}
  \begin{cases}
    \Delta u = 0 & \mbox{in } \{ u > 0 \} \cap U \\
    |\grad u| \in [Q_{*,cont}(n_x),Q^*(n_x)] & \mbox{on } \partial \{ u > 0 \} \cap U \\
    u = 1 & \hbox{on } \partial U.
  \end{cases}
  \end{equation}
Here the subsolution condition is upper semicontinuous and so needs to be interpreted carefully. The theory for this type of problem was developed in the previous paper of the author and Smart~\cite{FeldmanSmart}.  The augmented pinning problem can also be stated in the case when $\overline{U}$ is compact and convex.  Say that $u$ is a solution of the augmented pinning problem in this case if $\{u=0\}$ is convex and 
    \begin{equation}\label{e.fbconcave}
  \begin{cases}
    \Delta u = 0 & \mbox{in } \{ u > 0 \} \cap U \\
    |\grad u| \in [Q_{*}(n_x),Q^*_{cont}(n_x)] & \mbox{on } \partial \{ u > 0 \} \cap U \\
    u = 1 & \hbox{on } \partial U.
  \end{cases}
  \end{equation}
  Here the supersolution condition is the one which needs to be interpreted carefully since it is lower semicontinuous. The problems \eref{fbconvex} and \eref{fbconcave} are, in a sense, dual to each other.

 This paper only gives a notion of solution to the augmented pinning problem in these convex settings, it is not clear how a solution should be defined in the non-convex setting. The solution condition would seem to depend on the local convexity or concavity of the free boundary.  
 
 We do not currently have any example of a homogenization problem where equality fails in \eref{Qconteq}.  However, in Appendix~\ref{s.augmentedpinningex}, we give an example of a limit procedure approximating \eref{fb1} by other homogeneous problems of the form \eref{fb1} where the limit equation is indeed an augmented pinning problem of the form \eref{fbconvex}.

 Now we state our main result about $I_{cont}(p)$.

  \begin{theorem}\label{t.mainQcont}
 The following properties hold for the continuous part of the pinning interval.  See \sref{contpart} for the precise definitions of $Q^*_{cont}(e)$ and $Q_{*,cont}(e)$.
\begin{enumerate}[label=(\roman*)]
 \item Let $e \in S^{d-1}$ there exist 
 \[ \limsup_{e' \to e} Q_*(e') \leq Q_{*,cont}(e) \leq \langle Q^2 \rangle^{1/2} \leq  Q^*_{cont}(e) \leq \liminf_{e' \to e} Q^*(e')\]
  respectively upper and lower semicontinuous in $e$ such that the subsolution (supersolution) perturbed test function argument works for $\alpha > Q_{*,cont}(e)$ (resp. for $\alpha < Q^*_{cont}$), see \sref{contpart} for the precise definitions.  
 \item If $d=2$ then, for irrational directions $e \in S^{1} \setminus \R\Z^2$, $Q^*(e) = Q^*_{cont}(e)$ and $Q_*(e) = Q_{*,cont}(e)$.  For $e = \frac{\xi}{|\xi|}$ rational, with $\xi \in \Z^2 \setminus \{0\}$ irreducible,
 \[ Q_{*,cont}(e) \leq Q_*(e) + C|\xi|^{-1/2} \ \hbox{ and } \  Q^*_{cont}(e) \geq Q^*(\xi) - C|\xi|^{-1/2}\]
 for $C = C(\min Q,\max Q,\|\grad Q\|_\infty)$.
 \item\label{part.mainQcontdir} If $d=2$ then directional limits of $Q^*_{cont}$ and $Q_{*,cont}$ exist at rational directions $e \in S^{1}  \cap \R\Z^2$ and agree with the directional limits of $Q_*$ and $Q^*$.
\end{enumerate}
  \end{theorem}

  See \fref{intro2d} for a drawing representing $Q_*$, $Q^*$ and $Q_{*,cont}$.  The reader may notice that \tref{mainQcont} mirrors the first three parts of \tref{mainQ}, this is true at the level of the proofs as well.  Basically the same techniques are used to prove both results, as described above the key idea is the construction of approximate foliations by plane-like solutions. Then we sew together along the foliation to construct approximate sub/supersolutions near smooth sub/supersolutions $\varphi$ with small variation in the gradient.

Our main result is the existence of recovery sequences in the convex setting for solutions of the augmented pinning problem \eref{fbconvex}.

 \begin{theorem}\label{t.main}
 Suppose that $u$ solves \eref{fbconvex} in a domain $U \subset \R^d$, $\R^d \setminus U$ is convex and compact, and $\{u >0\}$ is convex. Then there exists a sequence of solutions $u^\ep$ of \eref{fb0} which converge uniformly to $u$ and the positivity sets converge in Hausdorff distance.  The sequence $u^\ep$ can be taken to be a local minimizers of the inhomogeneous Alt-Caffarelli energy \eref{energyep}.    
    \end{theorem}

 One key new idea in the proof of \tref{main} is that the construction of solutions to \eref{fb0}, or local minimizers of \eref{energyep}, can be reduced to the convergence of the minimal supersolution and maximal subsolution.  In effect this means that the construction of curved subsolutions and supersolutions can be localized using the perturbed test function method.  Such localized construction is exactly the content of \tref{mainQcont}.  In a previous paper the author and Smart \cite{FeldmanSmart} developed viscosity solution tools to prove the convergence of minimal supersolutions and maximal subsolutions.  We will use those tools again here, with some necessary refinements.  These ideas should also work without convexity.
 
 Note that the convergence of the minimal supersolutions in the convex setting is a corollary of the statement \tref{main}.  As described above, at the level of the proof, \tref{main} should really be seen as a corollary of the convergence of the minimal supersolutions (and maximal subsolutions).  The sequence of minimal supersolutions to \eref{fb0} were previously studied by \cite{CaffarelliLee}, they show subsequential convergence to a supersolution of \eref{fb1}.

\subsection{Literature and motivation}
One of the main physical motivations for our work is to explain the shapes of capillary drops on rough or patterned solid surfaces.  It has been observed in experimental literature that water droplets placed on micro-patterned surfaces with a lattice structure can appear to have polygonal shapes, see Raj et al. \cite{Raj-Adera-Enright-Wang}. A similar phenomenon appears in patterned porous media, see \cite{lenormand1990liquids,Kim-Zheng-Stone,Courbin}.  One is led to wonder whether these shapes are a microscale phenomenon, or a macroscale phenomenon that remains in the homogenization limit.  Starting in our previous work with Smart \cite{FeldmanSmart} we have been investigating this question.  In that paper we derived an equation like \eref{fb1} from a scaling limit for a discrete version of the Alt-Caffarelli functional.  From this perspective we argued that these facets appearing the physical experiments are indeed a macroscale phenomenon and they are caused by discontinuities of the pinning interval in the normal direction.  Then the shape of the large scale facets can be understood by studying the problem \eref{fb1} using viscosity solution techniques.  In this paper we are now able to derive at least some of the same results in the continuum.  The situation for the continuum problem is much more complicated, still many parts of the philosophy there have carried over here.

The closest results to the present paper are the works of Caffarelli and Lee~\cite{CaffarelliLee}, Caffarelli, Lee and Mellet~\cite{CaffarelliLeeMellet}, Kim~\cite{Kim}, and Kim and Mellet \cite{KimMellet}.  Caffarelli and Lee~\cite{CaffarelliLee} studied the same problem as us, they constructed plane-like solutions of the cell problem at the maximal slope.  They used this to show that any subsequential limit of the minimal supersolutions to \eref{fb0} is a supersolution of \eref{fb1}.  They also introduced, with some very beautiful arguments, the idea that facets in the free boundary are caused by discontinuities in $Q^*$.  Caffarelli, Lee and Mellet ~\cite{CaffarelliLeeMellet} studied a flame propagation problem which combines homogenization with a singular limit leading the the Alt-Caffarelli free boundary problem.  Among their results, they show existence of minimal slope plane-like solutions with Birkhoff monotonicity properties.  Kim~\cite{Kim} studied an evolution associated with \eref{fb0}, she showed the homogenization for that problem and the possibility of non-trivial pinning interval in laminar media.  The result of Kim, when specified to the case of stationary solutions, gives \tref{kcl} recalled above. Kim and Mellet~\cite{KimMellet} studied a $1$-$d$ evolutionary problem associated with \eref{fb0} on an inclined plane, they showed the existence of travelling wave, volume constrained solutions and explained the affects of pinning and de-pinning in that model.   We also mention a connection with the work of Po\v{z}\'{a}r~\cite{Pozar}, on the space-time periodic Hele-Shaw flow, where resonances cause pinning of the velocity at some directions.   In numerical experiments, see Po\v{z}\'{a}r and Palupi~\cite{PozarPalupi}, velocity pinning at a single direction also appears to cause creation of facets in the flow.

There have been several mathematical investigations of hysteresis phenomena in the capillarity model. The earliest we are aware of is Caffarelli and Mellet \cite{CaffarelliMellet2} which shows the existence of non-axially symmetric local minimizers in a slight generalization of the laminar setting.  DeSimone, Grunewald and Otto~\cite{DeSimoneGrunewaldOtto} have introduced a quasistatic rate-independent dissipative evolution to model the effects of hysteresis.  This was studied further by Alberti-DeSimone~\cite{AlbertiDeSimone}.  In that model the contact angle hysteresis is ``baked in" and rotation invariance is assumed for the pinning interval.  For us the lack of rotation invariance, and the presence of discontinuities in the pinning interval, is one of the key difficulties.  It would be very interesting to derive an energy-based quasistatic evolution of this type by homogenization of a microscopic model without hysteresis.

We also mention a connection with the boundary sandpile model introduce by Aleksanyan-Shahgholian~\cite{AleksanyanShahgholian1,AleksanyanShahgholian2}.  This was the original discrete model which motivated \cite{FeldmanSmart} and, as we showed in there, the scaling limit of the steady state for the boundary sandpile model is the minimal supersolution of a problem like \eref{fb1}.

We explain the relation between our results and the results in Caffarelli and Lee \cite{CaffarelliLee}. There is a small overlap where, in \sref{cell}, we reprove the existence of global plane-like solutions of \eref{cell0} at the maximal slope $Q^*(e)$. There are some minor technical changes in the proof.  This result is stated here as a subset of \tref{mainQ} part \ref{part.mainQexist}.  The other parts of \tref{mainQ} part \ref{part.mainQexist} are new, but still very much inspired by \cite{CaffarelliLee} and also by Caffarelli and de la Llave~\cite{CaffarellidelaLlave}.

\subsection{Acknowledgments} The author thanks Inwon Kim, Felix Otto and Charlie Smart for helpful conversations and suggestions which have helped to improve the exposition.


\section{Background}
We recall some basic properties of solutions to the free boundary problem
\begin{equation}\label{e.unitscale}
\left\{
\begin{array}{ll}
 \Delta u = 0 & \hbox{ in } \ \{u >0\} \cap U \vspace{1.5mm}\\
|\grad u| = Q(x) & \hbox{ on } \ \partial \{u >0\}  \cap U
\end{array}
\right.
\end{equation}
and/or minimizers/local minimizers/critical points of the Alt-Caffarelli energy
\begin{equation}\label{e.ACenergy}
 E(v,U) = \int_{U} |\grad v|^2 + Q(x)^2 1_{\{v >0\}} \ dx
 \end{equation}
over some domain $U \subset \R^d$.  Most of this section is review of results from the literature, however some additional arguments are needed in certain places.

\subsection{Notation}  We explain some notations and conventions which will be used in the paper.  We will say that a constant $C$ or $c$ is universal if depends at most on the dimension $d$, the upper and lower bounds $0 < \min Q \leq Q \leq  \max Q$, and the Lipschitz norm of $Q$.  These constants may change from line to line.  For $u, v \geq 0$ we say 
\[ u \prec v \ \hbox{ if } \ u \leq v \ \hbox{ and } \ u < v \ \hbox{ in } \ \overline{\{u>0\}}.\]
We say an open set $\Omega$ is inner/outer-regular if every boundary point has an interior/exterior ball touching at that point.  We say that $\Omega$ is $r$-inner/outer-regular if the touching balls have radius at least $r$.  For a continuous $u \geq 0$ we may say that $u$ is inner/outer-regular if $\{u>0\}$ is inner/outer regular.
\subsection{Viscosity solutions and comparison principle}\label{s.viscositysolutions} The equation \eref{unitscale} will be interpreted in the sense of viscosity solutions.  We will also work with local minimizers for \eref{ACenergy}, in that case we will typically need to establish that the minimizers we create are viscosity solutions.

Let $U$ a domain of $\R^d$.

\begin{definition}
  A supersolution of \eref{unitscale} is a non-negative function $u \in LSC(U)$ such that, whenever $\varphi \in C^\infty(\R^d)$ touches $u$ from below in $U$, there is a contact point $x$ such that either
  \begin{equation*}
    \Delta \varphi(x) \leq 0
  \end{equation*}
  or
  \begin{equation*}
    \varphi(x) = 0 \quad \mbox{and} \quad |\grad \varphi(x)| \leq Q(x).
  \end{equation*}
\end{definition}
\begin{definition}
  A subsolution of \eref{unitscale} is a non-negative function $u \in USC(U)$ such that, whenever $\varphi \in C^\infty(\R^d)$ touches $u$ from above in $\overline{\{u>0\}} \cap U$, there is a contact point $x \in \overline{\{u>0\}} \cap U$ such that either
  \begin{equation*}
    \Delta \varphi(x) \geq  0
  \end{equation*}
  or
  \begin{equation*}
    \varphi(x) = 0 \quad \mbox{and} \quad |\grad \varphi(x)| \geq Q(x).
  \end{equation*}
\end{definition}
\begin{definition}
We will say that $u$ is a strict supersolution (subsolution) of \eref{unitscale} if it is a supersolution (subsolution) of \eref{unitscale} for $\lambda Q(x)$ with some $\lambda <1$ ($\lambda >1$).  
\end{definition}

Typically we will want to work with super/subsolutions which are actually harmonic in their positivity set.  For this we can use the harmonic lift.
\begin{lemma}
Suppose that $U$ is outer-regular and $u$ is a super/subsolution of \eref{unitscale} and let $w$ be the minimal supersolution of
\[ \max\{\Delta w,u-w\} = 0 \ \hbox{ in } \ \{u>0\} \cap U \ \hbox{ with } \ w = 0 \ \hbox{ in } \ U \setminus \overline{\{u>0\}}.\]
Then $w$ is a super/subsolution of \eref{unitscale}.
\end{lemma}
\begin{proof}
The unusual definition of the harmonic lift is due to the possible irregularity of the set $\{u>0\} \cap U$ which is not even necessarily open in the subsolution case.  We check that case since it is slightly more interesting.  Suppose that $\varphi$ touches $w$ from above in $\overline{\{w>0\}} \cap U$ at some $x$.  First suppose $\varphi(x) >0$. Then either $w(x)=u(x)$, in which case the subsolution condition for $u$ applies, or $w(x)>u(x)$ in which case the subsolution condition for $w$ implies
\[ 0 \leq \max\{ \Delta \varphi(x), u(x) - w(x)\} =  \Delta\varphi(x). \]
If $\varphi(x) = 0$ then $0 = \varphi(x) = w(x) = u(x)$ and so the subsolution condition for $u$ applies.
\end{proof}

\begin{lemma}[Strict comparison]\label{l.strictcomparison}
Suppose that $u$ and $v$ are respectively a sub and supersolution of \eref{unitscale} in $U$, $u \leq v$ in $U$, and $u \prec v$ on $\partial U$.  Then $u$ cannot touch $v$ from below in $U$ at a regular free boundary point $x \in \partial \{u>0\} \cap \partial \{v>0\}$.
\end{lemma}
If $u$ is inner regular and $v$ is outer regular then any touching point would have to be a regular point.

There is a standard and convenient way to create inner-regular supersolutions / outer-regular subsolutions which is by $\inf$/$\sup$ convolution.  Given $u: U \to [0,\infty)$ and $\delta >0$ we define
\begin{equation}
 u^\delta(x) = \sup_{B_\delta(x)}u(y) \ \hbox{ and } \ u_\delta(x) = \inf_{B_\delta(x)} u(y).
 \end{equation}
 These are well defined in the domain 
 \[ U^\delta = \bigcup_{B_\delta(x) \subset U} B_\delta(x).\]
\begin{lemma}
Suppose that $u$ is a supersolution (resp. subsolution) of \eref{unitscale} in $U$ and $\delta>0$.  Then $u_\delta$ (resp. $u^\delta$) is a supersolution (resp. subsolution) of \eref{unitscale} in $U^\delta$ for 
\[ Q^\delta(x) = \sup_{B_\delta(x)} Q \quad \hbox{(resp. $Q_\delta(x) = \inf_{B_\delta(x)} Q$).}\]
  Furthermore $\{u_\delta>0\}$ is outer-regular with exterior balls of radius $\delta$ at every boundary point (resp. $\{u^\delta>0\}$ is inner-regular with interior balls of radius $\delta$).
\end{lemma}
The sup/inf convolutions actually have a stronger property called the $R$-subsolution (or $R$-supersolution) property.   See \cite[Chapter 2]{CaffarelliSalsa} for the proof.

We just state the $R$-subsolution property, the $R$-supersolution property is similar.  Say $v$ is an $R$-subsolution if the following hold.
\begin{enumerate}[label=(\roman*)]
\item $v$ is a viscosity subsolution of \eref{unitscale}
\item Whenever $x_0 \in \partial \{v>0\}$ has an interior touching ball then
\[ v(x) \geq Q(x_0)[(x - x_0) \cdot n]_+ + o(|x-x_0|)\]
where $n$ is the unit vector pointing from the center $x_0$ to the center of the touching ball.
\end{enumerate}
Note that the usual subsolution property requires the free boundary to be outer regular at a point to get the asymptotic expansion, for $R$-subsolutions the asymptotic expansion also holds at inner regular free boundary points.  For the sup convolution every free boundary point is inner regular. 

$R$-subsolutions and $R$-supersolutions satisfy a stronger comparison principle.  Again, see \cite[Chapter 2]{CaffarelliSalsa} for the proof.  

\begin{lemma}\label{l.Rcomparison}
Suppose that $u$ and $v$ are respectively an $R$-subsolution and a supersolution of \eref{unitscale} in $U$, $u \leq v$ in $U$, and $u \prec v$ on $\partial U$.  Then $u$ cannot touch $v$ from below in $U$ at an inner regular free boundary point $x \in \partial \{u>0\} \cap \partial \{v>0\}$.
\end{lemma}

The $R$-subsolution and $R$-supersolution condition and the corresponding comparison principle turn out to be rather useful for energy minimization arguments.  However, in any case we use them, they are really just a convenient rephrasing of the following trick:  If $u^\delta = \sup_{B_\delta(x)}u(y)$ touches $v$ from below at a free boundary point, then $u^{\delta/2}$ touches $v_{\delta/2}$ from below at a free boundary point.  In particular \lref{Rcomparison} and \lref{strictcomparison} are really the same when the $R$-supersolution or $R$-subsolution in question is an inf or sup convolution.

Finally we include a result on the asymptotic expansion for a positive harmonic function in a domain $\Omega$, vanishing on $\partial \Omega$, near one-sided regular boundary points.  This is copied from \cite[lemma 11.17]{CaffarelliSalsa}.
\begin{lemma}[Lemma 11.17 \cite{CaffarelliSalsa}]\label{l.asymptotics}
Let $u$ be a positive harmonic function in a domain $\Omega$.  Assume that $x_0 \in \partial \Omega$ and $u$ vanishes on $B_1(x_0) \cap \partial \Omega$.  Then the following hold.
\begin{enumerate}[label=(\roman*)]
\item If $x_0$ is an inner regular boundary point then either $u$ grows more than any linear function near $x_0$ or it has the asymptotic expansion
\[ u(x) \geq \alpha \alpha[(x-x_0)\cdot n]_++o(|x-x_0|)\]
for some $\alpha>0$, with $n$ the inward normal of the touching ball at $x_0$.  Equality holds in every nontangential region.
\item If $x_0$ is an outer regular boundary point then either $u$ grows slower any linear function near $x_0$ or it has the asymptotic expansion
\[ u(x) \leq \alpha [(x-x_0)\cdot n]_++o(|x-x_0|)\]
for some $\alpha>0$, with $n$ the outward normal of the touching ball at $x_0$.  Equality holds in every nontangential region and, in the case $\alpha>0$ actually $n$ is the normal direction to $\partial \Omega$ at $x_0$.
\end{enumerate}
\end{lemma}

\subsection{Minimal supersolutions / maximal subsolutions}  One important way of creating viscosity solutions of \eref{unitscale} is by Perron's method, finding the minimal supersolution or maximal subsolution above or, respectively, below a certain obstacle.

These properties can also be localized.  
\begin{definition}
Let $U \subset \R^d$ a domain.  We say that $u \in LSC(U)$ is a minimal supersolution in $U$ if it is a supersolution and, for any $D \subset U$ open and a supersolution $v \in LSC(\overline{D})$ with $v \geq u$ on $\partial D$, also $v \geq u$ in $D$.
\end{definition}
\begin{definition}
Let $U \subset \R^d$ a domain.  We say that $u \in USC(U)$ is a maximal subsolution in $U$ if it is a subsolution and, for any $D \subset U$ open and a subsolution $v \in USC(\overline{D})$ with $v \leq u$ on $\partial D$, also $v \leq u$ in $D$.
\end{definition}
It is standard to check that if $u$ is a minimal supersolution or maximal subsolution in $U$ then $u$ is a solution in $U$.  In particular, actually $u \in C(U)$. Moreover, as we will see in the next section, $u$ will satisfy a Lipschitz bound.
\begin{theorem}
Let $U$ be an outer regular domain and $g$ be a continuous function on $\partial U$.  Suppose $\overline{v}$ is an outer regular, continuous, $R$-supersolution in $\overline{U}$ with $g \prec \overline{v}$ on $\partial U$.  Then the function
\[ u(x) = \sup\{w: \ \hbox{ $w$ is a subsolution in $U$, $w \leq g$ on $\partial U$, and $w \leq \overline{v}$ in $\overline{U}$}\} \]
is a viscosity solution of the free boundary problem \eref{unitscale} in $U$ with $u \in C(\overline{U})$ and $u = g$ on $\partial U$.
\end{theorem}
The result for existence of minimal supersolutions is analogous and can be found in \cite[theorem 6.1]{CaffarelliSalsa}.  The Perron's method argument for the maximal subsolutions is similar but not exactly same as the supersolution case.

\subsection{Linear growth at the free boundary}  In this paper we will only use the most basic level of the local regularity theory for free boundary problems.  This is the Lipschitz bound and nondegeneracy at the zero level set.  Morally speaking the Lipschitz bound follows from the supersolution property $|\grad u| \leq \Lambda$ on the free boundary, while nondegeneracy follows from the sub-solution property $|\grad u| \geq \lambda$ on the free boundary.  

First the Lipschitz bound, see Caffarelli-Salsa~\cite[lemma 11.19]{CaffarelliSalsa} for the proof.
\begin{lemma}[Lipschitz continuity]\label{l.properties}
Suppose that $u \geq 0$ is a harmonic function in $\{u>0\} \cap B_1$. If $u$ solves $|\grad u| \leq \Lambda$ on $\partial \{u>0\} \cap B_1$, in the viscosity sense, then $u$ is Lipschitz continuous with constant $C(d)\Lambda$ in $B_{1/2}$.
\end{lemma}
Together with the harmonic lifts this allows to show that minimal supersolutions and maximal subsolutions of \eref{unitscale} are both Lipschitz with universal constant.

The nondegeneracy, it turns out, requires more information than just the subsolution property.  As far as we are aware, nondegeneracy is known to hold for minimal supersolutions, energy minimizers, a-priori outer-regular free boundaries, and, in $d=2$, for maximal subsolutions.
\begin{lemma}[Non-degeneracy]\label{l.nondegen}
Take one of the following assumptions:
\begin{enumerate}[label=(\roman*)]
\item\label{part.nondegen1} $u$ is a minimal supersolution in $B_{1}$.
\item\label{part.nondegen4} $d=2$ and $u$ is a maximal subsolution in $B_1$.
\item\label{part.nondegen2} $u$ is an energy minimizer in $B_1$ in the sense that, for any $v \in H^1(B_1)$ with $v \geq 0$ and $u-v$ compactly supported in $B_1$,
\[ E(u ,B_1) \leq E(v,B_1).\]
\item\label{part.nondegen3} $u$ solves $|\grad u| \geq \lambda$ on $\partial \{u>0\}$, in the viscosity sense and the positivity set $\{u>0\}$ has an exterior ball at $x \in \partial \{u>0\}$ with radius $1$.
\end{enumerate}
For any $x \in \partial \{u>0\} \cap B_{1/2}$, or the specific $x\in \partial \{u>0\}$ from \partref{nondegen3}, and $r \leq 1/2$
\[  \sup_{B_r(x)} u \geq c(d,\lambda) r.\]
In case \partref{nondegen1} and \partref{nondegen2}, for any $ x \in B_{1/2}$,
\[ u(x) \geq c(d,\lambda) d(x,\partial \{u>0\}).\]
\end{lemma}

Parts \partref{nondegen1} and \partref{nondegen2} can be found in Alt-Caffarelli~\cite{AltCaffarelli}, or the book Caffarelli-Salsa~\cite{CaffarelliSalsa}.  Part~\partref{nondegen3} is a straightforward barrier argument.  See Orcan-Ekmekci~\cite{Betul} for nondegeneracy of largest subsolution in $d=2$.  Since the nondegeneracy of the maximal subsolution is not always a given, we will say that a maximal subsolution $u$ is nondegenerate in a domain $U$ if the estimate of \lref{nondegen} holds with a universal constant for every $x \in \partial \{u>0\} \cap U$ and ball $B_r(x) \subset U$.

Note we can get nondegeneracy at another scale $r$ by applying \lref{nondegen} to 
\[ v(x)= \frac{u\left(rx\right)}{r}, \]
since the solution property / minimization property and the nondegeneracy estimate are invariant under this rescaling.

\subsection{Energy minimizers}\label{s.energysetup} In this section we discuss the existence of local minimizers for the Alt-Caffarelli energy \eref{ACenergy} $E(\cdot,U)$.  Here $U$ could be an outer regular domain of $\R^d$ such that $\partial U$ is compact, a half-space, or all of $\R^d$. It is natural to consider direct minimization of $E$ over subsets of $H^1(U)$ when $\partial U$ is compact.  

Note that the meaning of local minimizer or critical point needs to be made precise, the functional $E$ is not differentiable on the natural space $H^1(U)$ where it is defined.  We say that $u$ is a local energy minimizer for $E(\cdot,U)$ if there exists $\delta>0$ such that, for any $v \in H^1(U)$, $v \geq 0$, with
\[ \sup_U|v-u| \leq \delta \ \hbox{ and } \ d_H(\overline{\{v>0\}} \cap U,\overline{\{u>0\}} \cap U) \leq \delta \]
it holds
\[ E(u,U) \leq E(v,U).\]
This is a slightly different notion of local minimizer than the one appearing in Alt-Caffarelli~\cite{AltCaffarelli}.

We say that $u$ is an absolute minimizer if for any precompact subdomain $D\subset U$ and any $v \in H^1(D)$, $v\geq 0$ and $v-u$ compactly supported in $D$
\[ E(u,D) \leq E(v,D).\]
The concept of absolute minimizer replaces the notion of global energy minimizer when the total energy is not finite.

In order to find local minimizers we will often look at admissibility conditions of the following type.  Suppose that $g : \overline{U} \to \R$ is Lipschitz continuous and $\underline{v}\prec \overline{v}$ Lipschitz continuous functions in $\overline{U}$ with $\underline{v} \prec g \prec \overline{v}$ on $\partial U$.  Consider the class
\[ \mathcal{A} = \{ v \in H^1_g(U): \underline{v} \leq v \leq \overline{v}  \} \]
where $H^1_g(U) = \{ v \in H^1(U): v - g\in H^1_0(U)\}$. Existence of a global minimizer of $E(\cdot,U)$ in the class $\mathcal{A}$ is straightforward by the direct method, the issue is that the constrained minimizer may touch one of the barriers in its positivity set and, therefore, not be a local minimizer.
\begin{lemma}\label{l.localmin}
Suppose that $\underline{v}\prec\overline{v}$ in $\overline{U}$ and $\underline{v}$ and $\overline{v}$ are, respectively, a nondegenerate, inner regular, $R$-subsolution and an outer regular $R$-supersolution of \eref{unitscale} in $U$.  Then there exist minimizers for $E(\cdot,U)$ on $\mathcal{A}$, and any such minimizer $u$ is a viscosity solution of \eref{unitscale} and satisfies
\[\underline{v} \prec u \prec \overline{v} \ \hbox{ in } \ \overline{U}. \]
\end{lemma}
Note that $u$ constructed in \lref{localmin} is a local minimizer in the previous sense.  The proof is following Alt-Caffarelli~\cite{AltCaffarelli} and Caffarelli~\cite[theorem 4]{CaffarelliFlatLipschitz}, which do not deal with constrained minimization of the type we consider so we need some additional arguments.
\begin{proof}

The existence of a minimizer $u$ for $E(\cdot,U)$ over $\mathcal{A}$ is standard by the direct method.
  
We check the Lipschitz continuity and nondegeneracy of $u$, we just sketch the proofs which are from \cite{AltCaffarelli} to point out where the obstacles come in.  The key point for the Lipschitz estimate is the following: there is a universal constant $C$ such that for any $B_r(x) \subset U$
\[  \hbox{ if } \ \frac{1}{r}\fint_{\partial B_r(x)} u \geq C \ \hbox{ then } \ u > 0 \ \hbox{ in } \ B_r(x). \]
The proof of the estimate requires perturbing $u$ by replacing with the harmonic lift $\tilde{u}$ in $B_r(x)$. By maximum principle this replacement preserves the ordering $\tilde{u} \leq \overline{v}$ as long as $B_r(x) \subset \{\overline{v}>0\}$.  This is \cite[lemma 3.2]{AltCaffarelli}, to prove the Lipschitz estimate, as in \cite[corollary 3.3]{AltCaffarelli} the estimate only needs to be applied in balls $B_{r+\ep}(x)$ with $B_r(x) \subset \{u>0\}$ and $\ep>0$ sufficiently small.  Thus there is a potential issue only when $B_r(x)$ touches $\{u>0\}$ from the inside at a point of $\partial \{u>0\} \cap \partial \{\overline{v}>0\}$.  However in this case we use the Lipschitz estimate of $\overline{v}$, following from the supersolution property \lref{properties}, and $x \in \partial \{\overline{v}>0\}$ so that $\overline{v}(x) \leq Cr$ in $B_r(x)$ and
\[ |\grad u(x)| \leq \frac{1}{r} \fint_{\partial B_r(x)} u \leq \frac{1}{r} \fint_{\partial B_r(x)} \overline{v} \leq C. \]
Next we check the nondegeneracy of $u$, the argument has a similar flavor.  Let $x \in \partial \{u>0\}$ and $r>0$, if $B_{r/2}(x) \cap \partial \{\underline{v}>0\}$ is nonempty then use the nondegeneracy of $\underline{v}$, otherwise $B_{r/2}(x) \subset \{\underline{v} = 0\}$ and so arbitrary downward perturbations of $u$ (of course preserving nonnegativity) are allowed and the nondegeneracy argument of \cite[lemma 3.4]{AltCaffarelli} applies.

We show that $u$ cannot touch $\underline{v}$ from above in $\overline{\{\underline{v}>0\}}$, and $\overline{v}$ cannot touch $u$ from above in $\overline{\{u>0\}}$. Then the argument of \cite[theorem 4]{CaffarelliFlatLipschitz} carries over and the viscosity solution condition holds for $u$.    

  Suppose $x_0 \in \partial \{u>0\} \cap \partial \{\overline{v}>0\}$, the other case is similar.  Since $\{\overline{v}>0\}$ is outer regular, $x_0$ is an outer regular point for $\{u>0\}$, thus $u$ has the asymptotic expansion at $x_0$, for some $n \in S^{d-1}$ and $\alpha>0$,
\[ u(x) \leq \alpha [(x-x_0) \cdot n]_+ +o(|x-x_0|) \ \hbox{ as } \ x \to x_0\]
with equality in any non-tangential region, this is by \lref{asymptotics}.  Note that $\alpha>0$ because of the nondegeneracy, in the case $u$ touches $\overline{v}$ from below we would instead use the Lipschitz estimate to get the linear blow-up.  In the case of $u$ touching $\underline{v}$ from above we use the assumption that $\underline{v}$ is nondegenerate.  Since $\overline{v}$ is an $R$-supersolution it also has an asymptotic expansion, by \lref{asymptotics}, at $x_0$
\[\overline{v}(x) \leq \beta [(x-x_0) \cdot n]_+ +o(|x-x_0|)  \ \hbox{ with } \  \beta < Q(x_0)\]
by the ordering $u \leq \overline{v}$, $\alpha \leq \beta$, and so $\alpha < Q(x_0)$.

  The proof of \cite[lemma 5.4]{AltCaffarelli} applies just as well in our constrained setting, to show that the blow-up $u_{0}(x) =  \alpha [(x-x_0) \cdot n]_+$ is a one-sided minimizer of $E_{0}$ on $\R^d$ in the sense that for any $\varphi \in H^1_0(B)$ for some ball $B \subset \R^d$ with $\varphi \leq 0$
\[ E_0(u_0,B) \leq E_0(u_0 + \varphi,B) \]
where
\[ E_0(v,U) = \int_{U} |\grad v|^2 + Q(x_0)^2 1_{\{v >0\}} \ dx.\]
We claim this is inconsistent with $\alpha < Q(x_0)$.  The ``correct" proof is by comparing the energy per unit length (of the free boundary) of linear solutions, but we take a shortcut using the known results on uncontrained global minimizers.  Let $v_B$ be a global minimizer in $B$ with data $u_0$ on $\partial B$.  Without loss we can assume $v_B \geq u_0$ in $B$, otherwise $v_B \wedge u_0$ is a valid perturbation of $u_0$ and so $E_0(v_B \wedge u_0,B) \geq E(u_0,B)$. On the other hand
\[ E(v_B \wedge u_0,B) + E(v_B \vee u_0,B) = E(u_0,B) + E(v_B,B),\]
and so $E(v_B \vee u_0,B) \leq E(v_B,B)$ i.e. $v_B \vee u_0$ is also a global minimizer with the same boundary data.  Now, since $v_B \geq u_0$, we translate $u_0$ in the $-n$ direction until it touches $u_0$ from above at some free boundary point, then the subsolution condition for $v_B$ from \cite[theorem 4]{CaffarelliFlatLipschitz} implies
\[ Q(x_0) \leq \alpha.\]
This is a contradiction.

\end{proof}


\section{Plane-like solutions and the pinning interval}\label{s.cell}
The effective stable slopes are determined by a cell problem.  We would like to identify for which values of $p \in \R^d \setminus \{0\}$ there exists a solution of the free boundary problem behaving, at large scales, like $(p \cdot x)_+$.  More precisely we would like to find a global solution of
\begin{equation}\label{e.cell}
\left\{
\begin{array}{ll}
 \Delta u = 0 & \hbox{ in } \ \{u >0\} \vspace{1.5mm}\\
|\grad u| = Q(x) & \hbox{ on } \ \partial \{u >0\} 
\end{array}
\right.
\end{equation}
with, for some universal $C>1$,
\begin{equation}\label{e.cellb}
(p \cdot x-C)_+ \prec u(x) \prec (p \cdot x+C)_+.
\end{equation}
We call such a solution a plane-like solution.

The main goal of this section and the next is the following result on the existence of solutions of \eref{cell}.

\begin{theorem}\label{t.cell}
Let $e \in S^{d-1}$ there exist $Q_*(e) \leq \langle Q^2 \rangle^{1/2} \leq  Q^*(e)$ such that, there exists a global solution of \eref{cell} with slope $p = \alpha e$ if and only if $\alpha \in [Q_*(e),Q^*(e)]$.  Furthermore this solution can be chosen to be a local minimizer of the energy.
\end{theorem}

The parts of \tref{cell} having to do with energy minimization are handled below in \sref{energy}.

The construction of a maximal slope solution to \eref{cell} was carried out by Caffarelli-Lee~\cite{CaffarelliLee}.  We construct solutions from scratch here, taking a slightly different approach.  The slight differences in our argument are not the main point of the presentation.  We will need to see the intermediate stages of the construction in order to study the qualitative properties of $Q^*,Q_*$ in more detail, which is one of the aims of this article.  

The main issue is in proving that the free boundary for solutions of an approximate corrector problem stays within a band of finite width.  Here we follow the idea of \cite{CaffarelliLee}, which itself followed an idea of Caffarelli and de la Llave~\cite{CaffarellidelaLlave}.

We start with an approximate corrector problem, for $t>1$ define $u_t$ to be the minimal supersolution of
\begin{equation}\label{e.appcell}
\left\{
\begin{array}{ll}
\Delta u_t = 0 & \hbox{ in } \ \{u_t >0\} \cap \{ x\cdot p <0\} \vspace{1.5mm}\\
|\grad u_t| = Q(x) & \hbox{ on } \ \partial \{u_t >0\} \cap \{ x\cdot p <0\} \vspace{1.5mm}\\
u_t(x) = t & \hbox{ on } \ x \cdot p = 0 \ \
\end{array}
\right.
\end{equation}
We define the minimal distance from $x \cdot p = 0$ to the free boundary of $u_t$ and the approximate slope of $u_t$
\[ r(t) = \inf_{x \in \partial \{u_t >0\}} |x \cdot p| \ \hbox{ and } \  \alpha(t) = t/r(t).\]
We first show a universal bound on the oscillation of the free boundary in the direction $p$.  Then we use this to show that the quantity $r(t)$ is approximately subadditive, this allows us to conclude that the sequence of slopes $\alpha(t)$ has a limit as $t \to \infty$.  Symmetrical arguments will work for the maximal subsolution of \eref{appcell}, in that case we would find an approximately superadditive quantity.  There is a small subtlety here due to the different nondegeneracy results (\lref{nondegen}) available in the minimal supersolution / maximal subsolution case, see below for more details.

First we establish the so-called Birkhoff property which takes advantage of the periodicity and the minimal super-solution / maximal subsolution property to get monotonicity with respect to lattice translations.  The Birkhoff monotonicity property in direction $p$, for a function $v$ on $\R^d$, is
\begin{equation}\label{e.monotonicity}
 v(x + k) \leq v(x) \ \hbox{ if $k \in \Z^d$ with } \ k \cdot p \leq 0.
 \end{equation}
 Although $u_t$ is not defined on $\R^d$ we can extend to $\R^d$ by defining $u_t(x) = t$ for $x \cdot p \geq 0$.
\begin{lemma}\label{l.monotonicity}
The solution $u_t$ of \eref{appcell} satisfies the Birkhoff property \eref{monotonicity}.
\end{lemma}
\begin{proof}
 Note that $u_t(\cdot + k)$ solves \eref{fb0} in $p \cdot x <0$ with boundary data
\[ u_t(x + k) \leq t \ \hbox{ on } \ x \cdot p = 0 \ \hbox{ since } \ (x+k) \cdot p \leq 0.\]
Since the minimal supersolution property is preserved under restriction of the domain, and $u_t(x) \geq u_t(x+k)$ on $p \cdot x =0$, $u_t(x) \geq u_t(x-k)$ on $p \cdot x <0$.
\end{proof}

Now using also the nondegeneracy, Lipschitz estimates, and periodicity we get an oscillation bound on the free boundary for both minimal supersolutions and maximal subsolutions.  Note that the known nondegeneracy properties are a bit different for minimal supersolutions and maximal subsolutions, we will only use the nondegeneracy at outer regular free boundary points, \lref{nondegen} part \partref{nondegen3}, which only uses the viscosity solution property $|\grad u| \geq \lambda$ on the free boundary.
\begin{lemma}\label{l.width}
There is a universal constant $C$ such that
\[ \osc_{x\in \partial \{u_t >0\}} x \cdot p \leq C.\]
\end{lemma}
\begin{proof}
 Let $x_0 \in \partial \{ u_t >0\}$ with $p \cdot x_0 \leq 1+\inf_{x \in \partial \{ u_t>0\}} x \cdot p$.  For any $r>1$ slide the ball $B_r(x_0 + t p)$ in from $t=-\infty$ until it touches $\{u_t>0\}$ from the outside.  The touching point occurs at some $x_1$ with $x_1 \cdot p \leq x_0 \cdot p \leq 1+\inf_{x \in \partial \{ u_t>0\}} x \cdot p$.  By the nondegeneracy \lref{nondegen} part \partref{nondegen3} and Lipschitz estimate \lref{properties}, 
\[ cr \leq \max_{|x-x_1| \leq r/2} u_t =  u_t(y_0) \leq Cd(y_0,\partial \{u_t>0\})\]
so that
\[ B_{c_0r}(y_0) \subset \{ u_t >0\} \cap B_{r}(x_1).\]
Choose $r = c_0^{-1} \sqrt{d}$ so that $c_0 r = \sqrt{d} = \textup{diam}([0,1]^d)$. Now for any $k \in \Z^d$ with $k \cdot p \geq 0$
\[ B_{\sqrt{d}}(y_0) + k \subset \{u_t >0\} + k \subseteq \{u_t >0\}.\]
Now let $x$ with $0 \geq x \cdot p \geq y_0 \cdot p$, $x$ is in some unit cell $\Box = k + y_0 + [0,1]^d$ of the lattice $\Z^d+y_0$.  By convexity of $\Box$ one of the extreme points $(\Z^d+y_0) \cap \Box$ must also lie in $x \cdot p \geq y_0 \cdot p$. Call this point $y_0 + k$ satisfying $(y_0 + k) \cdot p \geq y_0 \cdot p$, i.e. $k \cdot p \geq 0$.  Then
\[ x \in \Box \subset B_{\sqrt{d}}(y_0) + k \subset \{u_t >0\}.\]
Thus 
\[ \{ 0 \geq  x \cdot p \geq \inf_{x \in \partial \{ u_t>0\}} x \cdot p+C\} \subset \{ u_t >0\}.\]
\end{proof}

\begin{lemma}\label{l.normalbd}
For $t>0$ sufficiently large universal and $x$ with $x \cdot p = 0$
\[|\grad u_t(x) - \alpha(t) p|\leq  C/t. \]
\end{lemma}
\begin{proof}
From the Lipschitz bound \lref{width} $|\grad u_t| \leq C$.  Extend $u$ by odd reflection about $x \cdot p = 0$ by 
\[ u_t(x) = 2t-u_t(x-2(x \cdot p )p) \ \hbox{ for } x \cdot p <0.\]
From the bound on the width of $\partial \{ u_t>0\}$, \lref{width}, and using maximum principle,
\[ (\alpha(t)x\cdot p + t) \wedge 2t \leq u_t(x) \leq (\alpha(t) x \cdot p +t+C) \vee  0. \]
Now, for any $x \cdot p = 0$, $\grad u_t - \alpha(t)p$ is harmonic in $B_{ct}(x)$ and
\[ \left|\int_{B_{ct}(x)} (\grad u_t - \alpha(t) p) \ dy\right| = \left|\int_{\partial B_{ct}(x)} (u_t - \alpha(t) y \cdot p - t) n(y) \ dS_y\right| \leq Ct^{d-1}.\]
Then by the mean value theorem
\[ |\grad u_t(x) - \alpha(t) p| \leq C/t.\]
\end{proof}

\begin{lemma}\label{l.subadd}
The distance function $r(t)$ is approximately subadditive
\[r(t+s) \leq r(t) + r(s) + C \]
and therefore the limit exists
\[ Q^*(e) = \lim_{t \to \infty} \frac{t}{r(t)}.\]
\end{lemma}
\begin{proof}
We create a supersolution for the problem with data $t+s$.  Call $\overline{\alpha} = \alpha(t) - C_0/t$, for universal $C_0$ as in the statement of \lref{normalbd}. Call $a = -s/\overline{\alpha}$ and define
\[ 
v(x) =\left\{ \begin{array}{ll}
t+s + \overline{\alpha} x \cdot p & \hbox{for } \  0 \geq x\cdot p \geq a \\
u_{t,a}(x) & \hbox{for } \ x\cdot p \leq a.
\end{array}\right.
\]
To see that $v$ is a supersolution of the free boundary problem we just need to check that the interior supersolution condition holds on $x \cdot p = a$ which amounts to requiring the correct ordering of the normal derivatives of the piecewise components,
\[ p \cdot [ \grad u_{t,a} - (\alpha(t)-C_0/t)p] \geq 0, \]
which indeed holds by \lref{normalbd}.  Now since $u_t$ is the minimal supersolution, $u_t \leq v$ and therefore,
\[ r(t+s) \leq r(t) + C + |a| = r(t) + s/(\alpha(t)-C_0/t) + C \leq (1+\frac{s}{t})r(t) + C(1+\frac{s}{t}).  \]
Switching the roles of $t,s$ we find,
\begin{align*} 
r(t+s) & \leq \min\{ (1+\frac{s}{t})(r(t) + C), (1+\frac{t}{s})(r(s)+C)\} \\
& \leq r(t) + r(s) + C
\end{align*}
where in the last step we used $\min\{a,b\} \leq \lambda a + (1-\lambda)b$ in this case with $\lambda = \frac{t}{t+s}$.  

To complete the proof we just note that the approximate sub-additivity we proved is enough to carry out the usual argument for the convergence of subadditive sequences.
\end{proof}

\begin{lemma}\label{l.minimalsuper}
For any $\alpha \in [Q_*,Q^*]$ there exists a solution $v$ of \eref{cell}.  The solution can be chosen to have the Birkhoff monotonicity property \eref{monotonicity}. If the maximal subsolution $u_*$ constructed above has the nondegeneracy property of \lref{nondegen} at every free boundary point, then $v$ can be chosen to have it as well. 
\end{lemma}
\begin{proof}
First we construct a solution with the maximal slope $Q^*$, the construction for $Q_*$ is symmetric.  Take $u_t$ as above the minimal supersolution of \eref{appcell} and take an arbitrary, but fixed for each $t$,
\[k(t) \in \Z^d \cap \{ -r(t)e \geq x\cdot e \geq -(r(t) +\sqrt{d}/2)e\}.\]
Then define
\[ v_t(x) = u_t(x +k(t)).\]
The $v_t$ satisfy the bounds, by maximum principle as in \lref{normalbd},
\begin{equation}\label{e.planecomparison}
 (\alpha(t)(x +k(t))\cdot e + t)_+ \leq v_t(x) \leq (\alpha(t)(x+k(t))\cdot e + t+C)_+.
 \end{equation}
Now 
\[  \alpha(t)k(t) \cdot e \leq - r(t) \alpha(t) =- t  \]
and
\[ \alpha(t)k(t) \cdot e \geq -r(t)\alpha(t) - \alpha(t)\sqrt{d}/2 \geq -t - C.\]
Plugging these estimates into \eref{planecomparison} we find
\[ (\alpha(t)x \cdot e - C)_+ \leq v_t(x) \leq (\alpha(t)x \cdot e + C)_+\]
Now from \lref{subadd} $\alpha(t)$ converge to some $Q^*$ as $t \to \infty$, and the $v_t$ are uniformly Lipschitz continuous, \lref{properties}, and so we can extract a subsequential locally uniform limit $v^*$ with
\[ (Q^*x \cdot e - C)_+ \leq v^*(x) \leq (Q^*x \cdot e + C)_+.\]
Since the viscosity solution property is preserved under locally uniform limits $u$ solves the free boundary problem \eref{fb0} and combining with the above bound we see that $v^*$ solves the global corrector problem \eref{cell}.  The monotonicity property \eref{monotonicity} holds for the $v_t$ by \lref{monotonicity} and therefore it also holds in the limit for $v^*$.

Now we construct correctors for slopes $\alpha \in (Q_*,Q^*)$.  Consider the minimal and maximal slope solutions of \eref{cell}, $v_*$ and $v^*$ constructed above.  By making an appropriate $\Z^d$ translation of $v_*$ we can retain all the properties of \eref{cell} and also have
\[ v_*(x)  \prec Q_* (x\cdot e)_+ \leq Q^*(x\cdot e)_+ \prec  v^*(x) \ \hbox{ in } \ \R^d.\]
  Now consider
\[u_*(x) = \frac{\alpha}{Q_*}v_*(x) \ \hbox{ and } \ u^*(x)=\frac{\alpha}{Q^*}v^*(x). \]
By assumption $\frac{\alpha}{Q_*} > 1$ and $\frac{\alpha}{Q^*} < 1$ and therefore $u_*$ and $u^*$ are respectively sub and supersolutions of \eref{fb0}, still satisfying $u_* \prec u^*$ and now with
\begin{equation}\label{e.alphabds}
 (\alpha x \cdot e - C)_+ \leq u(x) \leq (\alpha x \cdot e + C)_+ \ \hbox{ for } \ u \in \{u_*,u^*\}.
 \end{equation}
Thus by Perrons method there is a solution to \eref{fb0} $u_* \leq v \leq u^*$ which, satisfying the above bounds, is a solution to \eref{cell}. 

We need to be a bit more precise about the construction to get the monotonicity \eref{monotonicity} and nondegeneracy properties. Fix data $v_t(x) = \alpha t$ on $x \cdot e = t$, by the above set up $u_*(x) \leq \alpha t \leq u^*(x)$ on $x \cdot e = t$.  Now find the minimal supersolution $v_t$ between $u_*$ and $u^*$ on $\{x \cdot e <t\}$ with the given Dirichlet data.  The Birkhoff property, \lref{monotonicity}, holds for $v_t$ by almost the same proof as before, now using also \lref{monotonicity} applied to $v_*$.  

Now for nondegeneracy, let $x \in \partial \{ v_t >0\}$ and $r>0$.  Suppose that $B_{r/2}(x) \subset \{u_* = 0\}$, then the usual nondegeneracy proof for minimal supersolutions carries over.  Suppose otherwise, then $y \in B_{r/2}(x) \cap \partial \{ u_* >0\}$ and by the nondegeneracy estimate of $u_*$, 
\[\sup_{B_r(x)} v_t \geq \sup_{B_{r/2}(y)}u_* \geq cr.\]
Finally we send $t\to \infty$ and extract a subsequential locally uniform limit $v$.  Then $v$ solves the equation, has the bounds \eref{alphabds}, the Birkhoff property is preserved in the limit and so is the nondegeneracy.

\end{proof}

We make a useful note about periodic plane-like solutions, as exist in the case of rational slope $\xi \in \Z^d \setminus \{0\}$.  Not only do these solutions stay within bounded distances of a plane, but actually, away from the free boundary, they converge with exponential rate to a particular linear function with the appropriate slope.
\begin{lemma}\label{l.bdrylayer}
Let $\xi \in \Z^d \setminus \{0\}$ irreducible and let $v$ be a solution of \eref{cell} with slope $\alpha \hat\xi$ which is $\xi^\perp$-periodic and
\[ \sup_x |v(x) - \alpha ( x \cdot \hat\xi)_+| \leq C\]
then there exists $s \in \R$ such that
\[ \sup_{x \cdot \hat \xi \geq t}  |v(x) -  (\alpha x \cdot \hat\xi+s)_+| \leq C\exp(-Ct/|\xi|).\]
\end{lemma}
\begin{proof}
The function $v(x) - \alpha (x \cdot \hat\xi)_+$ is bounded, $\xi^\perp$-periodic and harmonic in the half space $x \cdot \hat\xi \geq C_0$ for an appropriate $C_0$.  Then it is a classical result that there is boundary layer limit with exponential rate of convergence, see \cite{FeldmanKimBDRY} for a complete proof.  Basically the idea is to use the Harnack inequality oscillation decay at distance $C|\xi|$ from the half-space boundary $x \cdot \hat\xi = C_0$ to get the oscillation to decay by a factor of $1/2$ on the entire plane $x \cdot \hat \xi = C_0 + C|\xi|$ (using periodicity).  Then iterating one gets the exponential decay of oscillations.
\end{proof}

Finally we establish an alternative characterization of the pinning interval endpoints which is well suited to checking the viscosity solution condition in the homogenization limit.
\begin{lemma}\label{l.visccorr}
The upper and lower endpoints of the pinning interval are characterized by:
\begin{enumerate}
\item $Q^*(e)$ is the supremum over all $\alpha >0$ such that there exists a global supersolution $u$ of \eref{cell} with 
\[ u \geq \alpha (e \cdot x)_+ \ \hbox{ and } \ \inf_{B_C(0)} u = 0.\]
\item $Q_*(e)$ is the infimum over all $\alpha >0$ such that there exists a global supersolution $u$ of \eref{cell} with 
\[ u \leq \alpha (e \cdot x)_+ \ \hbox{ and } \ \sup_{B_C(0)} u >0.\]
\end{enumerate}
\end{lemma}
\begin{proof}
We just do the characterization of $Q^*$.  From the above construction, \lref{minimalsuper}, for $\alpha \leq Q^*(e)$ there exists such a global supersolution.  Take an appropriate lattice translation of $\frac{\alpha}{Q^*} u^*$.  

If there was such a supersolution $v$ existing for some $\alpha >Q^*(e)$. Translate to $v_t(x) = v(x +\frac{1}{\alpha}te)$ so that $v_t(x) \geq t$ on $x \cdot e = 0$ and therefore $v_t \geq u_t$ the minimal supersolution of \eref{appcell}.  Then since $\inf_{B_C(0)} v = 0$,
\[ 0 = \inf_{B_{C}(-\frac{1}{\alpha}te)} v_t \geq \inf_{B_{C}(-\frac{1}{\alpha}te)} u_t \ \hbox{ implies } \ r(t) \leq t/\alpha - C.\]
Sending $t \to \infty$ we get $\liminf \frac{t}{r(t)} \geq \alpha > Q^*(e)$ which contradicts the definition of $Q^*$, \lref{subadd}.
\end{proof}

\section{Energy minimizers}\label{s.energy}

In this section we group several results related to energy minimization.  The first goal is to complete the proof of \tref{cell}.  The last part of \tref{cell} is to show that the slope $\langle Q^2 \rangle ^{1/2}$ achieved by the global energy minimizers is in the pinning interval and to construct locally minimizing plane-like solutions with slope $\alpha p \in (Q_*(p),Q^*(p)) \cup \langle Q^2 \rangle ^{1/2}$.  The ideas are quite similar to the work of Caffarelli-de~la~Llave~\cite{CaffarellidelaLlave}, and the proof is basically a rehash of \sref{cell} using energy minimization to find solutions instead of Perron's method.

We will use the same ideas to construct energy minimizers near curved surfaces.  The techniques are similar to those we will use for the cell problem, the usefulness will come later when we begin to discuss the continuous part of the pinning interval.

\subsection{Local and global energy minimizing plane-like solutions} Here we finish the proof of \tref{cell}.
 
We also need to discuss the meaning of local minimizer for states on $\R^d$.

\begin{proposition}\label{p.mean}
For all $p \in S^{d-1}$ 
\[ \langle Q^2 \rangle^{1/2} \in [Q_*(p),Q^*(p)]. \]
Furthermore for all $\alpha \in (Q_*(p),Q^*(p)) \cup \langle Q^2 \rangle^{1/2}$ there exists a global plane-like solution of \eref{cell} which is a local minimizer (absolute minimizer if $\alpha = \langle Q^2 \rangle^{1/2}$) and satisfies the Birkhoff property \eref{monotonicity}.
\end{proposition}
\begin{remark}
In general it is not clear to us whether one can construct local minimizers with minimal/maximal slope $Q_*(p),Q^*(p)$. In the $d=1$ case it is not possible, a straightforward calculation checks that plane-like solutions with the minimal/maximal slope are not local minimizers when $Q''$ is not zero at its min/max.  The situation is degenerate. In the $d=2$ laminar case these $1$-$d$ perturbations that violate the local minimization property are not compactly supported and therefore are not valid perturbations, it is possible the situation is better in higher dimensions.
\end{remark}
\begin{proof}
The heuristic idea is that the global energy minimizer solves the free boundary problem and for this solution the optimal configuration results in an approximate slope $\langle Q^2 \rangle^{1/2}$.  First we construct an, appropriately defined, energy minimizing solution of the approximate corrector problem \eref{appcell}. Then we show that the free boundary for the minimal energy minimizing solution satisfies the same oscillation bound as for minimal supersolutions / maximal subsolutions.  The proof of the oscillation bound relies on uniqueness, previously this came from the minimal supersolution or maximal subsolution property.  In this case we will take the smallest energy minimizer, which will have a similar uniqueness property.  Once we have proven that the free boundary is flat we can compute the energy explicitly as a function of the slope and minimize.

 We assume that $p = \hat\xi$ for a lattice direction $\xi \in \Z^d \setminus \{0\}$.  We will show the existence of a global plane-like solution $u$ satisfying the Birkhoff property \lref{monotonicity} with slope $\langle Q^2 \rangle^{1/2}p$.  This solution will also be an absolute energy minimizer in the sense that
 \[ E(u ,B) \leq E(v,B) \]
   for ball $B$ and any $v\geq 0$ in $H^1_{loc}(\R^d)$ such that $u-v$ is compactly supported in $B$. Then the existence of such a solution at irrational directions follows by taking limits.

1.  Consider minimizing the Alt-Caffarelli functional $E$ on an open domain $U$ of $\R^d$ with $\partial U$ compact
\[ E(v,U) = \int_{U} |\grad v|^2 + Q(x)^2 1_{\{v >0\}} \ dx\]
 over $v \in H^1(U)$ with $v = g$ on $\partial U$ (call the admissible class $H^1_g(U)$).  Since $\partial U$ is compact there are finite energy states. Suppose that $u$ and $v$ both minimize $E(\cdot,U)$ over $H^1_g(U)$, then $u \wedge v$ and $u \vee v$ are admissible and
 \begin{equation}\label{e.intersectionunion}
 E(u \wedge v,U) + E(u \vee v,U) = E(u,U) + E(v,U). 
 \end{equation}
 Thus $u\wedge v$ and $u \vee v$ are both minimizers as well.  
 
 We can define a smallest energy minimizer $u$ with the property that that any other minimizer $v$ must have $v \geq u$.  Call $\mathcal{M} \subset H^1_g(U)$ to be class of energy minimizers and let $u_k\in \mathcal{M}$ be a sequence with $\int_U u_k \to \inf_{v \in \mathcal{M}} \int_U v$.  Without loss $u_k$ are monotone decreasing, otherwise take instead the sequence $u_1 \wedge \cdots \wedge u_k$.  By \lref{localmin} the $u_k$ are solutions of \eref{unitscale} and by \lref{properties} they are uniformly Lipschitz continuous.  Since the energies $E(u_k,U)$ are constant, up to taking a subsequence the $u_k \rightharpoonup u$ in $H^1(U)$ and uniformly in $\overline{U}$.  Thus $u\in H^1_g(U)$ and $E(u,U) \leq \liminf E(u_k,U)$, so $u \in \mathcal{M}$, and
 \[ \int_U u  = \inf_{v \in \mathcal{M}} \int_U v.\]
 Therefore there cannot be any $u' \in \mathcal{M}$ with $u' < u$ somewhere.

2. What we would like to do is consider the global minimizer in the domain $U = \{ x \cdot p <0\}$ of
\[ E(v,U) = \int_{U} |\grad v|^2 + Q(x)^2 1_{\{v >0\}} \ dx\]
 over $v \in H^1_{loc}(U)$ with $v(x) = t$ on $x \cdot p \geq 0$.  We would expect this minimizer to have the Birkhoff property. This does not quite make sense due to the infinite domain.  
 
 We take a different approach, finding compactness by enforcing periodicity. We use that $p$ is rational, then $p^\perp \cap \Z^d$ is a periodicity lattice for $Q$ and for the boundary data on $\partial U$. Find the smallest energy minimizer $v_m$ over the periodized domain $U \bmod mp^\perp \cap \Z^d$.  Now $\partial (U \bmod mp^\perp \cap \Z^d) = \partial U \bmod mp^\perp \cap \Z^d$ which is compact, so the argument of the first part of the proof still applies to prove existence of a smallest minimizer.  The $v_m$ solve \eref{appcell}, they are uniformly Lipschitz continuous and $mp^\perp \cap\Z^d$-periodic.  Furthermore almost the same argument of \lref{monotonicity} applies and $v_m$ satisfy the Birkhoff property
 \begin{equation}\label{e.birkhoffvm}
  v_m(\cdot + k) \geq v_m(\cdot) \ \hbox{ in $U$ for } \ k \cdot \xi \geq 0.
  \end{equation}
 In particular $v_m$ is actually $p^\perp \cap \Z^d$-periodic, and therefore $v_m = v_1$.  By Lemma~\ref{localmin} $v_1$ is also viscosity solution of \eref{unitscale}, and by the same proof as above in \lref{width} the free boundary stays in a bounded width slab, in particular independent of $t$,
 \begin{equation}\label{e.bddwidthenergy}
  \{x \cdot p > -r(t)\}\subset \{v^1 >0\} \subset \{x \cdot p > r(t) - C\}
  \end{equation}
 where $r(t) = \inf_{x \in \partial \{ v>0\}} |x \cdot p|$ and $C$ is universal.
 
 Now we can check that $v_1$ is an absolute energy minimizer.  Let $\varphi \in H^1_0(B)$ for any ball $B \subset U$.  For $m$ sufficiently large $B$ is contained in a single unit period cell of $m p^\perp \cap \Z^d$.  Then, consider the periodization of $\varphi$
 \[\tilde{\varphi}(x) = \sum_{k \in m p^\perp \cap \Z^d} \varphi(x+k) \]
 which is well defined and equal to $\varphi(\cdot+k)$ in $B+k$ for any $k \in p^\perp \cap \Z^d$ and zero in the complement of $\cup_{k \in p^\perp \cap \Z^d} B+k$.
  Abusing notation we also write $B$ for the subset of $U \bmod p^\perp \cap \Z^d$ corresponding to it.  Using the minimization property of $v_1 = v_m$
 \begin{align*}
 E(v_1+\varphi,B)+E(v_1,U \bmod p^\perp \cap \Z^d \setminus B) &= E(v_1+\tilde\varphi,U \bmod p^\perp \cap \Z^d)  \\
 &\geq E(v_1,U \bmod p^\perp \cap \Z^d) \\
 &= E(v_1,U \bmod p^\perp \cap \Z^d \setminus B) + E(v_1,B) 
 \end{align*}
 which proves the absolute minimum property for $v_1$.

 3. Now we can compute the energy per unit period cell of the smallest energy minimizing solution $v_1 = v$ as a function of the approximate slope
 \[ \alpha(t) = \frac{t}{r(t)} \ \hbox{ with } \ r(t) = \inf_{x \in \partial \{ v>0\}} |x \cdot p|.\]
 Call $Q_p$ to be the unit period cell of $p^\perp \cap \Z^d$.  For any $\delta>0$ and $t \gg 1/\delta$ we can compute the energy
 \begin{align*}
  \frac{1}{|Q_p|}E(v,U \bmod p^\perp \cap \Z^d) &= \alpha(t)^2r(t)+\langle Q^2 \rangle r(t)+O(\delta t )\\
  &= [\alpha(t)+\langle Q^2 \rangle \alpha(t)^{-1}+O(\delta)]t
  \end{align*}
  Compare this with the energy of the linear solution $\ell(x) = (\langle Q^2 \rangle^{1/2} p\cdot x+t)_+$ which is also $p^\perp \cap \Z^d$ periodic,
   \begin{align*} 
   \frac{1}{|Q_p|}E(v,U \bmod p^\perp \cap \Z^d) &\leq  \frac{1}{|Q_p|}E(\ell,U \bmod p^\perp \cap \Z^d) \\
   &= \langle Q^2 \rangle\frac{t}{\langle Q^2 \rangle^{1/2}}\\
   &\quad +\int_{U \bmod p^\perp \cap \Z^d} Q^2(x) {\bf 1}_{\{-t/\langle Q^2 \rangle^{1/2} <x \cdot p< 0\}} \ dx \\
&= 2\langle Q^2 \rangle^{1/2}t+O( 1).
   \end{align*}
   Putting these together,
   \[ \alpha(t)+\langle Q^2 \rangle \alpha(t)^{-1} \leq 2\langle Q^2 \rangle^{1/2} + \frac{C}{t}.\]
   Note that the function $\alpha \mapsto \alpha+\langle Q^2 \rangle \alpha^{-1}$ is convex and has its unique minimum on $\R_+$ at $\alpha = \langle Q^2 \rangle^{1/2}$, furthermore the second derivative has a lower bound by $c\langle Q^2 \rangle^{-1/2}$ in a unit neighborhood of the minimum, thus
   \[ |\alpha(t) - \langle Q^2 \rangle^{1/2}| \leq \frac{C}{t^{1/2}}.\]

4. Now we take the limit $t \to \infty$, the minimizer $v$ constructed above of course depends on $t$ which we now need to keep track of, write $v = v^t$. Now translate, let $k_t \in \Z^d$ with $|k_t \cdot p + r(t)| \leq \sqrt{d}$.  Define
\[ \tilde{v}^t(x) = v^t(x-k_t).\]
The $\tilde{v}^t$ are uniformly Lipschitz continuous, by the bounded width \eref{bddwidthenergy}
\[ \{ x \cdot p > - C\} \subset \{\tilde{v}^t > 0 \} \subset \{ x \cdot p >C\}\]
for a universal $C$, $v^t$ satisfy the Birkhoff property, and they are absolute minimizers in the sense that for any ball $B \subset \{x \cdot p < |k_t \cdot p|\}$ and any perturbation $\varphi \in H^1_0(B)$
\begin{equation}\label{e.energytildev}
 E(\tilde v^t,B) \leq E(\tilde v^t+\varphi,B).
 \end{equation}
Using again the bounded width, the boundary data $\tilde{v}^t = t$ on $x \cdot p = k_t \cdot p = r(t) + O(1)$, and the maximum principle
\[ (\alpha(t) x \cdot p - C)_+\leq \tilde{v}^t(x) \leq (\alpha(t) x \cdot p + C)_+.\]
Finally we take the limit $t\to \infty$, up to a subsequence the $\tilde{v}^t$ converge locally uniformly to some $w$, by the nondegeneracy of global minimizers \lref{nondegen} the boundaries $\partial \{\tilde{v}^t >0\}$ converge locally in Hausdorff distance to $\partial \{w>0\}$.  By the stability of viscosity solutions under uniform convergence $w$ is a solution of \eref{unitscale} in $\R^d$.  

Next we aim to show that $\grad \tilde{v}^t \to \grad w$ almost everywhere. By the Hausdorff convergence of $\partial \{\tilde{v}^t>0\}$ if $x \in \R^d \setminus \partial \{w>0\}$ then $B_r(x) \subset  \R^d \setminus \partial \{\tilde{v}^t>0\}$ for sufficiently small $r$ and large $t$.  Then $\tilde{v}^t$ is either harmonic or identically zero in $B_r(x)$ so $\grad \tilde{v}^t \to \grad w$ uniformly in $B_{r/2}(x)$.  We just need to show that $\partial \{w>0\}$ has measure $0$, the argument is from \cite{ACF}, if the set had positive measure there would have to be a point $x_0 \in \partial \{w>0\}$ with lebesgue density $1$.  Then by Lipschitz continuity $w(x) = o(|x-x_0|)$ as $ x \to x_0$, this contradicts the nondegeneracy \lref{nondegen}.

It is easy to check by the uniform convergence, and Hausdorff convergence of positivity sets that $w$ inherits the bounded width, Birkhoff property, and the bounds
\[ (\langle Q^2 \rangle^{1/2} x \cdot p - C)_+\leq \tilde{v}^t(x) \leq (\langle Q^2 \rangle^{1/2} x \cdot p + C)_+ \]
since $\alpha(t) \to \langle Q^2 \rangle^{1/2}$ as $t \to \infty$ as shown above.  The energy minimization property follows from the local Hausdorff convergence of $\partial \{\tilde{v}^t >0\}$ and the a.e. convergence $\grad \tilde{v}^t \to \grad w$, from that the energies on both sides of \eref{energytildev} converge.

Finally we need to cover the irrational directions.  Take a sequence $w_n$ as constructed above with slopes $\langle Q^2 \rangle^{1/2}p_n$ with $p_n$ rational converging to $p$.  As before, up to extracting a subsequence, $w_n$ converge to some $w$ locally uniformly, $\partial \{w_n>0\}$ converge in Hausdorff distance, and $\grad w_n \to \grad w$ almost everywhere.  As just argued above, all of the desired properties are stable with respect to this convergence.

5. Finally we show the existence of global locally minimizing plane-like solutions with the Birkhoff property and slope $\alpha \in (Q_*(p),Q^*(p))$.  The argument is almost exactly the same as above except that, in step $2$ instead of looking for the smallest energy minimizer of $E(\cdot,U \bmod p^\perp \cap \Z^d)$ with boundary data $v = t$ on $\partial U$, we constrain the minimizer using the minimal and maximal plane-like solutions.  Let $v_*$ and  $v^*$ be, respectively, plane-like solutions of \eref{cell} with slopes $Q_*(p)$ and $Q^*(p)$ as constructed in \lref{minimalsuper} with lattice translations so that 
\begin{equation}\label{e.vsubvsup}
\alpha(x \cdot p- C)_+ \prec  \frac{\alpha}{Q_*(p)}v_* \prec \alpha (x \cdot p - \sqrt{d})_+  \prec \alpha (x \cdot p + \sqrt{d})_+ \prec \frac{\alpha}{Q^*(p)}v^* \prec \alpha(x \cdot p+ C)_+ .
 \end{equation}
Since $\tfrac{\alpha}{Q_*(p)} > 1 > \tfrac{\alpha}{Q^*(p)}$ we can choose $\delta>0$ sufficiently small, depending on $\alpha$, so that the sup/inf convolutions
\[\underline{v}(x) = \tfrac{\alpha}{Q_*(p)}\sup_{y\in B_\delta(x)} v_*(y) \ \hbox{ and } \ \overline{v}(x) = \tfrac{\alpha}{Q^*(p)}\inf_{y\in B_\delta(x)} v^*(y) \]
are, respectively, an inner-regular $R$-subsolution and an outer regular $R$-supersolution of \eref{unitscale} still satisfying \eref{vsubvsup}.  Now define the constraint set
\[
\mathcal{A}_t = \left\{
v \in H^1_{loc}(\R^d): 
\begin{array}{l}
\underline{v}(\cdot + k_t) \leq v  \leq \overline{v}(\cdot + k_t), \vspace{1.5mm}\\
 \hbox{$v$ is $p^\perp \cap \Z^d$-periodic, and $v = t$ on $\partial U$}\end{array}\right\}.
\]
Here $k_t \in \Z^d$ with $|k_t \cdot p +\tfrac{t}{\alpha}| \leq \sqrt{d}$ so that
\[ \underline{v}(x) < t - C < t+C < \overline{v}(x) \ \hbox{ on } \ (x +k_t)\cdot p =0.\] 
The constraints are $p^\perp \cap \Z^d$-periodic by \lref{minimalsuper}, so the arguments above give the existence of a smallest periodic minimizer $v^t$ in $\mathcal{A}_t$. By \lref{localmin} the minimizer $v^t$ is a solution of \eref{unitscale} and
\[ \underline{v}(\cdot+k_t) \prec v_t \prec \overline{v}(\cdot+k_t).\]
 Almost all of the remainder of the arguments in parts $2$ and $4$ above are the same, except we will only get the local minimization property, for any ball $B \subset U$ with sufficiently small radius and any $\varphi \in H^1_0(B)$ with $\|\varphi\|_\infty$ sufficiently small,
\[E(v_t,B) \leq E(v_t+\varphi,B). \]
After taking the limit of the $v_t(x - k_t)$, we get a $\Z^d \cap p^\perp$-periodic solution $w$ of \eref{unitscale} on $\R^d$ with $\underline{v} \leq w \leq \overline{v}$, we need to check that $w$ does not touch the constraints in $\overline{\{w>0\}}$
\[ \inf_{\{v_*>0\}} (w- \underline{v})>0 \ \hbox{ and } \ \inf_{\{w>0\}} (\overline{v}-w)>0\]
so that the same local minimization property as above holds.  By periodicity and maximum principle if one of the infima above is zero, then touching must happen at a point $x\in \partial \{w>0\}$, but this is a contradiction of the comparison principle for inner regular $R$-subsolutions / outer regular $R$-supersolutions \lref{Rcomparison}.

\end{proof}

\subsection{Energy minimizers near curved surfaces}
Now we make our last main argument having to do with energy minimization.  We construct global energy minimizers whose free boundary stays close to the graph of a smooth function.  Basically the argument amounts to the $\Gamma$-convergence of the energies $E_\ep$ \eref{energyep} to $E_0$ \eref{gammalimit}.  

We define a convenient type of domain for our construction.  Let $e \in S^{d-1}$ and $U \subset \{x \cdot e = 0\}$ relatively open and connected.  Define
\[ D_e(U) = \{ x \in \R^d : x \cdot e >0 \ \hbox{ and } \ x - (x\cdot e) e \in U\}.\]
It is the part of the half-space $x \cdot e>0$ above $U$.

\begin{lemma}\label{p.Qcontenergy}
Let $e \in S^{d-1}$ and the domain $D = D_e(U)$ for some relatively open, connected, and bounded $U \subset \{x \cdot e = 0\}$.  If $\varphi \in C^\infty(\overline{D})$ is harmonic in $\{\varphi>0\} \cap D$, $\inf_U\varphi>0$, $\varphi$ is a strict subsolution of
\[  |\grad \varphi| > \langle Q^2 \rangle^{1/2} \ \hbox{ on } \ \partial \{ \varphi >0\} \cap D,\]  
and
\[\frac{\grad\varphi}{ |\grad \varphi|} \cdot e >0 \ \hbox{ in } \ \overline{\{ \varphi>0\}} \cap \overline{D},\]
then for all $\ep>0$ there exists a subsolution $v^\ep$ of \eref{fb0} in $D$ such that
\[ \lim_{\ep \to 0} \sup_{\partial D} |v^\ep - v| = 0, \quad \lim_{\ep \to 0} \inf_{D} (v^\ep - v)  \geq 0, \]
and
\[\lim_{\ep \to 0} d_H(\{v^\ep>0\} \cap D,(\{v^\ep>0\} \cup \{\varphi>0\}) \cap D) = 0.\]
\end{lemma}

The same result holds for smooth supersolutions with the inequalities reversed.

\begin{proof}[Proof of \pref{Qcontenergy}]
 The construction of the solution is by finding the global energy minimizer.  Let $v$ be any solution of
\begin{equation}\label{e.minsoln}
 \Delta v = 0 \ \hbox{ in }  D' \cap \{u>0\} \ \hbox{ with } \ |\grad v| = \langle Q^2 \rangle^{1/2} \ \hbox{ on } \ \partial \{v>0\}
 \end{equation}
with $v = \varphi$ on $\partial D$. Then we claim $v \geq \varphi$.  If not slide $\varphi_t(x) = \varphi(x-te)$ increasing $t$, and decreasing $\varphi_t$, until it touches $v$ from below at a free boundary point $x$.  Touching cannot occur on $\partial D$ because $v = \varphi > \varphi_t$ there for $t>0$, and it cannot occur in $\{v>0\}$ by the strong maximum principle.  Now since $\varphi$ is smooth the viscosity supersolution condition says $|\grad \varphi(x)| \leq \langle Q^2 \rangle^{1/2}$ which is a contradiction of the strict subsolution property of $\varphi$.

Let $v^\ep$ be a global minimizer of the energy $E_\ep$ on the constraint set
\[ \mathcal{A} = \{ w \in H^{1}(D'): \ w = \varphi \ \hbox{ on } \ \partial D\}. \]
By \lref{localmin} there exists such a minimizer, $v^\ep$ is a viscosity solution of \eref{fb0} in $D$, and it satisfies the usual Lipschitz and nondegeneracy properties \lref{properties} and \lref{nondegen}.  Also there exists a minimizer $v^0$ corresponding to $E_0$ solving \eref{minsoln}.

By \cite[Theorem 4.3, Theorem 4.5]{AltCaffarelli}, which only relies on the upper and lower bounds for $Q$ and not the regularity, the free boundary $\partial \{v^\ep>0\}$ satisfies the Hausdorff dimension bound
\[ cr^{d-1}\leq \mathcal{H}^{d-1}(\partial \{v^\ep>0\} \cap B_r(x)) \leq Cr^{d-1}\]
for any $x \in \partial \{v^\ep>0\}$ with $B_r(x) \subset D$.  Thus the total number of the $\ep\Z^d$ lattice cubes which intersect $\partial \{v^\ep >0 \} \cup \partial D$ is bounded from above by
\[ \#\{k \in \Z^d: (\partial \{v^\ep >0 \} \cup \partial D) \cap([0,\ep)^d + \ep k) \neq \emptyset\} \leq C \ep^{1-d}\]
where the constant $C$ depends on the domain $D$.  Therefore
\[ |E_\ep(v^\ep) - E_0(v^\ep)| \leq C\ep.\]
Since $\partial \{v>0\}$ has the same Hausdorff measure bounds, the same estimate holds for $v$.  Then using the minimization properties of each $v$ and $v^\ep$ we obtain
\[ |E_0(v) - E_0(v^\ep)| \leq C\ep.\]
Now, taking a subsequence as we did in the proof of \pref{mean}, $v^\ep \to u$ uniformly, by nondegeneracy $\{v^\ep>0\} \to \{u>0\}$ in Hausdorff distance in $D$, and $\grad v^\ep \to \grad u$ almost everywhere.  This means that the energies converge and
\[ E_0(u) = E_0(v)\]
with the same boundary data on $\partial D$.   Thus $u$ minimizes $E_0$ over $\mathcal{A}$, and, and therefore is a solution of \eref{minsoln}.  Thus every subsequence has a subsequence converging uniformly to some $v \geq   \varphi$ and therefore
\[ \lim_{\ep \to 0} \inf_{D} (v^\ep - \varphi) \geq 0.\]

\end{proof}

\section{Examples}\label{s.examples}
In this section we give several examples where we can either exactly compute $Q_*,Q^*$ or achieve some explicit bounds.  The contents of this section will prove parts \partref{mainQdisc} and \partref{mainQint} of \tref{mainQ}.
\subsection{Laminar media}
Consider the special case of a laminar medium, $Q = Q(x_1)$ depends only one a single variable. The pinning interval can be explicitly identified, for $p \in S^{d-1}$,
\[
 I(p) = \left\{
\begin{array}{ll}
\langle Q^2 \rangle^{1/2} & p \neq \pm e_1 \vspace{1.5mm}\\
{[} \min Q, \max Q  {]} & p = \pm e_1.
\end{array}\right.
\]
The cell problem can be solved exactly in the case $p = e_1$ (or $-e_1$), for any $\alpha\in {[} \min Q, \max Q  {]}$,
\[ u_\alpha(x) = \alpha [(x -x_\alpha) \cdot e_1]_+ \ \hbox{ for any } \ x_\alpha \in Q^{-1}(\{\alpha\}). \]
From \pref{mean} we already know that $\langle Q^2 \rangle^{1/2} \in I(p)$.  The following lemma completes the characterization of $I(p)$ in the laminar case, and is a bit more general.
\begin{lemma}\label{l.laminar}
Suppose that $\grad Q \cdot e = 0$ for some unit direction $e$.  Then if $p \cdot e \neq 0$ then $Q_*(p) = Q^*(p) = \langle Q^2 \rangle ^{1/2}$.
\end{lemma}
The idea for this lemma was communicated to us by I. Kim.  Basically if there were two distinct slopes in the pinning interval we could slide the smaller slope solution in the direction $e$ until it touches the larger slope solution from below contradicting the strong maximum principle.  There are some technical difficulties, the usual difficulty of regularity in comparison principles for viscosity solutions, the unbounded domain, and the lack of a nondegeneracy estimate for the maximal subsolutions in dimensions $d \geq 3$ all need to be dealt with.  Unfortunately this causes the proof to be rather long despite the simple idea.
\begin{proof}
Suppose that $p \cdot e \neq 0$,  and that $Q_*(p) < Q^*(p)$.  We take $p \cdot e >0$, the other case is similar.  There are solutions to \eref{appcell} $u_*$ and $u^*$ with respective slopes $Q_*$ and $Q^*$ and
\begin{equation}\label{e.stargrowth}
 \sup_{x \in \R^d} |u^*(x) - Q^* (x \cdot p)_+|  \leq C \ \hbox{ and } \  \sup_{x \in \R^d} |u_*(x) - Q_* (x \cdot p)_+| <+\infty.
 \end{equation}
We need to regularize $u^*,u_*$ a bit for our comparison argument, we do standard inf and sup convolutions
\[ u^*_\delta(x) = \inf_{|y-x| \leq \delta} u^*(x) \ \hbox{ and } \ u_*^\delta(x) = \sup_{|y-x| \leq \delta} u_*(x).\]
Now $u^*_\delta$ and $u_*^\delta$ satisfy the same bounds as above, are, respectively, sub and super harmonic in their positivity sets, and satisfy the free boundary condition, in the viscosity sense,
\[ |\grad u^*_\delta|(x) \leq \sup_{B_\delta(x)}Q(y) \ \hbox{ and } \ |\grad u_*^\delta|(x) \geq  \inf_{B_\delta(x)}Q(y) \]
for $x$ in the respective free boundaries $\partial \{ u^*_\delta>0\}$ and $\partial \{u_*^\delta >0\}$.

Call $\lambda = (1-2\|\grad Q\|_\infty \delta/\min Q)$, then 
\[ \lambda |\grad {u}^*_\delta|(x) \leq \lambda \sup_{B_\delta(x)}Q(y) \leq  \inf_{B_\delta(x)}Q(y) \ \hbox{ for } \ x\in \partial \{ u^*_\delta>0\}.\]
For $\delta$ sufficiently small $\lambda Q^* > Q_*$ still.

Next translate $\lambda {u}^*_\delta$ in the $e$ direction by
\[v_t(x) = \lambda {u}^*_\delta(x-te).\]
From the invariance of $Q$ in the $e$ direction $v_t$ is still a supersolution with $|\grad v_t|(x) \leq  \inf_{B_\delta(x)}Q(y)$ on the free boundary and
\[ \sup_{x \in \R^d} |v_t(x) - \lambda Q^* (x \cdot p -te\cdot p)_+|.\]
For sufficiently large positive $t$, $t > T_+$, we will have $v_t(x) >  {u}_*^\delta(x)$ in $\overline{\{{u}_*^\delta >0\}}$, while for sufficiently large negative $t$, $t < T_-$, $\sup_{x} ({u}_*^\delta(x) - v_t(x)) >0$.  Then decreasing $t$ from $T_+$ to $T_-$, by continuity, we find that $\inf_{\overline{\{{u}_*^\delta >0\}}}(v_{t_0}-{u}_*^\delta) = 0$ at some value $t_0$. 

If the infimum is not achieved, take a sequence of lattice translations $k_n$, with $|k_n \cdot p| \leq C$ so that 
\[ \min_{[0,1)^d \cap \overline{\{{u}_*^\delta >0\}}}(v_{t_0}-{u}_*^\delta)(x+k_n)  \leq 1/n.\]
Say that the minimum occurs at a point $x_n \in [0,1)^d$.  By the Lipschitz continuity, \lref{properties}, up to taking a subsequence we can assume that $x_n \to x_\infty$ and the translations $v_{t_0}(x+k_n)$ and ${u}_*^\delta(x+k_n)$ converge locally uniformly to some $v$ and $u$ respectively, and hence satisfy the same viscosity solutions conditions.  Now we need to check that the touching point $x_\infty$ is actually in $\overline{\{u >0\}}$.  
For this we want to use the nondegeneracy \lref{nondegen}. Note that $v_{t_0}$, as an inf convolution of the minimal supersolution $u^*_\delta$, satisfies the nondegeneracy estimate from \lref{nondegen}
\[ v_{t_0}(x) \geq c d(x,\partial\{v_{t_0} >0\}).\]
Now $v_{t_0}(x_n+k) \leq 1/n$ and so there is a point $y_n \in \partial \{v_{t_0}(\cdot+k_n)>0\}$ with $|y_n-x_n| \leq C/n$.  The positivity set $\{v_{t_0}(\cdot+k_n)>0\}$ has an exterior ball $B$ of radius $\delta$ at $y_n$, let $B'$ be the touching ball of radius $\delta/2$.  Then slide $B'$ by $B'+t(x_n - y_n)$ until it touches $\{u_*^\delta>0\}$ from the outside at some point $z_n$ for some $0 \leq t \leq 1$.  Since the ball has moved by at most distance $C/n$ the touching could only occur at a point of $\partial B'$ which is within distance $C/n$ of $\partial B$.  The boundaries of $\partial B'$ and $\partial B$ separate quadratically near $y_n$
\[ d(z,\partial B) \geq \frac{c}{\delta} |z-y_n|^2 \ \hbox{ for } \ z \in \partial B'\]
and so the touching points $z_n = y_n + O(n^{-1}+\delta^{1/2}n^{-1/2}) = x_n +O(n^{-1}+\delta^{1/2}n^{-1/2})$, in particular it also converges to $x_\infty$ as $n \to \infty$.  Thus for any $0<r \leq \delta$, by the nondegeneracy at outer regular points, \lref{nondegen},
\[ \sup_{B_r(z_n)} u_{*}^\delta(\cdot+k_n) \geq c r.\]
Passing to the limit we obtain the same nondegeneracy at $x_\infty$ for $u$ implying that indeed $x_\infty \in \partial\{ u>0\}$.

Thus we find $u$ and $v$ sub/superharmonic in their positivity sets, $\{u>0\}$ and $\{v>0\}$ are respectively $\delta$ inner regular and $\delta$ outer regular, $u$ and $v$ are respectively sub and supersolutions of $|\grad w| = \inf_{B_\delta(x)}Q(y)$ on $\partial \{w>0\}$,  they have distinct asymptotic slopes $Q_* <\lambda Q^*$, and $u$ touches $v$ from below at some point of $\R^d$. This contradicts \lref{strictcomparison}.
\end{proof}

\subsection{An example with pinning at every direction}
The special structure of laminar media prevents pinning except at the laminar direction.  Without special structural assumptions our conjecture is that pinning at \emph{every} direction is generic.  Despite this expectation it is not that obvious even to come up with one field $Q$ with this property, we give such an example here.

Let $\rho$ be a smooth radially symmetric bump function, $\rho \equiv 0$ outside of $B_{1/2}(0)$ and $\int \rho^2  \geq 1$. Given parameters $A>1$, $1>\delta >0$ to be chosen large and small respectively, define
\[ Q(x) = 1+\sum_{k \in \Z^d} A\rho(\tfrac{x-k}{\delta}).\]
\begin{lemma}\label{l.everydir}
Let $Q$ as above. If $\delta$ is sufficiently small, depending on dimension, and $A \geq C(d)\delta^{-(d-1)/2}$, then $[Q_*(e),Q^*(e)]$ is nontrivial for all $e \in S^{d-1}$.
\end{lemma}
\begin{proof}
By the results of the previous section
\[ Q^*(e) \geq \langle Q^2 \rangle^{1/2} \geq (1+A^2\delta^d)^{1/2}.\]
 Next, for $\delta$ sufficiently small universal, we construct a subsolution with slope at most $1+C\delta$, a small perturbation of $(e \cdot x)_+$, yielding
\[Q_*(e) \leq 1+C\delta.\]
Then as long as
\[ A \geq C\delta^{-(d-1)/2},\]
for a sufficiently large universal $C$, the pinning interval is nontrivial.

To make the second part of the argument precise it is convenient to use the type of perturbations described below in \lref{caffperturb1} and \lref{perturbsubcond}.  Consider $\Lambda$ the projection of $\Z^d \cap \{ - 2\delta < x \cdot e < -\delta/2\}$ onto $x \cdot e = 0$.  For $\delta$ sufficiently small, universal, each pair $z,w \in \Lambda$ are separated by at least distance $1/2$.  Let $\zeta(s)$ be a smooth function on $\R_+$ which is equal to $1$ for $0 \leq s \leq 1/5$ and equal to zero for $s \geq 1/4$.  Suppose that $\delta/2 < 1/5$.  Define
\[ h(x) = \delta+\tfrac{3}{2}\delta \zeta\left(d(x,\Lambda)\right)\]
and
\[ 
\begin{cases}
\Delta \psi = 0 &\hbox{in } x \cdot e >-\tfrac{5}{2}\delta \\
\psi(x) = \begin{cases}
h(x)^{\frac{1}{2-d}} & d \geq 3\\
\log h(x) & d=2
\end{cases} &\hbox{on } x \cdot e = -\tfrac{5}{2}\delta.
\end{cases}
\]
Since $h$ is smooth with $\|h\|_{C^2} \leq C\delta$, by the boundary regularity for the Dirichlet problem 
\[ \|\psi\|_{C^1} \leq C\delta^{\frac{1}{2-d}} \ \hbox{ in $d \geq 3$ or } \ \|\psi\|_{C^1} \leq C \ \hbox{ in } \ d=2.\]
  Then define $\varphi(x) = \psi(x)^{2-d}$, or $\varphi(x) = \exp(\psi(x))$ in $d=2$.  By maximum principle $\delta \leq \varphi \leq 5 \delta/2$.  Since $h$ is smooth with $\|h\|_{C^2} \leq C\delta$, by the boundary regularity for the Dirichlet problem, $\|\grad \varphi\|_\infty \leq C\delta$.  Furthermore, calling $x' = x- (x\cdot e)e$,
  \[ \delta \leq \varphi(x) \leq \delta(1+C\delta) \ \hbox{ for } \ -\tfrac{5}{2}\delta \leq x\cdot e \leq -\delta \ \hbox{ and } \ d(x',\Lambda) \geq 1/5.  \]
  
  Then define the sup convolution
\[ v(x) = \sup_{|\sigma| \leq 1} [(x+\sigma\varphi(x))\cdot e]_+.\]
By \lref{caffperturb1} and \lref{perturbsubcond}, using the upper bound on $\|\grad \varphi\|_\infty$, $v$ is subharmonic in its positivity set and
\[ |\grad v(x)| \geq 1-C\delta \ \hbox{ on } \ \partial \{v>0\}. \]
We aim to show that the free boundary of $v$ does not intersect a $\delta/2$ neighborhood of any lattice point.  Then $(1+C\delta)v$ will be a subsolution of \eref{unitscale} and we could conclude.  

Call the infinite cylinder $\Gamma_r = \{|x'| <r\}$.  Away from $\Lambda+\Gamma_{1/5}$ the free boundary of $v(x)$ satisfies
\[  \partial \{v>0\} \cap (\Lambda+\Gamma_{1/5})^C  \subset \{ -(1+C\delta)\delta \leq x \cdot e \leq -\delta\}\]
For $\delta$ sufficiently small so that $C\delta <1/2$ this will not intersect the $\Z^d+B_{\delta/2}$.   On the other hand, for any $z \in \Lambda$,
\[ \partial \{v>0\} \cap (z+\Gamma_{1/5}) = \{x \cdot e = -5\delta/2\} \cap \Gamma_{1/5}\]
which, by the set up, will also not intersect $\Z^d+B_{\delta/2}$.

\end{proof}

\subsection{Structure of discontinuities of $Q^*,Q_*$ in $d \geq 3$} Now we combine the previous two examples, take $Q$ as in the previous section on $\R^2$ and then extend to $\R^3$ as a constant in the $x_3$ direction.  That is
\[ Q(x) = 1+\sum_{k \in \Z^2} A\rho(\tfrac{x'-k}{\delta}) \]
where $x = x'+x_3e_3$.  Then by \lref{laminar} and \lref{everydir}
\[ [Q_*(p),Q^*(p)] = \langle Q^2 \rangle^{1/2} \ \hbox{ for } \ p_3 \neq 0 \]
while for $p_3 = 0$ it holds $Q^*(p) > \langle Q^2 \rangle^{1/2} > Q_*(p)$.  Thus the endpoints of the pinning interval are discontinuous along the hyperplane $p_3 = 0$.

A similar construction is possible for any rational subspace, we carry it out below.

\begin{proof}[Proof of \tref{mainQ} part \partref{mainQdisc}.]
Let $\xi_1,\dots \xi_{k} \in \Z^d \setminus \{0\}$ linearly independent and consider the rational subspace spanned by the columns of $\Xi = [\xi_1,\dots,\xi_k]$.  Complete $\xi_1,\dots \xi_{k}$ to a basis of $\R^d$ by addending lattice vectors $\xi_{k+1},\dots,\xi_d$.  We will choose $Q$ with $\xi_{k+1} \cdot \grad Q = \cdots = \xi_{d} \cdot \grad Q= 0$. 

Then, by \lref{laminar}, if $e \in S^{d-1} \setminus \Xi$ then $e \cdot \xi_j \neq 0$ for some $k+1 \leq j \leq d$ and therefore
\[ Q_*(e) = Q^*(e) = \langle Q^2 \rangle^{1/2}.\]
Now, to be more precise, we choose
\[ Q(x) = 1+\sum_{m \in \Z^k} A\rho(\frac{x\cdot \xi_1-m_1}{\delta},\dots,\frac{x\cdot \xi_k-m_k}{\delta}) \]
If $e \in \Xi$ then, using the invariance of $Q$ in the $\Xi^\perp$ directions, by \sref{cell} above there are cell problem solutions sharing the same invariances, so we can just look for a cell problem solution of the form
\[ v(x) = u(x \cdot f_1,\dots,x\cdot f_k) \]
where $F = [f_1\dots f_k]$ is an orthonormal basis for $\Xi$ and $u: \R^k \to [0,\infty)$ is a solution of
\[ \Delta u = 0 \ \hbox{ in } \ \{u>0\}, \ \hbox{ with } \ |\grad u|(y) = Q(Fy) \ \hbox{ on } \ \partial \{u>0\}\]
with
\[ \sup_{\R^k} |v(y) - \alpha (e\cdot Fy)_+| < + \infty.\]
Now $Q$ is periodic with respect to the lattice generated by $\mathcal{Z} = \{F^T\xi_j\}_{j=1,\dots,k}$, even though this is not the lattice $\Z^k$, basically the same arguments as \lref{everydir} show that for $\delta$ sufficiently small and $A$ sufficiently large, depending only on universal parameters and the minimum distance between lattice points of $\mathcal{Z}$, there is a nontrivial pinning interval bounded below in width independent of $e$.
\end{proof}

\section{Limits of solutions to the $\ep$-problem}\label{s.limitsofep}

Let $U \subset \R^d$ an open domain, consider a sequence of solutions $u^\ep$ to
\begin{equation}\label{e.fbep}
  \begin{cases}
    \Delta u^\ep = 0 & \mbox{in } \{ u > 0 \} \cap U \\
    |\grad u^\ep| = Q(x/\ep) & \mbox{on } \partial \{ u > 0 \} \cap U
  \end{cases}
\end{equation}
which converge locally uniformly in $U$ to some $u$.  Then, we will show in this section, that $u$ solves in the viscosity sense
\begin{equation}\label{e.fbhom}
  \begin{cases}
    \Delta u = 0 & \mbox{in } \{ u > 0 \} \\
    \grad u \in [Q_*(\grad u),Q^*(\grad u)] & \mbox{on } \partial \{ u > 0 \}.
  \end{cases}
  \end{equation}
  The above equation is to be interpreted in the viscosity sense.  This is the content of \tref{kcl}.
  
  As mentioned in the introduction  The result of this section is not new, it can be derived from the paper of Kim \cite{Kim} on an associated dynamic problem.  Also, a special case is done Caffarelli-Lee~\cite[lemma 3.4]{CaffarelliLee}.  We include the argument here for completeness, and because it is quite simple in the static setting we consider here.

  Note that, assuming the $u^\ep$ are uniformly bounded, they are also uniformly Lipschitz continuous by \lref{properties} and therefore have a uniformly convergent subsequence.  Convergence of the whole sequence is unlikely to hold without some additional specification, e.g. minimality, maximality, energy minimization or in the case $Q^*(e) = Q_*(e)$ at every direction.  We will make this rigorous below in \sref{minimalsoln} when we discuss the limits of the minimal supersolution and maximal subsolution.

  \begin{proof}[Proof of \tref{kcl}.]
  It is standard to check $\Delta u = 0$ in $\{u >0\}$. We check the supersolution condition on the free boundary, the subsolution condition is analogous.   Suppose $\varphi$ is a smooth test function touching $u$ from below at some point $x_0 \in \partial {\{u >0\}} \cap U$ with
  \[ \Delta \varphi(x_0) > 0. \]  
  By standard arguments one can perturb so that $\varphi(x) < u(x)$ for $x \neq x_0$ and $\Delta\varphi(x) >0$ in a small neighborhood of $x_0$, which we still call $U$.  Now there exists a sequence $U \ni x_\ep \to x_0$ and constants $c_\ep$ such that
  \[ \varphi(x)+c_\ep \ \hbox{ touches $u^\ep(x)$ from below at $x_\ep$ in $K \subset \subset U$.} \]
  Since $u^\ep$ are harmonic and $\varphi$ is strictly subharmonic the touching points $x_\ep$ must be on the free boundary $\partial \{ u^\ep >0\} \cap \partial \{\varphi +c_\ep >0\}$.  Let $k_\ep \in \ep \Z^d$ with $|k_\ep - x_\ep| \leq C\ep$ and $k_\ep-x_\ep \cdot \grad \varphi(x_\ep) >0$. Up to taking a subsequence
  \[ \ep^{-1}(k^\ep-x_\ep) \to \tau \ \hbox{ with } \ \tau \cdot \grad \varphi(x_0) >0\].
   Now we blow up at $k_\ep$, defining
  \[ v^\ep(x) = \frac{1}{\ep}u^\ep(k_\ep + \ep x) \ \hbox{ and } \ \varphi^\ep = \frac{1}{\ep}\varphi(k_\ep+\ep x).\]
  By the Lipschitz estimate, \lref{properties}, $v^\ep(x) \leq C+C|x|$ and is uniformly Lipschitz continuous.  Thus, up to a subsequence, we can take limits $v^\ep \to v$ and $\varphi^\ep \to \grad \varphi(x_0) \cdot (x+\tau)$ locally uniformly.  Then, by the stability of viscosity solutions under uniform convergence, $v$ solves in $\R^d$
  \[ \Delta v = 0 \ \hbox{ in } \ \{ v>0\}, \ \hbox{ with } \ |Dv| = Q(x) \ \hbox{ on } \ \partial \{ v >0\}, \]  
  and furthermore
  \[ v(x) \geq (\grad \varphi(x_0) \cdot (x+\tau))_+ \geq (\grad \varphi(x_0) \cdot x)_+ \ \hbox{ in } \ \R^d.\]
By \lref{visccorr} 
\[|\grad \varphi(x_0)| \leq Q^*(\tfrac{\grad\varphi(x_0)}{|\grad\varphi(x_0)|}).\]
Thus we obtain the viscosity solution condition, if $\varphi$ is a smooth test function touching $u$ from below at some point $x_0 \in \partial {\{u >0\}} \cap U$,
\[ \min \{|\grad \varphi(x_0)| -Q^*(\tfrac{\grad\varphi(x_0)}{|\grad\varphi(x_0)|}),\Delta u(x_0)\} \leq 0  .\]
  \end{proof}


\section{The continuous part of the pinning interval}\label{s.contpart}
In this short section we give an abstract definition for what we call the continuous part of the pinning interval $I_{cont}(e)$ which will be a subset of the pinning interval $I(e) = [Q_*(e),Q^*(e)]$.  The definition is basically exactly designed so that the perturbed test function argument will work when we consider the convergence of the minimal supersolutions / maximal subsolutions.  This makes the perturbed test function argument easy, the entirety of the difficulty is transferred onto proving properties of $I_{cont}$.

Recall the half-space subsets we introduced before.  Let $e \in S^{d-1}$ and $U \subset \{x \cdot e = 0\}$ relatively open and connected.  Define
\[ D_e(U) = \{ x \in \R^d : x \cdot e >0 \ \hbox{ and } \ x - (x\cdot e) e \in U\}.\]
The definitions will use domains of this type because they come up naturally in the perturbed test function argument.

\begin{definition}\label{d.contsubsuper}
Let $e \in S^{d-1}$, we say that the slope $\alpha e$ is subsolution continuously pinned if the following holds.  For all $\lambda>0$ there exists a $\delta>0$ such that if $\varphi$ smooth on $D = D_{e'}(U)$, for some $U$ a domain of $\{x \cdot e' = 0\}$,
\[ \sup_{D \cap \{ \varphi>0\}}\left|\frac{\grad\varphi}{ |\grad \varphi|} -e\right| +|e'-e|\leq \delta,\]
  $\varphi$ is harmonic in $\{\varphi>0\} \cap D(U)$, and is a subsolution of
\[  |\grad \varphi| \geq (1+\lambda)\alpha \ \hbox{ on } \ \partial \{ \varphi >0\} \cap D,\]  
then for all $\ep>0$ there exists a subsolution $v^\ep$ of \eref{fb0} in $D$ such that
\[ \lim_{\ep \to 0} \sup_{\partial D} |v^\ep - v| = 0, \quad \lim_{\ep \to 0} \inf_{D} (v^\ep - v)  \geq 0, \]
and
\[\lim_{\ep \to 0} d_H(\{v^\ep>0\} \cap \partial D,\{\varphi>0\} \cap \partial D) = 0, \quad \lim_{\ep \to 0} d_H(\{v^\ep>0\} \cap D,\{v^\ep\vee\varphi>0\} \cap D) = 0.\]
Supersolution continuously pinned is define analogously with inequalities reversed where necessary.
\end{definition}

We call $I_{cont}(e)$ to be the set of slopes $\alpha e$ which are both subsolution and supersolution continuously pinned.  As we will see below $I_{cont}(e)$ is actually an interval.  The parameter $\lambda>0$ in the above definition is necessary to make sure that $I_{cont}$ closed and nonempty.  It turns out that on the interior of $I_{cont}$ a stronger condition holds, basically it is \dref{contsubsuper} without the parameter $\lambda>0$. We write that out here.

\begin{definition}\label{d.contstrictsubsuper}
Let $e \in S^{d-1}$, we say that the slope $\alpha e$ is strongly subsolution continuously pinned if the following holds.  There exists a $\delta>0$ such that if $\varphi$ smooth on $D = D_{e'}(U)$, for some $U$ a domain of $\{x \cdot e' = 0\}$,
\[ \sup_{D \cap \{ \varphi>0\}}\left|\frac{\grad\varphi}{ |\grad \varphi|} -e\right| +|e'-e|\leq \delta,\]
  $\varphi$ is harmonic in $\{\varphi>0\} \cap D(U)$, and is a subsolution of
\[  |\grad \varphi| >\alpha \ \hbox{ on } \ \partial \{ \varphi >0\} \cap D,\]  
then for all $\ep>0$ there exists a subsolution $v^\ep$ of \eref{fb0} in $D$ such that
\[ \lim_{\ep \to 0} \sup_{\partial D} |v^\ep - v| = 0, \quad \lim_{\ep \to 0} \inf_{D} (v^\ep - v)  \geq 0, \]
and
\[\lim_{\ep \to 0} d_H(\{v^\ep>0\} \cap \partial D,\{\varphi>0\} \cap \partial D) = 0, \quad \lim_{\ep \to 0} d_H(\{v^\ep>0\} \cap D,\{v^\ep\vee\varphi>0\} \cap D) = 0.\]
Strongly supersolution continuously pinned is define analogously with inequalities reversed where necessary.
\end{definition}

We give a result collecting some easy consequences of the definitions, plus a more difficult result, but it is one we have already proven above in \pref{Qcontenergy}.
\begin{lemma}\label{l.Qcont1}
Let $e \in S^{d-1}$.  
\begin{enumerate}[label=(\roman*)]
\item The set of subsolution continuously pinned slopes at direction $e$ is an interval $[Q_{*,cont}(e),\infty)$. The interior values are strongly subsolution continuously pinned. 
\item The set of supersolution continuously pinned slopes at direction $e$ is an interval $[0,Q^*_{cont}(e)]$. The interior values are strongly supersolution continuously pinned.
\item The endpoints $Q_{*,cont}(e) \leq Q^*_{cont}(e)$ are, respectively, upper semicontinuous and lower semicontinuous as functions on $S^{d-1}$.
\item\label{part.Qcontenergy} The energy minimizing slope is both subsolution and supersolution continuously pinned
\[ Q_{*,cont}(e) \leq \langle Q^2 \rangle^{1/2} \leq Q^*_{cont}(e).\]
\end{enumerate}
\end{lemma}

\begin{definition}\label{d.contpinned}
Let $e \in S^{d-1}$, we say that the slope $\alpha e$ is continuously pinned if it is both subsolution and supersolution continuously pinned, i.e. $\alpha \in I_{cont}(e)=[Q_{*,cont}(e),Q^*_{cont}(e)]$.
\end{definition}

We reiterate that the definition is designed to be exactly what we need to prove \tref{main}.  The difficulty is then transferred to showing nice properties of $I_{cont}(e)$. The strongest possible result we could expect to prove about $Q_{*,cont}$ and $Q^*_{cont}$ is that
\[ Q_{*,cont}(e) = \limsup_{e'\to e} Q_*(e) \ \hbox{ and } \ Q^*_{cont}(e) = \liminf_{e'\to e}Q^*(e).\]
 In $d=2$ we will make significant steps in this direction, see below in \sref{irr}.  We will prove that the above hold at irrational directions, and hold approximately at rational directions with large modulus.  In $d \geq 3$ the best estimate we will obtain is the one above in \lref{Qcont1} part~\partref{Qcontenergy}.  To really handle $d \geq 3$ we expect it would be necessary to refine \dref{contsubsuper} significantly to keep information about the range of $\grad f$, i.e. if it is faceted, lying in a certain rational subspace.

\begin{proof}[Proof of \lref{Qcont1}]
The conditions \dref{contsubsuper} are monotone.  If $\alpha e$ is subsolution continuously pinned then $s\alpha e$ is subsolution continuously pinned for $s >1$.  This is because if $\varphi$ is a subsolution as in \dref{contsubsuper} then $s\varphi$ is as well. One can argue analogously for supersolutions with $s <1$.

Suppose $\alpha e$ is subsolution continuous pinned then let $\alpha'>\alpha$.  Suppose that $\varphi$ is a strict subsolution with the free boundary condition $|\grad \varphi| > \alpha'$.  Then also $|\grad\varphi| \geq (1+\lambda)\alpha$ with $\lambda = \frac{\alpha'}{\alpha}-1$.  Then apply \dref{contsubsuper} using that $\alpha e$ is subsolution continuously pinned, there is a $\delta>0$ depending on $\frac{\alpha'}{\alpha}-1$ so that \dref{contstrictsubsuper} holds.

The conditions \dref{contsubsuper} are closed.  Suppose $\alpha'$ is subsolution continuous pinned for every $\alpha'>\alpha$, by the above $\alpha'$ is strongly subsolution continuously pinned. Let $\lambda>0$ and choose $\alpha' = (1+\lambda)\alpha$, there is $\delta>0$ so that .  

We prove the upper semi-continuity of 
\[ Q_{*,cont}(e) = \inf \{ \alpha : \hbox{ $\alpha e$ is subsolution continuously pinned}\}. \]
Suppose that $Q_{*,cont}(e) < \alpha' <\alpha$ so $\alpha'$ and $\alpha$ are strongly subsolution continuously pinned.  Let $\delta>0$ from \dref{contstrictsubsuper} for $\alpha'e$, suppose that $D = D_{e''}(U)$, $e',e'' \in S^{d-1}$ with $|e'-e| \leq \delta/3$, $|e''-e'| \leq \delta/3$, and $\varphi$ is a subsolution with
\[ |\grad \varphi| \geq \alpha' \ \hbox{ on } \ \partial \{\varphi>0\} \cap D \ \hbox{ and } \ \sup_{D \cap \{\varphi >0\}}|\frac{\grad \varphi}{|\grad \varphi|} - e'| \leq \delta/3.\]
Then there exists a sequence of subsolutions $v^\ep$ to \eref{fb0} converging to $\varphi$ in $D$ in the sense of \dref{contstrictsubsuper}.  Thus $Q_{*,cont}(e) \leq \alpha' < \alpha$, and so $\{ Q_{*,cont} < \alpha\}$ is open for every $\alpha$.

Finally part \partref{Qcontenergy} was already proven in \pref{Qcontenergy}.
\end{proof}

\section{Irrational directions}\label{s.irr}
In this section we consider plane-like solutions at irrational directions, $e$ not parallel to any lattice vector in $\Z^d \setminus \{0\}$.  The main result of this section is the continuity of $Q_*,Q^*$ at irrational directions in $d=2$, \tref{mainQ} part \partref{mainQcont}, which we repeat here.  
\begin{theorem}\label{t.irrcont}
When $d=2$ the upper and lower endpoints of the pinning interval, $Q^*$ and $Q_*$ respectively, are continuous at at irrational directions $e \in S^{1}\setminus \R\Z^2$.
\end{theorem}

By the same techniques we are also able to derive information on $Q_{*,cont}$ and $Q^{*,cont}$ at irrational directions and rational directions with large modulus.

\begin{lemma}\label{l.irrQcont}
Let $d=2$.
\begin{enumerate}[label=(\roman*)]
\item Let $\xi \in \Z^2 \setminus \{0\}$ irreducible.  Then 
\[ Q_{*,cont}(\xi) \leq Q_*(\xi) + C|\xi|^{-1/2} \ \hbox{ and } \  Q^*_{cont}(\xi) \geq Q^*(\xi) - C|\xi|^{-1/2}.\]
\item Let $e \in S^{1}\setminus \R\Z^2$ irrational.  Then
\[ Q_{*,cont}(\xi) = Q_*(e) \ \hbox{ and } \ Q^*_{cont}(e) = Q^*(e).\]
\end{enumerate}
\end{lemma}

We make some remarks. The result of \tref{irrcont} cannot be true as stated in $d \geq 3$.  As we have seen in \sref{examples} it is possible for $Q^*,Q_*$ to be discontinuous at some irrational directions when $d\geq 3$.  In the author's previous work with Smart \cite{FeldmanSmart} we studied the scaling of a discrete free boundary problem with a similar structure, in that case $Q^*$, $Q_*$ are only continuous at the totally irrational directions, those satisfying no rational relations $\xi \cdot e =0$ for some $\xi \in \Z^d \setminus \{0\}$.  In $d \geq 3$ there are irrational directions which satisfy some nontrivial rational relations.  A similar structure here is plausible, and, as we have shown in \sref{examples}, discontinuities of any co-dimension are possible in this problem as well.

We divide the proof into several parts.  First, in \sref{irrfoliation}, will be the construction of a foliation of $\R^2 \times (0,\infty)$ by the graphs of plane-like solutions.  This is not quite possible, in general the foliation may have gaps, the main result is that we still recover a weak type of continuity for the foliation.  Next, in \sref{bending}, we will introduce a method for bending solutions of the free boundary problem while still maintaining, approximately, the sub or supersolution property.  This is based on a nice family of perturbations suited to the problem which were introduced by Caffarelli~\cite{CaffarelliFBreg}.  Then, in \sref{irrcurved}, we sew the plane-like solutions of the foliation together using the bending perturbations to create approximate plane-like solutions at nearby directions, to show the continuity of $Q_*,Q^*$, the same method is used to show \lref{irrQcont}.

\subsection{A family of plane-like solutions sweeping out $\R^d$}\label{s.irrfoliation}

The main tool in the proof will be a monotone one-parameter family of global plane-like solutions $v_{s}(x)$ with slope $p = \alpha e$ for $\alpha \in [Q_*(p),Q^*(p)]$, $s \in S$ for some closed index set $S$. In the irrational case $S = \R$. In the rational case, $p = \xi/|\xi|$ for some $\xi \in \Z^d \setminus \{0\}$ irreducible, $S$ is $1/|\xi|$-periodic on $\R$.  The graphs of the family $v_s(x)$ will be, approximately, a foliation of $\R^2 \times (0,\infty)$.

More precisely, we claim there exists a family with the following properties.  Let $p \in S^{d-1}$ and $\alpha \in [Q_*(p),Q^*(p)]$.
\begin{enumerate}[label=(\roman*)]
\item $v_{s} : \R^d \to [0,\infty)$ defined for $s \in S$, $S$ is a closed subset of $\R$ which is $1/|\xi|$ periodic if $p = \xi/|\xi|$ for an irreducible lattice vector $\xi \in \Z^d \setminus \{0\}$, or $S= \R$ if $p$ is irrational.
\item For every $s \in S$, $v_{s}$ solves
\begin{equation}\label{e.cells}
\left\{
\begin{array}{ll}
 \Delta v_{s} = 0 & \hbox{ in } \ \{v_{s} >0\} \vspace{1.5mm}\\
|\grad v_{s}| = Q(x) & \hbox{ on } \ \partial \{v_{s} >0\} \vspace{1.5mm}\\
(\alpha p \cdot x + s+C)_+ \prec v_{s}(x) \prec   (\alpha p \cdot x + s+C)_+
\end{array}
\right.
\end{equation}
for a universal constant $C$.
\item The family $v_{s}$ is monotone increasing in $s$ and continuous in the following sense. For all $\delta>0$ there exists $r(\delta)\geq1$ so that for $0 \leq \sigma \leq \delta$, any interval $I \subset \R$ of length at least $r$, and any $\ell \geq 1$
\[ \inf_{y' \in I}\sup_{|t| \leq \ell  }  [v_{s+\sigma} - v_{s}](y'p^\perp+tp) \leq C\ell\delta. \]
Note that it is possible $S \cap [s,s+\sigma) = \{s\}$ in which case the statement is trivial.
\end{enumerate}

\begin{proposition}\label{p.irrfamily}
For any $p \in S^{d-1}$ and $\alpha \in [Q_*(p),Q^*(p)]$ there exists a family of solutions $v_{s}$ of \eref{cells} as above.
\end{proposition}

\begin{proof}[Proof of \pref{irrfamily}]
Let $v$ be the solution of \eref{cell} constructed in \lref{minimalsuper}. Call $T = \{k \cdot p: k \in \Z^2\}$, then for $p$ irrational $T$ is dense in $\R$.  For $p = \hat\xi$ rational with $\xi \in \Z^d \setminus \{0\}$ irreducible $T$ is a $1/|\xi|$-periodic discrete subset of $\R$.  

Define, for $s \in T$,
\[ v_s(x) = v(x+k) \ \hbox{ for the $k$ such that } \ p \cdot k = s.\]
In the rational case $p \cdot k = s$ does not uniquely specify $k$, but, by periodicity, it does uniquely specify $v(x+k)$. By \lref{minimalsuper} $v_s$ is monotone increasing in $s$. 

When $p$ is irrational extend $v_s$ to $s \in S = \overline{T} = \R$ by left limits, i.e. define
\[ v_s(x) = \lim_{T \ni s' \nearrow s} v_{s'}(x).\]
The limit exists by monotonicity arguments. By the Lipschitz bound on $v$ the limit is actually locally uniform in $\R^d$.  By the stability of the viscosity solution property under uniform convergence $v_s$ solve \eref{cells}. Now $v_s$, so defined, is continuous in $s$ with respect to locally uniform convergence, except for at most countably many $s \in \R$.    Note that if $\{v = 0\}$ is not connected then $v_s$ would necessarily have jump discontinuities in $s$.  We expect, although it is not proven, that this is possible for the minimal supersolution when $Q$ has strong and localized de-pinning regions.

Consider
\[ V_s(x) = v_s(x) - (p\cdot x + s)_+.\]
These are bounded uniformly in $s$.  Let $\delta >0$ and $k\in \Z^d \setminus \{0\}$ such that $  \delta/2 \geq   |k \cdot p| >0$ small.  Call 
\[ A(\delta) = \inf\{ |k|: k \in \Z^2 \ \hbox{ with } \ |k\cdot p|\leq \delta/2\}.\]  
 When $p$ is irrational there is guaranteed to be such a $k$ as long as $\delta \geq 2/|\xi|$ and in that case $A(2/|\xi|) \leq |\xi|$.

  Consider a rectangle with axes parallel to the $p$ and $p^\perp$ directions
\[ \Box_{\ell,r} = \{ y: |y \cdot p| \leq \ell/2, \ |y \cdot p^\perp| \leq r/2\}\]
 and corresponding translations $\Box_{\ell,r}(x)$.  Let $\delta>0$, or $\delta \geq 1/|\xi|$ if $p$ is rational, and $k$ such that $|k \cdot p| \leq \delta$, and $|k| = A(\delta)$.  Note that
 \[ |\Box_{\ell,r}(x) \Delta \Box_{\ell,r}(x-k)| \leq |k \cdot p| r + |k \cdot p^\perp| \ell.\]
 Then, using the boundedness of $V_s$,
\begin{equation}\label{e.Vdiff}
 \left|\frac{1}{\ell r}\int_{\Box_{\ell,r}(x)}[V_s(y+k) - V_s(y)] dy \right| \leq C\frac{1}{r\ell}(|k \cdot p| r + |k \cdot p^\perp| \ell)  \leq C(\frac{\delta}{\ell} + \frac{A(\delta)}{r}).
 \end{equation}
 Call $t = y \cdot p$ and $y' = y \cdot p$ to be the coordinates in the $p,p^\perp$ basis and then
\[ \frac{1}{r\ell}\int_{\Box_{\ell,r}(x)}[V_s(y+k) - V_s(y)] dy  = \frac{1}{r}\int_{x \cdot p^\perp-r/2}^{x \cdot p^\perp +r/2}\frac{1}{\ell}\int_{s-\ell/2}^{s + \ell/2}[V_{s+k \cdot p}(t,y') - V_s(t,y')] dt dy' \]
so there is a $y'$ with $|y' - x \cdot p^\perp| < r/2$ and
\[ \left|\int_{s-\ell/2}^{s + \ell/2}[V_{s+k \cdot p}(t,y') - V_s(t,y')] dt\right| \leq C(\delta + \frac{\ell A(\delta)}{r}). \]
Now rephrasing in terms of $v_{s+k \cdot p} - v_s$
\[ \left|\int_{s-\ell/2}^{s + \ell/2}[v_{s+k \cdot p}(t,y') - v_s(t,y')] dt\right| \leq C(\ell\delta + \frac{\ell A(\delta)}{r}). \]
Then using the Lipschitz continuity of $v_s$, and emphasizing the dependencies of the parameters on the right hand side,
\[\max_{s-\ell/2 \leq t \leq s+\ell/2} (v_{s+k \cdot p}(t,y') - v_s(t,y')) \leq C(\ell\delta + \frac{\ell A(\delta)}{r}) \leq C\ell \delta \]
as long as $r \geq r_0(\delta) =A(\delta)/\delta$.

\end{proof}

\subsection{Bending the free boundary}\label{s.bending}  Before we proceed with the proof of \tref{irrcont} we need a technical tool.  In order to construct sub and supersolutions at nearby directions out of the family $v_s$ we will need to bend the free boundary while approximately maintaining the solution property.

A suitable family of perturbations has been constructed already by Caffarelli~\cite{CaffarelliFBreg}, the book of Caffarelli-Salsa~\cite[Lemma 4.7 and Lemma 4.10]{CaffarelliSalsa} is a convenient reference.  We recall the main points here.
\begin{lemma}[Lemma 4.7, Caffarelli-Salsa~\cite{CaffarelliSalsa}]\label{l.caffperturb1}
Let $\varphi$ be a $C^2$ positive function satisfying
\[ \Delta \varphi \geq \frac{(d-1)|\grad \varphi|^2}{\varphi}  \ \hbox{ in } \ B_1.\]
  Let $u$ be continuous, defined in a domain $\Omega$ sufficient large so that
\[ w(x) = \sup_{|\sigma| \leq 1} u(x + \varphi(x)\sigma)\]
is well defined in $B_1$.  Then if $u$ is harmonic in $\{ u>0\}$, $w$ is subharmonic in $\{w >0\}$.
\end{lemma}

We consider applying the above type of perturbation to one of the plane-like solutions $v$ with slope $p$, defining
\[ v^\varphi = \sup_{|\sigma| \leq 1} v(x + \varphi(x)\sigma)\]
By the previous Lemma, as long as $\varphi$ is defined and satisfies the condition $\varphi \Delta\varphi \geq |\grad \varphi|^2$ in a sufficiently large neighborhood of $\{ v>0\}$ we will have $v^\varphi$ subharmonic in $\{v^\varphi >0\}$.  The following Lemma explains how the perturbation affects the free boundary condition.
\begin{lemma}\label{l.perturbsubcond}
Let $v$ and $\varphi$ as above, then $v^\varphi$ satisfies, in the viscosity sense,
\[ |\grad v^\varphi(x)| \geq (1-|\grad \varphi(x)|)\inf_{B_{\varphi(x)}(x)}Q  \ \hbox{ on } \ \partial \{v^\varphi >0\}.\]
\end{lemma}
This is a minor modification of Lemma 4.9, 4.10 from \cite{CaffarelliSalsa}.  An analogous supersolution condition holds for the corresponding inf-convolution.

Now this bending procedure will cause a strict increase in $v$ near the free boundary, due to nondegeneracy, and far from the boundary due to the linear growth.  In the intermediate region there may be degeneracy, we deal with this by doing a ``harmonic lift".  As in the proof of \lref{normalbd}
\[ |\grad v(x) - p | \leq \frac{C}{C+(x\cdot p)_+}.  \]
In particular for $R_0>1$ universal there is a universal lower bound on the gradient 
\begin{equation}\label{e.gradlowerbd}
|\grad v(x)| \geq |p|/2 \geq c \ \hbox{ for } \ x \cdot p \geq R_0.
\end{equation}
  The we define the lift $\overline{v}^\varphi$ by solving
\[ \begin{cases}
\Delta \overline{v}^\varphi = 0 & \hbox{in } \ \{v^\varphi >0\} \cap \{x \cdot p < R_0\} \\
\overline{v}^\varphi = 0 & \hbox{on } \partial \{v^\varphi >0\} \\
\overline{v}^\varphi = v^\varphi & \hbox{on } x \cdot p \geq  R_0
\end{cases}
 \]
   Since $v^\varphi$ was a subsolution $\overline{v}^\varphi \geq v^\varphi$ and is still a subsolution of the condition in \lref{perturbsubcond}.  As we will make precise later, if $\varphi$ is small then $v^\varphi$ is close to $v$ and also $\overline{v}^\varphi$ is close to $v$.  

We make more precise the choice of $\varphi$.  Note that a positive $\varphi$ is a solution of
\begin{equation}\label{e.phieqn}
 \varphi \Delta \varphi = (d-1)|\grad \varphi|^2
 \end{equation}
if and only if $\varphi^{2-d}$ is harmonic, or $\log \varphi$ harmonic in $d=2$ (as is the case for us).  This property is preserved by dilation and scalar multiplication.  We proceed in the case $d=2$, but all of this works with minor modification in $d \geq 3$ as well.

Let $M$ to be chosen (will be universal) and $h: \R \to [1,M]$ be smooth, even, radially decreasing, $h(t) = M$ for $|t| \leq 1/3$, $h(t) = 1$ for $|t| \geq 2/3$ and $|\grad h| \leq CM$.  Let $\psi$ be the solution of 
\[ \left\{
\begin{array}{ll}
\Delta \psi = 0 & \hbox{ in } \{ x \cdot p >0\} \vspace{1.5mm}\\
\psi(x) = \log [h(x\cdot p^\perp)] & \hbox{ on } \{ x \cdot p =0\},
\end{array}\right.
\]
there is a unique bounded solution of the above problem with $0 \leq \psi \leq \log 2$.  Furthermore, by the continuity up to the boundary of solution of the Dirichlet problem, for any $0 <\beta <1$
\[ |\psi(x) - \log [h(x\cdot p^\perp)]| \leq C[\log(h)]_{C^\beta}(x \cdot p)^\beta \]
The estimate could be improved for $|x \cdot p^\perp| \gg 1$, but we will only care about the behavior of $\psi$ in the strip $-1 \leq x \cdot p^\perp \leq 1$ and for $x\cdot p \ll 1$.  The quantity $[\log(h)]_{C^\beta}$ is universal.

Now we define,
\begin{equation}\label{e.bendingphi}
 \varphi_1(x) =  \exp(\psi(x)).
 \end{equation}
Then $1 < \varphi_1 < M$ in $x \cdot p >0$ and
\[
 \left|\log\frac{\varphi_1(x)}{ h(x\cdot p^\perp)}\right| \leq C(x \cdot p)^\beta
\]
 Thus for some $c>0$ universal
\begin{equation} \label{e.phibdrycont}
h(x)/2 \leq \varphi_1(x) \leq 2h(x) \ \hbox{ for } \ 0 \leq x \cdot p \leq c.
\end{equation}

Next we take the sup-convolution of a plane-like solution $v$ by a rescaling of $\varphi_1$, $\varphi  =\ep \varphi_1(\cdot/r)$ with $\ep>0$ small and $r>1$ large.  Due to the nondegeneracy of $v$ the sup convolution causes a strict increase of order $\sim \varphi$.  This is expressed in the following Lemma. 

\begin{lemma}\label{l.bendingphi}
Let $\varphi = \ep\varphi_1(\cdot/r)$ and $v$ a solution of \eref{cell}.  If $r \geq CM$ then
\[  c\varphi(x)\leq \overline{v}^\varphi(x) -v(x)\leq  C\varphi(x)\]
with constants $c,C$ universal (in particular independent of $M$).  The right inequality holds everywhere, the left holds for $x$ such that $d(x ,\partial \{ v >0\} ) \leq \varphi(x)/2$.
\end{lemma}
\begin{proof}
  By the nondegeneracy \lref{nondegen}
\begin{equation}\label{e.bdryphibd}
 \overline{v}^\varphi(x) \geq v^\varphi(x)  \geq v(x)+  c \varphi(x) \ \hbox{ for $x$ s.t. } \   \ d(x ,\partial \{ v >0\} ) \leq \varphi(x)/2
 \end{equation}
By \eref{gradlowerbd}, for $x\cdot p \geq C$ universal $|\grad v(x)| \geq |p|/2 \geq c$ universal and so
\begin{equation}\label{e.intrphibd}
 \overline{v}^\varphi(x) = v^\varphi(x) \geq v(x) + c\varphi(x) \ \hbox{ on } \ x \cdot p \geq R_0.
 \end{equation}
Thus by maximum principle, combining \eref{bdryphibd}, \eref{intrphibd}, the subharmonicity of $\varphi$ \eref{phieqn}, and the harmonicity of the lift $\overline{v}^\varphi$ in $\{v^\varphi>0\} \cap \{ x \cdot p <R_0\}$,
\[  \overline{v}^\varphi(x) \geq v(x) + c\varphi(x) \ \hbox{ for all $x$ s.t.  } \ d(x,\{ v>0\}) \leq \varphi(x)/2\]
This gives one direction of the estimate.

On the other hand, by the Lipschitz estimate \lref{properties}, 
\[ v^\varphi(x) \leq v(x) + C\varphi(x).\]
In $x\cdot p \geq R_0$ this is the same for $\overline{v}^\varphi$.  Then, using again the equation for $\varphi$ \eref{phieqn}, and $|\grad \varphi|^2/\varphi \leq CM^2\ep/r^2$, by maximum principle in the strip $\{v^\varphi > 0 \} \cap \{ x \cdot p < R_0\}$,
\[ \overline{v}^\varphi(x) \leq v(x) + C\varphi(x) + CR_0^2M^2\ep/r^2. \]
  Then we can choose $r$ sufficiently large in order that $CR_0^2M^2/r^2 \leq 1$ and so
\[ \overline{v}^\varphi(x) \leq v(x)+(C+1)\varphi(x). \]
\end{proof}

\subsection{Curved surface near an irrational direction}\label{s.irrcurved}

With the set-up above we finally are able to carry out the proof of \tref{irrcont}.  We prove \lref{irrQcont} at the same time since the proof is the same.

\begin{proof}[Proof of \tref{irrcont} and \lref{irrQcont}]
We just do the subsolution case, the supersolution case is similar.  We argue for rational and irrational directions at once.  In the rational case suppose that $e = \hat\xi$ for some irreducible $\xi \in \Z^d \setminus \{0\}$.  We will use $|\xi|^{-1}$ as a parameter, in the irrational case we abuse notation and say $|\xi|^{-1} = 0$. 

Let $\lambda>C|\xi|^{-1/2}$ and suppose that $\psi$ smooth on $D = D_{e'}(U)$, for some $U$ a domain of $\{x \cdot e' = 0\}$,
\[ \sup_{D \cap \{ \psi>0\}}\left|\frac{\grad\psi}{ |\grad \psi|} -e\right| +|e'-e|\leq \eta_0,\]
  $\psi$ is harmonic in $\{\psi>0\} \cap D(U)$, and is a subsolution of
\[  |\grad \psi| \geq (1+\lambda)Q_*(e) \ \hbox{ on } \ \partial \{ \psi >0\} \cap D.\] 
The parameter $\eta_0$ will be chosen small below depending on $e$ and $\lambda$.  Write the free boundary $\partial \{ L\psi(\cdot/L)>0\}$ as a graph over $x \cdot e = 0$ by 
\[ \tau \mapsto x_\tau = \tau e^\perp +Lf(\tau/L)e \ \hbox{ for } \ \tau \in U. \]  
Then $f$ is $C^1$ and $\|f'\|_\infty \leq C\eta_0$.  In the proof of \tref{irrcont}, $D = \R^d$, $\psi$ is a half-plane solution $\psi(x) = \alpha(e' \cdot x)_+$, with $\alpha \geq (1+\lambda)Q_*(e)$ and $f(\tau) = - \tau \frac{e' \cdot e^\perp}{e' \cdot e}$.

Let $\delta(\lambda) = c_0' \lambda^2$ so that $\delta \geq 1/|\xi|$.  By \pref{irrfamily} there exists $ r_0(\lambda,e) \geq 1$ large such that, for all $s \in \R$, $0 \leq \sigma \leq \delta$ and interval $I \subset \R$ of length at least $r$,
\begin{equation}\label{e.closeline}
 \inf_{\tau \in I} \sup_{|t| \leq R_0/\lambda} \ (v_{s+\sigma} - v_s)(\tau e^\perp + te) \leq Cc_0'R_0\lambda.
 \end{equation}
The constants $c_0$ and $R_0$ will be chosen, universal, in the course of the proof, call $c_1 = Cc_0R_0$ for convenience.   Now we specify $\eta_0$, we require that $\|f'\|_{L^\infty} \leq C\eta_0 \leq \delta/3r_0$ and call $r = 3r_0$.

Let $\tau_j = j r$ for $j \in \Z$ and push forward the partition $\{\tau_j\}_{j \in \Z}$ of the domain onto the range
\[ s_j =  [Lf(\tau_j/L)]_S \] 
  Then
  \[ |s_{j+1}-s_j| \leq \|f'\|_{\infty} r \leq \delta.\]
  From \eref{closeline} for each $j$ there is $y_j,z_j$ with $|y_j - \tau_{j-1}| \leq r_0 \leq r/3$ and $|z_j - \tau_{j+1}| \leq r_0 \leq r/3$ such that
\begin{equation}\label{e.closelinesp}
 \sup_{|t| \leq 1/\lambda} \ |v_{s_{j-1}} - v_{s_j}|(y_je^\perp + te) \leq c_1\lambda \ \hbox{ and } \  \sup_{|t| \leq 1/\lambda} \ |v_{s_{j-1}} - v_{s_j}|(z_je^\perp + te) \leq c_1\lambda. 
 \end{equation}

Now use the bending sup-convolutions of \sref{bending} to create a subsolution.  With $\varphi_1$ as in \sref{bending}, let $\varphi = c_1\lambda\varphi_1(\cdot/r)$, defined as above in \eref{bendingphi} with the parameter $M$ in the definition of $\varphi_1$ still to be chosen (it will be chosen universal).  For each $j \in \Z$ define
\[\tilde{w}_{j}(x) =   \overline{v}_{s_j}^{\varphi(\cdot - x_{\tau_j})}(x). \]
Each $\tilde{w}_j$ is harmonic in its positivity set and, on $\partial \{\tilde{w}_j >0\}$,
\begin{equation}\label{e.subsoln110}
 |\grad \tilde{w}_j| \geq (1-CMc_1\lambda)(Q(x) - 2\|\grad Q\|_\infty Mc_1\lambda) \geq (1+\tfrac{\lambda}{2})^{-1}Q(x),
 \end{equation}
for $c_1(M)$ chosen sufficiently small.  We reiterate $M$ will be chosen later universal, and will not depend on $c_1$.  Localize each $\tilde{w}_{j}$ to a vertical strip near $x  \cdot e^\perp = \tau_j $
\[ w_{j}(x) = \begin{cases} (1+\frac{\lambda}{2})\tilde{w}_{j}(x) &\hbox{ if } y_j \leq  x \cdot e^\perp \leq z_j  \\
-\infty &\hbox{ else}
\end{cases}\]
which is a subsolution of \eref{unitscale} in the strip where it is finite.  Finally define, for $x \in LD$,
\begin{equation*}
 w(x) =  \max\{ \max_{j \in \Z} w_{j}(x), L\psi(\frac{x}{L}-Ke)\}
 \end{equation*}
the translation $K$, universal, will be specified below.  Although this appears to be a maximum over an infinite set, at each $x$ only three of the $w_{j}(x)$ take a finite value.    We will show that
\begin{equation}\label{e.bdrywjregion0}
  \begin{array}{c}
  \tilde{w}_j(x) < \tilde{w}_{j-1}(x) \ \hbox{ on } \ x \cdot p^\perp = y_j, \ x \cdot e \leq Lf(x \cdot e^\perp/L)+R_0/\lambda, \vspace{1.5mm}\\
  \tilde{w}_j(x) < \tilde{w}_{j+1}(x) \ \hbox{ on } \ x \cdot p^\perp = z_j, \ x \cdot e \leq Lf(x \cdot e^\perp/L)+R_0/\lambda, 
  \end{array}
 \end{equation}
and 
\begin{equation}\label{e.farbdryregion0}
  \begin{array}{ll}
 w(x) = \max_{j \in \Z} w_j(x)  & \hbox{for $x \in \partial \{w>0\} + B_1$}, \vspace{1.5mm}\\
  w(x) = L\psi(\frac{x}{L}-Ke) & \hbox{for $x \cdot e \geq Lf(x \cdot e^\perp/L)+R_0/\lambda$}. 
  \end{array}
 \end{equation}
Once these two are proven, then $w$ defined as above will be continuous subsolution of \eref{unitscale}.

First consider \eref{bdrywjregion0}.  Let $x \cdot e \leq R_0/\lambda$ with $x \cdot e^\perp = y_j$, then by \eref{closelinesp} and \lref{bendingphi}
\begin{align*}
 \tilde{w}_{j-1}(x) &\geq v_{s_{j-1}}(x) + c\varphi(x - x_{\tau_{j-1}}) \\
 &\geq v_{s_{j}}(x) -c_1\lambda+cMc_1\lambda \\
 &\geq \tilde{w}_j(x)- C\varphi(x - x_{\tau_{j-1}}) -c_1\lambda+cMc_1\lambda ,\\
 &\geq \tilde{w}_j(x) +c_1(cM-C)\lambda
 \end{align*}
while, similarly, on $x \cdot e^\perp = z_j$
\begin{align*}
\tilde{w}_{j+1}(x) &\geq v_{s_{j+1}}(x) +c\varphi(x - x_{\tau_{j+1}}) \\
&\geq v_{s_{j}}(x) - c_1\lambda + cMc_1\lambda \\
&\geq \tilde{w}_j(x) +c_1(cM-C)\lambda. 
\end{align*}
Choosing $M$ large universal so that $cM-C>0$ above, we get \eref{bdrywjregion0}.  Now we also see that the choice of $c_1$, depending on $M$ and universal quantities, is indeed universal as well.

Now we aim to show \eref{farbdryregion0}.  We assume $f(0) = 0$ and show the result in $|x \cdot e^\perp| \leq r$.  The bounds for $\tilde{w}_{-1}$, $\tilde{w}_0$ and $\tilde{w}_1$ gives
\begin{equation}\label{e.bd001}
 Q_*(e)(x \cdot e -C)_+\leq  \max_{j \in \Z} w_j(x) \leq Q_*(e)(x \cdot e +C)_+ \ \hbox{ in } \ |x \cdot e^\perp| \leq r.
 \end{equation}
Note that since $f$ is $C^1$
\[ |Lf(\tau/L) - f'(0)\tau| \leq \omega(\tau/L)\]
where $\omega$ is the modulus of continuity of $f'$.  Now let $L \geq L_0 = r^{-1} \omega^{-1}(1)$ so we have $\omega(\tau/L) \leq 1$ for $|\tau| \leq r$ and $L \geq L_0$.  Note $r = 3r_0(\lambda,e)$ so the choice of $L_0$ depends on $\lambda$, $e$, and the modulus of continuity of $f'$.  Then, since $|f'(0)\tau| \leq C\eta_0 r \leq C\delta <1$,
\[ |Lf(\tau/L)| \leq 2 \ \hbox{ for } \ |\tau| \leq r. \]
Therefore as long as $K$ chosen large enough universal,
\[ L\psi(\frac{x}{L}-Ke) = 0 \ \hbox{ in } \ -C \leq x \cdot e \leq C \]
and the first part of \eref{farbdryregion0} holds. 

Let $x \cdot e^\perp = \tau$ and $x \cdot e = R_0/\lambda$
\begin{align*}
 L\psi(\frac{x}{L}-Ke) &= \grad \psi(\frac{x_\tau}{L}) \cdot (x-x_\tau-Ke)+O(|x-x_\tau|^2/L) \\
 & = (1+\lambda)Q_*(e)(x \cdot e -Lf(\tau/L)-K)+O(\frac{1}{\lambda^2L}+\frac{\eta_0}{\lambda})
 \end{align*}
 Note that $\lambda^{-1}\eta_0 = \delta/(3r_0\lambda) = c_0\lambda/(3r_0) <1$ and, choosing $L \geq L_0(\lambda)$ larger if necessary also $\frac{1}{\lambda^2L} <1$.  Thus, using \eref{bd001},
 \[L\psi(\frac{x}{L}-Ke) \geq (1+\lambda)Q_*(e)(x \cdot e) - C \geq Q_*(e)(x \cdot e)+R_0-2K \geq  w(x)\]
 as long as $R_0$ is sufficiently large universal.  This completes the proof of \eref{farbdryregion0}.  Also we see now, since $R_0$ is universal, and $c_1 = CR_0c_0$ was chosen to be small universal, also $c_0$ can be chosen small universal to fulfill the needed requirements.
 
 Finally we take $\ep = 1/L < \ep_0 = 1/L_0$ and define
 \[ w^\ep(x) = \ep w(x/\ep).\]
From the estimates proven above $w^\ep \to \psi$ uniformly in $D$ and also 
\[d_H(\partial \{w^\ep>0\} \cap D,\partial\{\psi>0\} \cap D) \to 0\]
 as $\ep \to 0$.  This is actually stronger than \dref{contsubsuper} requires.

\end{proof}

%
\section{Rational directions}
In this section we consider more carefully the solutions of the cell problem at a rational direction.  As before we will consider a general dimension $d \geq 2$ for as long as possible, but eventually we will focus on the case $d=2$. The main reason for this restriction is the lack of nondegeneracy estimate \lref{nondegen} for maximal subsolutions in $d \geq 3$.

Let $\xi \in \Z^d \setminus \{0\}$ irreducible and we consider the cell problem \eref{cell} at direction $\xi$.  As seen in \sref{examples}, $Q^*$ and $Q_*$ may be discontinuous at $\xi$.  Define the directional limits, for $\tau \in \xi^\perp \cap S^{d-1}$
\[ Q^*_\tau(\xi) = \limsup_{e \to_\tau \hat\xi} Q^*_\tau(e) \ \hbox{ and } \ Q_{*,\tau}(\xi) = \limsup_{e \to_\tau \hat\xi} Q_{*,\tau}(e)\]
where we say a sequence $e_n \to_\tau \hat{\xi}$ to mean that $e_n \to \hat\xi$ and
\[ \frac{e_n-\hat\xi}{|e_n-\hat\xi|} = \tau \ \hbox{ for $n$ sufficiently large}.\]
The $\tau$ direction limit of the pinning interval is defined
\[ I_\tau(\xi) = [Q^*_\tau(\xi),Q_{*,\tau}(\xi)].\]

When $d=2$ there are only two directional limits, which we refer to as the left and right limit.  Recall that we take the convention $\xi^\perp = (\xi_2,-\xi_1)$, we call directions $e \cap S^{d-1}$ with $e \cdot \xi^\perp >0$ to be to the right of $\xi$, and with $e \cdot \xi^\perp < 0$ to be to the left of $\xi$.  We define the left and right limits of $Q^*,Q_*$
\begin{equation}\label{e.lrlimits}
\begin{array}{ccc}
 Q^*_\ell(\xi) = \displaystyle{\limsup_{e \to_\ell \hat\xi} Q^*(e)} & \hbox{ and } & Q^*_r(\xi) = \displaystyle{\limsup_{e \to_r \hat\xi} Q^*(e)}, \vspace{1.5mm}\\
 Q_{*,\ell}(\xi) = \displaystyle{\limsup_{e \to_\ell \hat\xi} Q_*(e)} & \hbox{ and } & Q_{*,r}(\xi) = \displaystyle{\limsup_{e \to_r \hat\xi} Q_*(e)}, 
 \end{array}
\end{equation}
and corresponding $I_\ell(\xi)$ and $I_r(\xi)$.

Speaking informally, the free boundary can bend in the $\tau$ direction when the slope $|\grad u| \in I_\tau(\grad u)$.  In this section we will make this idea rigorous at the level of the $x$-dependent problem.  


  The main result of the section is \tref{mainQ} part~\partref{mainQcont1}, that the $\limsup$'s and $\liminf$'s in \eref{lrlimits} actually exist as limits.
\begin{theorem}\label{t.lrcont}
Suppose $\xi \in \Z^2 \setminus \{0\}$ is a rational direction. Left and right limits of $Q^*,Q_*$ exist at $\xi$, i.e.
\[ \lim_{e \to_\ell \hat\xi} I(e) = I_\ell(\xi) \ \hbox{ and } \ \lim_{e \to_r \hat\xi} I(e) = I_r(\xi).\]
\end{theorem}
The restriction to $d=2$ in this theorem is only because we do not know the nondegeneracy estimate \lref{nondegen} for maximal subsolutions when $d \geq 3$.  We expect that a more general result, along the same lines, computing the limit of $I(e)$ given a sequence of approach directions would be possible using similar ideas.

As a corollary of \tref{lrcont} and \lref{irrQcont} we obtain also part \partref{mainQcontdir} of \tref{mainQcont}.
\begin{corollary}
Suppose $\xi \in \Z^2 \setminus \{0\}$ is a rational direction then left and right limits of $I_{cont}$ exist at $\xi$ and agree with the left and right limits of $I(e)$
\[ \lim_{e \to_\ell \hat\xi} I_{cont}(e) = I_\ell(\xi) \ \hbox{ and } \ \lim_{e \to_r \hat\xi} I_{cont}(e) = I_r(\xi).\]
\end{corollary}
This is just because if $e_k = \hat{\xi}_k$ converges to $\hat{\xi}$ rational with $e_k \neq \hat \xi$ then $|\xi_k| \to \infty$ necessarily.  Then, by \lref{irrQcont}, $|I_{cont}(e_k) \Delta I(e_k)| \to 0$ as $k \to \infty$.

\subsection{A family of periodic plane-like solutions with oriented connections sweeping out $\R^d$} 
The first goal is to construct a continuous family of plane-like solutions sweeping out $\R^d$, as we did in the irrational direction case.  This will be the main tool in the proof of \tref{lrcont}.  The situation here is a bit different however, as can be guessed by considering the case of laminar media.  In general the sweepout family will consist of a monotone family of $\xi^\perp \cap \Z^d$-periodic plane-like solutions, possibly with gaps, and plane-like but non-periodic heteroclinic connections which fill the gaps.  For this construction we can consider general $d \geq 2$.

\begin{definition}\label{d.sweepout}
Let $\xi \in \Z^d \setminus \{0\}$ irreducible, $\tau \in \xi^\perp \cap S^{d-1}$, and let $\alpha \in I_\tau(\xi)$. A $\tau$-oriented sweepout family of plane-like solutions consists of a closed set $S \subset \R$, the parametrization domain, which is a $|\xi|^{-1}$-periodic, and strictly monotone decreasing family of plane-like solutions $\{v_s\}_{s \in S}$ solving \eref{cell}, $\xi^\perp \cap \Z^d$-periodic and
\[  \lim_{x \cdot \hat \xi \to \infty}|v_s(x) - (\alpha x \cdot \hat{\xi} + s)_+| = 0,\]
and for each pair $s_1<s_2 \in S$ with $(s_1,s_2) \cap S = \emptyset$ there is a plane-like solution $v_{s_1,s_2}$, monotone increasing with respect to $\xi^\perp \cap \Z^d$ translations with $k \cdot \tau \geq 0$, connecting $v_{s_1}$ at $x \cdot \tau = - \infty$ to $v_{s_2}$ at $x \cdot \tau = +\infty$ in the sense that
\[ v_{s_1} \leq v_{s_1,s_2} \leq v_{s_2} \ \hbox{ on } \R^d\]
and 
\[ \lim_{\substack{m \cdot \tau \to \infty \\ m \in \xi^\perp \Z^d }} v_{s_1,s_2}(\cdot+m) = v_{s_2} \ \hbox{ and } \ \lim_{\substack{m \cdot \tau \to \infty \\ m \in \xi^\perp \Z^d }} v_{s_1,s_2}(\cdot-m) = v_{s_1}. \]
with the limits holding uniformly in $x \cdot \tau \geq r$ and $x \cdot \tau \leq r$ respectively for any fixed $r \in \R$.
\end{definition}

\begin{remark}\label{r.lroriented}
Note that if $(S,\{v_s\}_{s \in S})$ is an oriented sweepout such that $S = \R$ has no discrete part, then actually it is $\tau$-oriented for any $\tau \in \xi^\perp \cap S^{d-1}$.  However we should not expect to have this situation except in a very special case.  By considering the laminar medium $Q(x) = Q(x_1)$, one may guess that $S$ being a discrete set of $\R$ is generic.  Note that in that case, any oriented sweepout family with slope $\langle Q^2 \rangle^{1/2}e_1$ must have $S \subset \langle Q^2\rangle^{-1/2}\{x_1: Q(x_1) = \langle Q^2\rangle^{1/2}\}$ which is, generically, a discrete set.
\end{remark}
\begin{lemma}\label{l.sweepoutexistence}
For each $\alpha \in I_\tau(\xi)$ there exists a $\tau$-oriented sweepout family of plane-like solutions with slope $\alpha$.
\end{lemma}
We also point out a regularity property of the family $v_s$ in the $s$ variable.
\begin{lemma}\label{l.sweepoutreg}
Suppose that $(S,\{v_s\}_{s \in S})$ is an oriented sweepout as defined in \dref{sweepout}.  Then $v_s : S \to C(\R^d)$ is continuous in the supremum norm.
\end{lemma}
Of course this statement is only interesting when $S$ is not a discrete set.
\begin{remark}
We expect this regularity could be made quantitative (Lipschitz in $s$) with some quantitative information about the Poisson kernel in the rough, half-space-like, domain $\{v_s>0\}$.  See Kenig-Prange~\cite[Prop. 21]{KenigPrange} for the required Poisson kernel estimate when the domain is a graph.
\end{remark}

The proof of existence of this family is rather delicate, we explain some heuristic ideas.  The solutions $v_{\alpha_ne_n}$ at the nearby direction will be close to periodic solutions with slope $\alpha e$ over large regions.  However, because the direction $e_n \neq e$ the $v_{\alpha_ne_n}$ will have to leave any neighborhood of a particular solution with slope $\alpha e$.  This could occur by a heteroclinic connection, transferring over a unit length scale from a small neighborhood of one periodic solution with slope $\alpha e$ to a small neighborhood of another such periodic solution.  Another possibility is the existence of a continuous family of periodic solutions with slope $\alpha e$, then the $v_{\alpha_ne_n}$ can vary slowly (length scale $\gg 1$) between them.  Vaguely speaking we think that $v_{\alpha_ne_n}$ is built out of a monotone family of periodic solutions of slope $\alpha e$, with possible heteroclinic connections between pairs of periodic solutions when there is a gap in the monotone family.  

These heuristics motivate the basic idea of the proof, which is to take limits of lattice translations of the $v_{\alpha_ne_n}$.  This sounds extremely simple, the difficulty is that such a monotone family may not be unique, so to prove existence we need to construct a subsequence of the $v_{\alpha_ne_n}$ which is, asymptotically, built out of a single such monotone family.  Furthermore, in order to construct the heteroclinic connections we need the monotone family to be maximal in an appropriate sense.  Constructing such a maximal family is the main issue of the proof.
\begin{proof}[Prrof of \lref{sweepoutexistence}]
1. (Existence of periodic limits) First take an arbitrary sequence of plane-like solutions $w_{\alpha_ne_n}$ solving \eref{cell}. Up to a subsequence, they converge locally uniformly to a plane-like solution $w_{\alpha e}$ with asymptotic slope $\alpha e$.  Without loss we can assume that $e_n \cdot e^\perp >0$ for all $n$ ($e_n \cdot e^\perp <0$ case is similar).  Let $k$ be the lattice vector with minimal norm parallel to $e^\perp$.  Then $k \cdot e_n >0$ for all $n$ and so, by \lref{minimalsuper},
\[w_{\alpha_ne_n}(\cdot+k) \leq w_{\alpha_ne_n}(\cdot) \ \hbox{ for all } \ n.\]
Hence the same holds in the limit for $w_{\alpha e}$.  Consider the sequence $w_{\alpha e}(\cdot +mk)$ for $m \in \mathbb{N}$.  By the previous argument $w_{\alpha e}(\cdot +mk)$ is a decreasing sequence and so, taking into account the Lipschitz estimate \lref{properties}, the sequence converges locally uniformly to some $v_{\alpha e}$.  Now $k$ is a period of $v_{\alpha e}$ since $w_{\alpha e}(x+mk \pm k)$ converges to both $v_{\alpha e}(x)$ and $v_{\alpha e}(x\pm k)$ as $m \to \infty$.  

2. (A monotone family of periodic limits) Consider a sequence $w_n = w_{\alpha_ne_n}$ of plane-like solutions converging to $v_{\alpha e}$.  This sequence generates a family of solutions $v_{\alpha \hat\xi}(x+k)$ for $k \in \Z^d$.  Consider now the larger family $\mathcal{F}$ of all limits of
\[ w_{n}(x+k_n) \ \hbox{ for a sequence } \ k_n \in \Z^d\]
which are $e^\perp \cap \Z^d$-periodic.  We index the family by the boundary layer limit $s \in \R$, via \lref{bdrylayer}, which is the value such that
\[ \lim_{x \cdot  \to \infty} [v(x) - (\alpha \hat\xi \cdot x+s)_+] = 0 .\]
 The index set $S \subset \R$ is $|\xi|^{-1}$-periodic.  It is not immediately clear that the correspondence between $v \in \mathcal{F}$ and $s \in S$ is one-to-one, this will be justified below.  
 
 We claim that this family is monotone increasing, i.e. that if $s_1 \geq s_2$ then $v_{s_1} \leq v_{s_2}$ on $\R^d$.  Let $v^1,v^2 \in \mathcal{F}$, there exist corresponding sequences $k^1_n,k^2_n \in \Z^d$ such that $w_n(x+k^j_n)$ converge locally uniformly on $\R^d$ to the respective $v^j$.  Now the sequence $(k^1_n - k^2_n) \cdot e_n \in \R$ is either non-positive or non-negative infinitely often, and so either $w_n(\cdot+k^1_n) - w_n(\cdot+k^2_n)$ is either non-positive or non-negative on all of $\R^d$ infinitely often.  Thus $v^1 - v^2$ has a sign on $\R^d$.  Since the limit at $x \cdot \xi \to \infty$ is non-negative, the sign is non-negative.
 
 We also can see now that $s \mapsto v_s \in \mathcal{F}$ is single valued.  Suppose $v^1,v^2$ both have boundary layer limit $s$, in particular 
 \[\lim_{x \cdot \hat\xi \to \infty} v^1(x) - v^2(x) = 0.\]
  By the above arguments we can assume $v^1 \geq v^2$ on $\R^d$.  Suppose they are not equal, by the strong maximum principle $v^1 >v^2$ in $\overline{\{v^2 >0\}}$.  Let $t>0$ sufficiently large so that $\{x \cdot \hat\xi = t\} \subset \overline{\{v^2 >0\}}$. Then $v^1 - v^2$ is $\xi^\perp$-periodic on $x \cdot \hat\xi = t$ and strictly positive, thus there is $\delta>0$ such that $v^1 - v^2 \geq \delta >0$ on $x \cdot \hat\xi = t$.  By maximum principle also $v^1 - v^2 \geq \delta$ in the half-space $x \cdot \hat \xi \geq t$, this is a contradiction.
  
  3. (Existence of a maximal family) Now it is possible that by taking a subsequence of the $w_n$ we could enlarge $S$.   Let us show that, after taking an appropriate subsequence, this is not possible.  
  
  Consider the class of subsequences $X = \{ f : \mathbb{N} \to \mathbb{N}:  \hbox{$f$ strictly increasing}\}$ partially ordered by the relation
  \[ f \leq g \ \hbox{ if } \  f(M+\mathbb{N}) \subset g(\mathbb{N}) \ \hbox{ for $M\in \mathbb{N}$ sufficiently large}. \]
That is $f(\cdot+M)$ is a subsequence of $g$ for sufficiently large $M \in \mathbb{N}$.  Corresponding to each subsequence $f \in X$ is a monotone family $\mathfrak{m}(f) = (S,\{v_s\}_{s \in S})(f)$ given by the above construction.  Call the class of such monotone families 
  \[ \mathcal{M} = \{ \mathfrak{m}: \ \mathfrak{m} = (S,\{v_s\}_{s \in S})(f)  \ \hbox{ for some } \ f \in X\}\]
  partially ordered by the relation
  \[ \mathfrak{m}^1 \leq \mathfrak{m}^2 \ \hbox{ if $S^1 \subset S^2$ and $v^1_s = v^2_s$ on $S^1$.} \]
 Note that in the case $\mathfrak{m}^1 = \mathfrak{m}(f^1)$ and $\mathfrak{m}^2 = \mathfrak{m}(f^2)$ for some $f^1 \geq f^2$ indeed $\mathfrak{m}^1 \leq \mathfrak{m}^2$. 
 
 Actually every ordering in $\mathcal{M}$ arises in that form. Suppose that $\mathfrak{m}(f^1) \leq \mathfrak{m}(f^2)$ but there is no ordering between $f^1$ and $f^2$. We can define another subsequence $f\geq f^1,f^2$ with $\mathfrak{m}(f) = \mathfrak{m}(f^1)$.  Simply choose $f$ to count the elements of $f^1(\mathbb{N}) \cup f^2(\mathbb{N})$ in increasing order, since $f^1 \leq f$ it is clear that $\mathfrak{m}(f) \leq \mathfrak{m}(f^1)$.  For $s \in S^1$
 \[ v^1_s(x) = \lim_{n \to \infty} w_{f^1(n)}(x+k^1_{f^1(n)}) = \lim_{n \to \infty} w_{f^2(n)}(x+k^2_{f^2(n)}) \]
 for some sequences of lattice vectors $k^1_j,k^2_j \in \Z^d$ defined on $f^1(\mathbb{N}),f^2(\mathbb{N})$ respectively.  Then define $k_j$ on $f(\mathbb{N})$ as $k^1_j$ on $f^1(\mathbb{N})$ and $k^2_j$ on $f^2(\mathbb{N})$.  Then $v^1_s = \lim_{n \to \infty} w_{f(n)}(x+k_{f(n)})$ and so $\mathfrak{m}(f^1) \leq \mathfrak{m}(f)$.

 Suppose that $\mathcal{N} \subset \mathcal{M}$ is a totally ordered family. Let $S_\infty = \cup_{(S,\{v_s\}) \in \mathcal{N}} S$. For $s \in S_\infty$ we have $s \in S$ for some $\mathfrak{m} = (S,\{v_s\}_{s \in S})$ with $\mathfrak{m} \in \mathcal{N}$, define $v^\infty_s = v_s$.  Note that this definition is consistent, if $s \in S \cap S'$ for some $\mathfrak{m} = (S,\{v_s\}_{s \in S})$ and $\mathfrak{m}' = (S',\{v'_s\}_{s \in S'})$ with $\mathfrak{m},\mathfrak{m}' \in \mathcal{N}$, then by total ordering without loss $\mathfrak{m} \leq \mathfrak{m}'$.  By the above $v_s = v'_s$ for $s \in S \cap S' = S$.  Now consider the family
 \[ \mathfrak{m}^\infty = (S_\infty,\{v_s\}_{s \in S_{\infty}})\]  
which is a natural candidate for an upper bound on $\mathcal{N}$, however we need to check that it actually arises as $\mathfrak{m}(f)$ for an appropriate subsequence $f$.  Actually we will show that $\mathfrak{m}^\infty \leq \mathfrak{m}(f)$ for some subsequence $f$.

Since $S_\infty$ is a subset of $\R$ it is separable, call $S_\infty' \subset S_\infty$ to be a countable dense subset.  There is a countable collection $f^j$ of subsequences with $\mathfrak{m}(f^j) \in \mathcal{N}$ so that the union of $S(f^j)$ contains $S_\infty'$ and, by the total order, $\mathfrak{m}(f^j) \leq \mathfrak{m}(f^{j+1}) $ for all $j$.  By the arguments of the second paragraph above, we can also ensure that $f^j \leq f^{j+1}$ for all $j$, up to a replacement of the sequences which does not change the values $\mathfrak{m}(f^j)$. Taking a diagonal subsequence, $f(n) = f^n(n)$, we find an $f$ such that $S_\infty' \subset S(f)$ and $v_s^\infty = v_s(f)$ for $s \in S_\infty'$.  

Now we claim that actually $S_\infty \subset S(f)$ and $v_s^\infty = v_s(f)$ for $s \in S_\infty$.  Let $s \in S_\infty$, there is a sequence $s_j \in S_\infty'$ converging to $S$ and corresponding sequences of lattice vectors $k^j_n$ such that
\[ w_{f(n)}(x+ k^j_n) \to v_{s_j}^\infty(x) \ \hbox{ as $n \to \infty$ in $\R^d$.}\]
Then, by a basic analysis argument, we can choose a $g \in X$ so that 
\[ w_{f(n)}(x+k^{g(n)}_n) \to v_s^\infty(x)\ \hbox{ as $n \to \infty$ in $\R^d$.}\]
We have proven that $\mathfrak{m}(f) \in \mathcal{M}$ is an upper bound for $\mathcal{N}$.

Since every totally ordered family in $\mathcal{M}$ has an upper bound in $\mathcal{M}$, by Zorn's Lemma there is a maximal element in $\mathcal{M}$.  That is, there is a sequence $w_n$ (a subsequence of the original $w_n$) such that the monotone family $(S,\{v_s\}_{s \in S})$ of limits of lattice translations associated with the sequence $w_n$ cannot be enlarged by taking a subsequence of $w_n$.

  4. (Existence of left-right connections in the maximal family) For the final part of the proof we will work with a monotone family of plane-like solutions $(S,\{v_s\}_{s \in S})$ as constructed in part 2 above, which is maximal in the sense that applying the same construction to any subsequence of the $w_n$ cannot enlarge the monotone family.  
  
  Let $s_1, s_2 \in S$ with $s_1 < s_2$ and $S \cap (s_1,s_2) = \emptyset$.  By $\xi^\perp \cap \Z^d$-periodicity and strong maximum principle $v_{s_1} < v_{s_2} - \delta$ for some $\delta>0$ in $\{v_{s_1} >0\}$.  Since $e_n \cdot \tau >0$, for $x \cdot \tau$ sufficiently large and positive $w_n(x) > v_{s_2}(x)$, while for $x \cdot \tau$ sufficiently large and negative $w_n(x) < v_{s_1}(x)$.  Since $\{v_{s_1}>0\}$ is connected, and $\{x \cdot \hat\xi \geq s_1+C\} \subset \{v_{s_1} >0\}$ for sufficiently large universal $C$ also $\{ v_{s_1} >0\} \cap \{x \cdot \hat\xi \leq s_1 + C\}$ is connected.  Thus, by continuity and the previous connectedness, there is $x_n \in \{ v_{s_1} >0\} \cap \{x \cdot \hat\xi \leq s_1 + C\}$ such that 
  \[ v_{s_1}(x_n) + \delta/2 \leq w_n(x_n) \leq v_{s_2}(x_n) - \delta/2. \]
  Let $k_n\in \Z^d$ be a closest lattice point to $x_n$.  
  
  Now consider the sequence $w_n(x+k_n)$, taking a subsequence if necessary the $w_n(x+k_n)$ converge locally uniformly on $\R^d$ to some $v$.  The monotonicity 
  \[ v(\cdot + m) \geq v(\cdot) \ \hbox{ for any $m \perp \xi$ with $m \cdot \tau \geq 0$}\]
   holds for $v$ since the same monotonicity holds $w_n$.  This is because 
   \[m \cdot e_n = m \cdot (e_n - \hat\xi) = |e_n-\hat\xi| \tau\cdot m.\]
     Suppose that $m_j \perp \xi$ is a sequence with $m_j \nearrow +\infty$.  Then $v(\cdot + m_j)$ and $v(\cdot - m_j)$ are respectively monotone increasing and monotone decreasing in $j$ and therefore the sequences converge locally uniformly on $\R^d$ to respective limits $v_+$ and $v_-$.  Actually, by the monotonicity property, 
     \[ \lim_{\substack{m \cdot \tau \to \infty \\ m \in \xi^\perp \Z^d } } v(x +m) = v_+(x) \ \hbox{ and } \ \lim_{\substack{m \cdot \tau \to -\infty \\ m \in \xi^\perp \Z^d }} v(x -m) = v_+(x)\]
     with the limits hold uniformly on any set of the form $x \cdot \tau \geq r$ or $x \cdot \tau \leq r$ (respectively).  By the arguments in part 1 above the limits $v_+ \geq v_-$ are $\xi^\perp \cap \Z^d$-periodic solutions.  We claim that the respective limits are actually
  \[v_+ = v_{s_2} \ \hbox{ and } \ v_- = v_{s_1}.\]
The arguments for both are basically the same, so we just consider the first limit above.

First we point out that $v_+ = v_{s_+}$ for some $s_+ \in S$, this is the key place where we need the maximality property.  Otherwise we could choose a subsequence of the $w_n$ and a sequence of lattice vectors $\ell_n$ such that $w_n(\cdot+\ell_n)$ converges to $v_+$, but this contradicts the maximality property of $w_n$.  Recall that $ v_{s_1}(x) < v(x) < v_{s_2}(x)$ at some point within distance $\sqrt{d}/2$ of the origin, and $v_+ \geq v$. Thus $v_+ >v_{s_1}$ at some point, and hence everywhere by monotonicity of the family, and so $s_+ >s_1$.  Since $S \cap (s_1 ,s_2) = \emptyset$ then $s_+ \geq s_2$ and $v_+ \geq v_{s_2}$.  By a similar argument $v_- \leq v_{s_1}$.

  Last we show $v_+ \leq v_{s_2}$.  Consider the sequence $w_n(x+k_n) \to v$ as $n \to \infty$ (along a subsequence).  We know that $w_n(x+\ell_n) \to v_{s_2}$ as $n \to \infty$ for some other sequence of translations $\ell_n \in \Z^d$.  Note that $\ell_n \cdot \xi, k_n \cdot \xi$ must converge in $\R$ since $w_n$ are strictly monotone in the $\xi$ direction for $n$ sufficiently large. If $(\ell_n - k_n) \cdot \tau$ remains bounded along any subsequence then $w_n(x+\ell_n)$ converges to a lattice translation of $v$, this is not possible since $v \neq v_{s_2}$.  Otherwise $\lim_{n \to \infty} |(\ell_n- k_n) \cdot \tau | = \infty$. First lets suppose the limit is $-\infty$.  Then for any $m \perp \xi$ with $m \cdot \tau >1$ there is $n$ sufficiently large such that $\ell_n \cdot \tau \leq k_n \cdot \tau - m \cdot \tau$, then by the monotonicity of $w_n$
\[ v_{s_2}(x) = \lim_{n \to \infty} w_n(x+\ell_n) \leq \lim_{n \to \infty} w_n(x+k_n-m ) = v(x-m). \]
Sending $m\cdot \tau \to \infty$ we find $v_{s_2} \leq v_- < v_{s_2}$ which is not the case. Thus the limit $\lim_{n \to \infty} (\ell_n- k_n) \cdot \xi^\perp = +\infty$. Then for any $m \cdot \tau >1$ we find
\[ v_{s_2}(x) = \lim_{n \to \infty} w_n(x+\ell_n) \geq \lim_{n \to \infty} w_n(x+k_n+m ) = v(x+m). \]
Sending $m\cdot \tau \to \infty$ we find $v_{s_2} \geq v_+$, this was the desired result.
\end{proof}

\begin{proof}[Proof of \lref{sweepoutreg}]
Suppose that $v_{s_j}$ is a sequence of $s_j \in S$ with $s_j \to s$.  Without loss we can assume $s_j < s$, otherwise just split into two subsequences and argue separately, so any subsequential limit of the $v_{s_j}$ is $\leq v_s$.  The $v_{s_j}$ are periodic with respect to $\xi^\perp$, uniformly Lipschitz continuous, and due to the remark 
\[ |v_{s_j}(x) - (\alpha x \cdot \hat{\xi} + s_j)_+| \leq C \exp(-C(\alpha x \cdot \hat{\xi} + s_j)_+/|\xi|)\]
with constants independent of $j$.  Thus any subsequential limits are uniform on $\R^d$.  Again by the uniform estimate above any subsequential limit $v$ of the $v_{s_j}$ will have 
\[  |v(x) - (\alpha x \cdot \hat{\xi} + s)_+|  \to 0 \ \hbox{ as } \ x \cdot \hat\xi \to \infty\]
the same as $v_s$.  As in part 2 of the proof of \lref{sweepoutexistence}, since $v \leq v_s$ and both have the same boundary layer limit they must agree.  Finally since $s \mapsto v_s$ is a continuous $|\xi|^{-1}$-periodic function $\R \to C(\R^d)$ (with supremum norm) it is uniformly continuous.
\end{proof}

\subsection{Left and right limits of $Q_*,Q^*$ exist}  Using the oriented sweepouts constructed in the previous section we can prove that left and right limits of $Q^*,Q_*$ exist at rational directions.  The proof is quite analogous to the proof of continuity of $Q^*,Q_*$ at irrational directions.

\begin{proof}[Proof of \tref{lrcont}]
We just consider the case of a left limit for $Q_*$, the right limit and $Q^*$ cases are similar.  Let $\xi \in \Z^2 \setminus \{0\}$ irreducible and call $p = \xi/|\xi|$ the unit vector in the same direction.  As in the proof of \tref{irrcont} we construct a plane-like subsolution at a nearby direction $q$ with $q \cdot p^\perp < 0$ and with slope slightly larger than $Q_{*,\ell}(p)$.  

Let $(S,\{v_s\}_{s \in S})$ be the left oriented sweepout with slope $Q_{*,\ell}(p)$ which is given by \lref{sweepoutexistence}.  Let $\ep >0$, by \lref{sweepoutreg} there is $\delta>0$ such that if
\begin{equation}\label{e.contpartineq}
\hbox{ $s,s' \in S$ with $|s-s'| \leq \delta$ then $\sup |v_s - v_{s'}| \leq \ep$.}
\end{equation}  
We will always assume $q \cdot p^{\perp} \geq 1/2$, but will make further requirements based on $\ep$ later.

We divide $S$ up into a discrete and continuous part
\[ S_{cont} = \bigcup \{ (s,s'] : 0 \leq s' - s \leq \delta/3, \ s,s' \in S\} \ \hbox{ and } \ S_{disc} = S \setminus S_{cont}.\]
Note that $S_{cont}$ is not really a subset of $S$, but every point of $S_{cont}$ is at most distance $\delta/4$ from a point of $S$.  Viewed as a subset of the torus $\R / |\xi|^{-1}\Z$, $S_{cont}$ is a finite union of half-open intervals and $S_{disc}$ is a finite set of points.  Now create a partition $s_0<\cdots<s_N$, $s_N = s_0 + |\xi|^{-1}$, of the unit periodicity cell of $S$ by points of $S$ in the following way, include all the points of $S_{disc}$, for each (maximal) interval $(a,b]$ of $S_{cont}$ there is a finite partition by points of $S$ such that every interval of the partition has length at most $\delta$ and at least $\delta/3$.  More precisely, given $s_j \in [a,b)$ we know $S \cap s_j+(\delta/3,2\delta/3]$ is nonempty so choose $s_{j+1}$ maximal from the set, unless $b \in (2\delta/3,\delta]$ in which case choose $s_{j+1} = b$.

Now consider the collection of kink-type solutions connecting the points of $S_{disc}$, $v_{s_{j-1},s_j}$ with $s_j \in S_{disc}$.  Recall from \dref{sweepout} that for each $j$ there is $r = r(\ep)>0$ such that
\begin{equation}\label{e.leftineq}
 v_{s_{j-1},s_j}(x) \geq v_{s_{j-1}}(x) - \ep \ \hbox{ for } \ x \cdot p^\perp \leq - r/3
 \end{equation}
and
\begin{equation}\label{e.rightineq}
 v_{s_{j-1},s_j}(x) \leq v_{s_{j}}(x) + \ep \ \hbox{ for } \ x \cdot p^\perp \geq  r/3. 
 \end{equation}
Since this collection is finite there is an $r(\ep)$ which works for all $v_{s_{j-1},s_j}$, $s_j \in S_{disc}$.  Without loss we can assume that this $r(\ep) \geq C/\ep$, and fix it from here on.

Now use the bending sup-convolutions of \sref{bending} to create a subsolution.  Write the hyperplane $q \cdot x = 0$ as a graph over $x \cdot p = 0$ by 
\[ \tau \mapsto x_\tau = \tau p^\perp - \tau \frac{q \cdot p^\perp}{q \cdot p} p. \] 
Pull back the partition $\{s_j\}_{j \in \Z}$ of the range into the domain
\[ \tau_j =  -\frac{q \cdot p}{q \cdot p^\perp} s_j \] 
which is well defined and still an increasing sequence since $q \cdot p^\perp <0$.  Now we want that $\tau_j - \tau_{j-1} \geq r$ for each $s_j \in S_{disc}$, we will enforce it actually for all $j$. For this is suffices that $|q \cdot p^{\perp}| \leq  \delta/6r$ since
\[ \tau_j  - \tau_{j-1} = -\frac{q \cdot p}{q \cdot p^\perp} (s_j-s_{j-1}) \geq \tfrac{1}{3} \delta |q \cdot p|/|q \cdot p^\perp| \geq r.\]
Note that since $\delta,r$ are already fixed depending on $\ep$, the requirement on the size of $|q \cdot p^{\perp}|$ is also a function only of $\ep$.  Choosing $r$ larger if necessary depending on $|\xi|$, we can choose $\ell_j$ to be an integer multiple of $|\xi|$ such that $ r/3 \leq \min\{\ell_j - \tau_j ,\tau_{j+1}-\ell_j\}$.

We use the bending sup-convolutions again as in \sref{bending}, let $\varphi = \ep\varphi_1(\cdot/r)$, defined as above in \eref{bendingphi} with the parameter $M$ still to be chosen (it will be chosen universal).  For each $j \in \Z$, if $s_j \in S_{disc}$ then
\[\tilde{w}_{j}(x) =   \overline{v}_{s_{j-1},s_{j}}^{\varphi(\cdot - x_{\tau_j})}(x-\ell_j p^\perp) \]
while if $s_j \in S_{cont}$
\[ \tilde{w}_{j}(x) =   \overline{v}_{s_j}^{\varphi(\cdot - x_{\tau_j})}(x).\]
Each $\tilde{w}_j$ is subharmonic in its positivity set and 
\begin{equation}\label{e.subsoln11}
 |\grad \tilde{w}_j| \geq (1-CM\ep)(Q(x) - 2\|\grad Q\|_\infty M\ep) \ \hbox{ on } \ \partial \{\tilde{w}_\tau >0\}.
 \end{equation}

Then localize each $\tilde{w}_{j}$ to a vertical strip near $x  \cdot p^\perp = \tau_j $
\[ w_{j}(x) = \begin{cases} \tilde{w}_{j}(x) &\hbox{ if } \tau_{j-1}\leq  x \cdot p^\perp \leq \tau_{j+1} \\
-\infty &\hbox{ else.}
\end{cases}\]
Finally define
\begin{equation*}
 w(x) =  \max\{ \max_{j \in \Z} w_{j}(x), (1+\ep)(Q_{*,\ell}(p)x \cdot q - C_0)_+\}
 \end{equation*}
although this appears to be a maximum over an infinite set, at each $x$ only three of the $w_{j}(x)$ take a finite value.  The constant $C_0$, depending on universal parameters, will be specified below.  We will show that
\begin{equation}\label{e.bdrywjregion}
 w(x) = \max\{w_{j-1}(x),w_j(x),w_{j+1}(x)\} = w_j(x) \ \hbox{ on } \ x \cdot p^\perp = \tau_j, \ x \cdot p \leq C/\ep
 \end{equation}
and 
\begin{equation}\label{e.farbdryregion}
  w(x) = (1+\ep)(Q_{*,\ell}(p)x \cdot q - C_0)_+ \ \hbox{for $x \cdot p \geq C/\ep$},
 \end{equation}
once these two are shown then $w$ defined as above will be continuous subsolution.  Since $w$ will satisfy \eref{subsoln11} on the free boundary, $(1+C\ep) w$ will be a subsolution of \eref{cell} with slope $(1+C\ep)Q_{*,\ell}(p)q$ showing that $Q_*(q) \leq (1+C\ep)Q_{*,\ell}(p)$.

The proof of \eref{farbdryregion} is basically the same as in the proof of \tref{irrcont} so we skip it.  Now consider \eref{bdrywjregion}.  Let $x \cdot p \leq C/\ep$ with $x \cdot p^\perp = \tau_j$, then, if $s_j \in S_{cont}$,
\[ w_j(x) \geq v_{s_j}(x) + c\varphi(x - x_{\tau_j}) \geq v_{s_{j-1}}(x) -\ep+cM\ep,\]
using \lref{bendingphi} and \eref{contpartineq}, or in the case $s_j \in S_{disc}$
\[ w_j(x) \geq v_{s_{j-1},s_j}(x - \ell_j p^\perp) +cM\ep \geq v_{s_{j-1}}(x) - \ep + cM\ep \]
using again \lref{bendingphi} and
 \[ (x - \ell_j p^\perp) \cdot p^\perp = -(\tau_{j+1} -\ell_j) \leq -r/3\]
  so \eref{leftineq} applies.
 
For $w_{j+1}(x)$, if $s_{j+1} \in S_{cont}$
 \begin{align*}
  w_{j+1}(x) &\leq v_{s_{j+1}}(x) +C \varphi(x-x_{\tau_{j+1}}) \\
  &\leq v_{s_j} (x) +C \varphi(x-x_{\tau_{j+1}}) \\
  &\leq  v_{s_j} (x)+ C_0\ep 
  \end{align*}
 using the monotonicity of $v_{s}$ and \lref{bendingphi}.  If $s_{j+1} \in S_{disc}$ then
\[ w_{j+1}(x) \leq v_{s_{j+1},s_j}(x -  \ell_j p^\perp)+C\ep \leq v_{s_j}(x) + C_0\ep\]
using again \lref{bendingphi} and the monotonicity.

For $w_{j-1}(x)$, if $s_{j-1} \in S_{cont}$
\[ w_{j-1}(x) \leq v_{s_{j-1}}(x)+C\varphi(x -x_{\tau_{j-1}}) \leq v_{s_{j-1}}(x)+C_0\ep \]
while if $s_{j-1} \in S_{disc}$
\[ w_{j-1}(x) \leq v_{s_{j-1},s_{j-2}}(x - \ell_{j-1} p^\perp) + C\ep \leq v_{s_{j-1}}(x)+\ep + C\ep  \]
using again \lref{bendingphi} and
\[ (x - \ell_{j-1} p^\perp) \cdot p^\perp = (\tau_j - \ell_{j-1}) \geq r/3\]
so \eref{rightineq} applies.
 
Combining all the above, if we choose $M \geq C_0/c_0$ then we have confirmed \eref{bdrywjregion}.

\end{proof}

\section{Minimal supersolutions / maximal subsolutions}\label{s.minimalsoln}

In this section consider the minimal supersolutions / maximal subsolutions of \eref{fb0} in the complement of a convex obstacle.  Then the existence of a recovery sequence for general solutions of the augmented pinning problem \eref{fbconvex} will follow from a simple argument.  This will prove \tref{main}.

Let $U \subset \R^d$ be outer regular with $\R^d \setminus U$ convex and compact.  Consider the minimal supersolutions and maximal subsolutions of

\begin{equation}\label{e.exteriorep}
\begin{cases}
\Delta u = 0 &\hbox{in } \{ u>0\}\cap U \\
|\grad u | = Q(x/\ep ) &\hbox{on } \partial \{ u>0\} \cap U \\
u = 1 &\hbox{on } \R^d \setminus U.
\end{cases}
\end{equation}
We aim to show that the sequence of minimal supersolutions converges to the solution $\underline{u}$ of
\begin{equation}\label{e.exterior}
\begin{cases}
\Delta u = 0 &\hbox{in } \{ u>0\}\cap U \\
|\grad u | = Q^*(\grad u ) &\hbox{on } \partial \{ u>0\} \cap U \\
u = 1 &\hbox{on } \R^d \setminus U.
\end{cases}
\end{equation}
For the sequence $u^\ep$ of maximal subsolutions the goal is, instead, to show that $(\overline{u} - u^\ep)_+ \to 0$ uniformly where $\overline{u}$ solves
\begin{equation}\label{e.exterior}
\begin{cases}
\Delta u = 0 &\hbox{in } \{ u>0\}\cap U \\
|\grad u | = Q_{*,cont}(\grad u ) &\hbox{on } \partial \{ u>0\} \cap U \\
u = 1 &\hbox{on } \R^d \setminus U.
\end{cases}
\end{equation}
The asymmetry between the results has to do with the convexity assumption on $\R^d \setminus U$.  If we imposed that $U$ is convex and compact instead the results would be reversed.  The more difficult part is the convergence of the maximal subsolutions, however all of the hard work was already done in \sref{contpart} and \sref{irr}. At this stage the proof is a relatively easy application of the definition of the continuous pinning interval \dref{contpinned}.

The main ideas to prove the convergence of the minimal supersolution outside of a convex obstacle have already been developed in the author's previous work with Smart~\cite{FeldmanSmart}.  The main work is to give the correct subsolution property satisfied by the minimal supersolution, and then to prove a comparison principle.  Basically we are defining a notion of viscosity solution problems of the form
\[ \Delta u = 0 \ \hbox{ in } \ \{u>0\}, \ \hbox{ and } \ H(\grad u ) = 1 \ \hbox{ on } \ \partial \{u>0\} \]
when the free boundary condition $H(p)$ is only lower semi-continuous in the gradient variable. Those results are recalled below.

\subsection{Viscosity solutions with discontinuous Hamiltonian}

\begin{definition}\label{d.*super}
  A supersolution of \eref{exterior} is a function $u \in LSC(\R^d)$ that is compactly supported, satisfies $u \geq \id_{\R^d \setminus U}$, and such that, whenever $\varphi \in C^\infty(\R^d)$ touches $u$ from below in $U$, there is a contact point $x$ such that either
  \begin{equation*}
    \Delta \varphi(x) \leq 0
  \end{equation*}
  or
  \begin{equation*}
    \varphi(x) = 0 \quad \mbox{and} \quad |\grad \varphi(x)| \leq Q^*(\grad \varphi(x)).
  \end{equation*}
\end{definition}
It is standard to check, by Perron's method, that there is a minimal supersolution of \eref{exterior} and it satisfies the following subsolution condition.
\begin{definition}\label{d.*fullsub}
  A subsolution of \eref{exterior} in a function $u \in USC(\R^d)$ that is compactly supported, satisfies $u \leq \id_{\R^d \setminus U}$ and such that, whenever $\varphi \in C^\infty(\R^d)$ touches $u$ from above in $\overline{\{ u > 0 \}}\cap D$ some $D \subset U$ open, there is a contact point $x$ such that either
  \begin{equation*}
    \Delta \varphi(x) \geq 0,
  \end{equation*}
  or $\varphi(x) =0$ and
  \begin{equation*}
    |\grad \varphi(x)| \geq \liminf_{y \to x}Q^*(\grad \varphi(y)).
  \end{equation*}
\end{definition}
Again one can check by standard techniques that the maximal subsolution, in the sense of \dref{*fullsub}, of \eref{exterior} is a supersolution.  

In a convex setting a weaker subsolution condition is sufficient to identify the minimal supersolution.  Basically, the free boundary condition only needs to be checked by linear test functions.

\begin{definition}\label{d.*wksub}
  A weakened subsolution of \eref{exterior} in a function $u \in USC(\R^d)$ that is compactly supported, satisfies $u \leq \id_{\R^d \setminus U}$ and such that, whenever $\varphi \in C^\infty(\R^d)$ touches $u$ from above in $\overline{\{ u > 0 \}}\cap D$ some $D \subset U$ open, there is a contact point $x$ such that either
  \begin{equation*}
    \Delta \varphi(x) \geq 0,
  \end{equation*}
  or $\varphi(x) =0$ and either
  \begin{equation*}
    |\grad \varphi(x)| \geq Q^*(\grad \varphi(x)),
  \end{equation*}
  or
  \begin{equation*}
    \nabla \varphi(U) \mbox{ contains two linearly independent slopes.}
  \end{equation*}
\end{definition}
One of the main results of \cite{FeldmanSmart} was comparison principle / uniqueness for the above notion of solution when the set $\R^d \setminus U$ is compact and convex.
\begin{theorem}[Theorem 3.19 of \cite{FeldmanSmart}]\label{t.convexcompare}
Suppose $\R^d \setminus U$ is compact, convex, and inner regular.  There exists a unique $u$ which is a supersolution and a weakened subsolution of \eref{exterior}.  Moreover $\{ u>0\}$ is convex.
\end{theorem}
In particular the minimal supersolution of \eref{exterior}, when $\R^d \setminus U$ is convex, is the same as the maximal subsolution, in the sense of \dref{*wksub}.  Furthermore, given a supersolution $u$ of \eref{exterior}, one only needs to check the weakened subsolution condition \dref{*wksub} to see that $u$ is minimal.

Thus, in the convex setting, to show the convergence of the minimal supersolutions $u^\ep$ to \eref{exteriorep} to the minimal supersolution $u$ of \eref{exterior}, we just need to show the supersolution and weakened subsolution property for any subsequential limit of the $u^\ep$.
\begin{proposition}\label{p.minimalsuperconv}
Let $u^\ep$ be the minimal supersolution of \eref{exteriorep}.  If $u^\ep \to u$ uniformly along some subsequence $\ep \to 0$ then $u$ is a supersolution and weakened subsolution of \eref{exterior}.
\end{proposition}
\begin{proof}
The supersolution property has already been established in \sref{limitsofep}, and $u$ harmonic in $\{u>0\}$ is standard.  Note that, by \lref{nondegen}, the uniform convergence $u^\ep \to u$ also implies that the free boundaries $\partial \{u^\ep >0\}$ converge in Hausdorff distance to $\partial \{u>0\}$.  Suppose that $\varphi = p \cdot (x-x_0)$ touches $u$ from above in $\overline{\{ u > 0 \}}\cap D$ for some open $D \subset U$ with 
\[  |p| < Q^*(p)\]
for some $p \in \R^d \setminus \{0\}$.  We may assume that $\overline{D}$ is compact since $\overline{\{u>0\}}$ is compact.  By the strong maximum principle the contact set is a compact subset of $\partial \{u>0\} \cap D$.  By the strict ordering on $\partial D$ we may choose $\delta>0$ so that $\{u > \varphi - \delta\} \cap \overline{\{ u>0\}} \cap D$ is compactly contained in $D$.

Let $v = \frac{|p|}{Q^*(p)}v^*$ where $v^*$ is a plane-like solution with slope $Q^*(p)$, then $v$ is a supersolution of \eref{unitscale} since $|p|/Q^*(p) <1$.  There is a sequence of points $k_n \in \Z^d$ such that
\[ \ep v(\tfrac{x-\ep k_n}{\ep}) \to (\varphi - \delta)_+ \ \hbox{ uniformly in $D$}\]
and the free boundaries converge in the Hausdorff distance.  Thus $\{v^\ep < u^\ep\} \cap \{u^\ep >0\} \cap D$ is nonempty for sufficiently small $\ep>0$.  Then 
\[
w^\ep = \begin{cases}
v^\ep \wedge u^\ep & x \in D \\
u^\ep & x \notin D
\end{cases}
\]
is a strictly smaller supersolution than $u^\ep$ of \eref{exteriorep}.  This is a contradiction.
\end{proof}

\subsection{Augmented pinning problem}\label{s.augpinning}  In this section we introduce a free boundary problem with pinning interval, with some additional information augmenting the free boundary condition.  We will motivate this problem by deriving it as a limit of spatially homogeneous problems.  

Suppose that we are given $[Q_*,Q^*]$ satisfying all the properties of \tref{mainQcont}.  That is 
\begin{enumerate}[label=(\alph*)]
\item $Q_*,Q^*$ are respectively lower and upper semi-continuous on $S^{d-1}$.
\item There is some number $\langle Q^2 \rangle^{1/2} \in [Q_*(e),Q^*(e)]$ for all $e \in S^{d-1}$.
\item\label{part.augiii} Left and right limits of $Q_*,Q^*$ exist at every $e \in S^{d-1}$ and $Q_*,Q^*$ are continuous at irrational directions.
\end{enumerate}
Then we augment this with a continuous pinning interval $[Q_{*,cont},Q^*_{cont}]$ satisfying
\begin{enumerate}[resume,label=(\alph*)]
\item\label{part.augiv}  $Q_{*,cont},Q^*_{cont}$ are respectively upper and lower semi-continuous on $S^{d-1}$.
\item\label{part.augv}  For all $e \in S^{d-1}$, $\langle Q^2 \rangle^{1/2} \in [Q_{*,cont}(e),Q^*_{cont}(e)] \subset[Q_*(e),Q^*(e)]$.
\item\label{part.augvi} Left and right limits of $Q_{*,cont},Q^*_{cont}$ exist at every $e \in S^{d-1}$ and $[Q_{*,cont},Q^*_{cont}] = [Q_{*},Q^*]$ at irrational directions.
\end{enumerate}
Note that, combining assumptions, the left and right limits of $Q_{*,cont}$ and $Q^*_{cont}$ at a rational direction agree with the corresponding left and right limits of $Q_*,Q^*$.

We do not claim to completely classify the limit shapes.  We will just consider the exterior case $\R^d \setminus U$ is convex and compact, analogous results hold for the interior case $\overline{U}$ convex and compact. Consider the problem
\begin{equation}\label{e.exterioraug}
\begin{cases}
\Delta u = 0 &\hbox{in } \{ u>0\}\cap U \\
|\grad u | \in[Q_{*,cont}(\grad u), Q^*(\grad u)] &\hbox{on } \partial \{ u>0\} \cap U \\
u = 1 &\hbox{on } \R^d \setminus U.
\end{cases}
\end{equation}
Note that, unlike in \eref{fb1}, the subsolution condition is upper semi-continuous.  This means we need to be careful with the notion of subsolution.

\begin{definition}\label{d.*convaugsub}
  A subsolution of \eref{exterioraug} in a function $u \in C(\R^d)$ supported in a compact convex domain $\overline{\{u>0\}}$, satisfying $u \leq \id_{\R^d \setminus U}$ and such that, whenever $\varphi \in C^\infty(\R^d)$ touches $u$ from above in $\overline{\{ u > 0 \}}\cap D$ some $D \subset U$ open, there is a contact point $x$ such that either
  \begin{equation*}
    \Delta \varphi(x) \geq 0,
  \end{equation*}
  or $\varphi(x) =0$ and
  \begin{equation*}
   |\grad \varphi(x)| \geq   \liminf_{\{u>0\} \ni y \to  x}Q_{*,cont}(\grad \varphi(y)) .
  \end{equation*}
\end{definition}

\subsection{Limit shapes of local minimizers}
\begin{proposition}\label{p.mainprop}
Let $U$ such that $\R^d\setminus U$ is convex. 
\begin{enumerate}[label=(\roman*)]
\item Let $u^\ep$ be the minimal supersolution of \eref{exteriorep}. Then $u^\ep \to \underline{u}$ where $\underline{u}$ is the minimal supersolution of \eref{exterioraug}.
\item Let $u^\ep$ be the maximal subsolution of \eref{exteriorep}. Suppose that $u^\ep \to u$ along some subsequence, then $u \geq \overline{u}$ where $\overline{u}$ is the maximal subsolution of \eref{exterioraug}.
\item Let $u$ be a solution of \eref{exterioraug} such that $\{u \geq \lambda\}$ is compact and convex for all sufficiently small $\lambda >0$.  Then there exists a sequence $u^\ep$ solving \eref{exteriorep}, local energy minimizers in the sense of \sref{energysetup}, such that $u^\ep \to u$ as $\ep \to 0$.
\end{enumerate}
\end{proposition}
\begin{proof}
We have already addressed the convergence of the minimal supersolutions above in \pref{minimalsuperconv}.  First we prove the result on the maximal subsolutions, then we show how the first two parts imply the third.

1. Let $u^\ep$ be the maximal subsolution of \eref{exteriorep} and suppose that $u^\ep \to u$ uniformly along some subsequence (not relabeled).  The maximal supersolution $\overline{u}$ of \dref{*convaugsub} is also the minimal supersolution of \eref{exterior} with $Q_{*,cont}$. We aim to show the supersolution property
\[|\grad u| \leq Q_{*,cont}(\grad u) \ \hbox{ on } \ \partial \{u>0\}. \]
Then we will find $u \geq \overline{u}$.

Let $\varphi \in C^{\infty}(\R^d)$ touch $u$ from below in $D$ at a free boundary point $x_0$ for some open $D \subset U$, without loss take $x_0 = 0$.  Call $e$ to be the unit vector in the direction $\grad \varphi(x_0)$.  Suppose that
\[  \Delta \varphi(0) < 0 \ \hbox{ and } \ |\grad \varphi|(0) > \alpha > Q_{*,cont}(e).\]
Let $\delta>0$, we can assume, by shrinking $D$, perturbing $\varphi$ by a quadratic, and making a small translation in the $e$ direction, that 
\[ |\grad\varphi| \geq \alpha  \ \hbox{ on $\partial \{\varphi>0\} \cap D$,} \ \hbox{ $\varphi(0) >u(0) = 0$,} \ \varphi \prec u \ \hbox{ on } \partial D, \]
and
\[ |\tfrac{\grad \varphi}{|\grad \varphi|}(x) - e| \leq \delta \ \hbox{ for } \ x\in D.\]
From the definition of $Q_{*,cont}(e)$, if $\delta>0$ is sufficiently small $\varphi$ has a recovery sequence $v^\ep$ subsolutions of \eref{fb0} in $D$ with
 \[ \lim_{\ep \to 0 } \inf_D (v_\ep - \varphi_+) \geq 0, \ \lim_{\ep \to 0}\sup_{\partial D}|v_\ep- \varphi| =0\]
 and
 \[\lim_{\ep \to 0} d_H(\{v^\ep>0\} \cap D,  (\{v^\ep >0\} \cup \{ \varphi>0\}) \cap D) = 0. \]
In particular $v_\ep \prec u$ on $\partial D$ and $v^\ep(0) > u(0) = 0$ for $\ep>0$ sufficiently small, and therefore,
\[ \overline{u}^\ep(x) = \begin{cases}
v^\ep(x) \vee u^\ep(x) & \hbox{for } x \in D \\
u^\ep & \hbox{for } x \in U \setminus D
\end{cases}\]
is a strictly larger subsolution than $u^\ep$, which is a contradiction.

2.   Without loss $0$ is in the interior of $\R^d \setminus U$.  Choose a level set $\{ u \geq \lambda\}$ convex, and rescale
\[K_\lambda = (1-C_0\lambda)\{ u \geq \lambda\} \ \hbox{ and } \ K^\lambda = (1+C_0\lambda)\{u \geq \lambda\}. \]
 The universal constant $C_0$ will be made precise below.  Then define $u^\lambda$ and $u_\lambda$ mapping $\R^d \to [0,\lambda]$ to be, respectively, the minimal supersolution and maximal subsolution of \eref{exterioraug} in the respective domains $U^\lambda = \R^d \setminus K^\lambda$ and $U_\lambda = \R^d \setminus K_\lambda$ with
\[ u^\lambda = \lambda \ \hbox{ on } \ K^\lambda \ \hbox{ and } \ u_\lambda = \lambda \ \hbox{ on } \ K_\lambda.\]
Similarly let $u^{\lambda,\ep}$ and $u^\ep_\lambda$ be, respectively, the minimal supersolution and maximal subsolution of \eref{exteriorep} respectively in $U_\lambda$ and $U_\lambda$ with data
\[ u^{\lambda,\ep} = \lambda \ \hbox{ on } \ K^\lambda \ \hbox{ and } \ u^\ep_\lambda = \lambda  \ \hbox{ on } \ K_\lambda.\]
As shown in \tref{convexcompare} both $\{u^\lambda>0\}$ and $\{ u_\lambda >0\}$ are convex.  Furthermore 
\[ u^\lambda \prec u(\tfrac{1}{1+C_0\lambda} \cdot) \ \hbox{ in } \ U^\lambda,\]
because $u(\frac{1}{1+C_0\lambda} \cdot)$ is a strict supersolution and $u^\lambda$ is the minimal supersolution and they agree on $\partial K^\lambda$, similarly $ u(\frac{1}{1-C_0\lambda} \cdot) \prec u_\lambda$ in $U_\lambda$.

Non-degeneracy follows from \lref{nondegen}, and then using the upper bound $u^\lambda,u_\lambda \leq \lambda$
\[ d_H(\partial \{u_\lambda^\ep>0\},\partial K_\lambda) +d_H(\partial \{u^{\lambda,\ep}>0\},\partial K^\lambda)\leq C\lambda\]    
with $C$ universal.  Now $C_0$ is chosen so that
\[ \{u^{\lambda,\ep}>0\} \subset K_\lambda + B_{C\lambda} \subset K^\lambda \subset \subset \{u^{\lambda,\ep}>0\}.\]
This is possible because of the star-shapedness of $K$ with respect to a neighborhood of the origin.

As shown above $u^{\lambda,\ep} \to u^\lambda$ and $\liminf_{\ep \to 0} u^\ep_\lambda \geq u_\lambda$ as $\ep \to 0$ uniformly in $\R^d$.
 Thus, by nondegeneracy \lref{nondegen}, for $\ep>0$ sufficiently small
 \[ \{u^{\lambda,\ep}>0\} \subset \subset \{ u(\tfrac{1}{1+C_0\lambda} \cdot)>0\}  \ \hbox{ and } \ \{ u(\tfrac{1}{1-C_0\lambda} \cdot)>0\}\subset\subset \{u^\ep_\lambda>0\}\]
 and by maximum principle
 \[ u^{\lambda,\ep} \prec u(\tfrac{1}{1+C_0\lambda} \cdot) \ \hbox{ in $U^\lambda$ and } \ u(\tfrac{1}{1-C_0\lambda} \cdot) \prec u^\ep_\lambda \ \hbox{ in } \ U_\lambda. \]
  Extend $u^{\lambda,\ep}$ and $u_\lambda^\ep$ to $U$ by
\[ \overline{v}^\ep(x) = 
\begin{cases}
u^{\lambda,\ep}(x) & x \in U^\lambda \\
u(\tfrac{1}{1+C_0\lambda}x) & x \in  K^\lambda
\end{cases}
\ \hbox{ and } \ 
\underline{v}^\ep(x) = 
\begin{cases}
u^\ep_\lambda(x) & x \in U_\lambda \\
u(\tfrac{1}{1-C_0\lambda}x) & x \in K_\lambda.
\end{cases}
 \]
 Then the superharmonic/subharmonic properties of $\overline{v}^\ep$ and $\underline{v}^\ep$ are easily checked, for $x \in \partial K^\lambda$ and $r>0$ sufficiently small
 \[ \overline{v}^\ep(x) =u(\tfrac{1}{1+C_0\lambda}x) = \frac{1}{|B_r|}\int_{B_r} u(\tfrac{1}{1+C_0\lambda}y) \ dy \geq \frac{1}{|B_r|}\int_{B_r} \overline{v}^\ep(y) \ dy \]
 and similar for $\underline{v}^\ep$.
 
 Now, using $\underline{v}^\ep \prec\overline{v}^\ep$, we apply \lref{localmin} to find that there is a minimizer $v^\ep$ of the energy $E_\ep(\cdot,U)$ in the constraint set
\[ \mathcal{A} = \{ v \in H^1(U): \underline{v}^\ep \leq  v \leq \overline{v}^\ep \ \hbox{ and } \ v = 1 \ \hbox{ on }  \ \partial U  \} \]
which is, furthermore, a viscosity solution of \eref{exteriorep} satisfying the strict separation
 \[\underline{v}^\ep \prec v^\ep \prec \overline{v}^\ep.\]
  Thus, for any $\ep \leq \ep_0(\lambda)$,
\[  |v^\ep - u| \leq C \lambda \ \hbox{ and } \ d_H(\{v^\ep>0\},\{u>0\}) \leq C\lambda.\]
Since $\lambda >0$ was arbitrary we conclude.
\end{proof}

\appendix

\section{Augmented pinning problem as a limit of spatially homogeneous problems }\label{s.augmentedpinningex}

In this section we derive the augmented pinning problem via a limit of regular spatially homogeneous pinning problems.  This gives at least a plausibility argument for why $Q_{*,cont}$ and/or $Q^*_{cont}$ may be nontrivially different from the upper and lower semicontinuous envelopes of $Q_*$ and $Q^*$.

Consider a natural limiting procedure to derive \eref{fb1}, one might choose to regularize the jump discontinuities of $[Q_*,Q^*]$.  It is natural to do this in a monotone way by an inf/sup convolution.  We define the inf and sup convolving operators $\Box_{*,n}$ and $\Box^*_{n}$ respectively on $LSC(S^{d-1})$ and $USC(S^{d-1})$
\[ \Box_{*,n} f(e):=  \inf_{e'\in S^{d-1}}\{ f(e') + n|e'-e|\} \ \hbox{ and } \ \Box^*_{n}f(e) = \sup_{e'\in S^{d-1}}\{ f(e') - n|e'-e|\}. \]
Note that $\Box_{*,n} f$ and $\Box^*_{n} f$ are Lipschitz continuous with constant $n$ on $S^{d-1}$.  The natural monotone approximation procedure would be to define
\[ Q_{*,n}(e) = \Box_{*,n}Q_{*}(e) \ \hbox{ and } \ Q^*_n(e) = \Box^*_{n}Q^*_{n}(e) .\]
Basically we are regularizing the discontinuities of $I(e)$, replacing by Lipschitz spikes.  In this case it is straightforward to check that the minimal supersolution $\underline{u}_n$ and maximal subsolution $\overline{u}_n$ associated with $Q_{*,n}$ and $Q^*_n$ do converge, respectively, to the minimal supersolution $\underline{u}$ and maximal subsolution $\overline{u}$ of \eref{exterior}.

Now we consider a different approximation procedure which is not monotone.   Assume that we are given $I(e) = [Q_*(e),Q^*(e)]$ and $I_{cont}(e) = [Q_{*,cont}(e),Q^*_{cont}(e)]$ satisfying the assumptions listed in \sref{augpinning}.  Define
\begin{equation}\label{e.Qm}
Q_{*,m}(e) = 
\begin{cases}
\Box^*_{m}Q_{*,cont}(e) & e \hbox{ irrational} \\
Q_{*}(e) & e \hbox{ rational}
\end{cases} \ \hbox{ and } \ 
Q^*_{m}(e) = 
\begin{cases}
\Box_{*,m}Q^*_{cont}(e) & e \hbox{ irrational} \\
Q^*(e) & e \hbox{ rational}
\end{cases}
\end{equation}
and send $m \to \infty$.  This isn't really a regularization, $Q_{*,m}$ and $Q^*_{m}$ may still have jump discontinuities at rational directions, but one can think of regularizing again
\[ Q_{*,m,n}(e)= \Box_{*,n}Q_{*,m}(e) \ \hbox{ and } \ Q^*_{m,n}(e)= \Box^*_{n}Q^*_{m}(e).\]
and sending both $m,n \to \infty$ but with $m= o(n)$.  

The pinning intervals $I_m(e)$ still converge as $m \to \infty$ pointwise to $I(e)$, however the convergence is no longer monotone.  Are all solutions of \eref{fb1} achieved as a limit of solutions to $\eref{fb1}_m$ for the approximating pinning intervals $[Q_{*,m}(e),Q^*_m(e)]$?  It turns out that the answer is no, limits of solutions to $\eref{fb1}_m$ actually satisfy a stronger condition in general.  

\begin{proposition}
Let $d=2$, $U$ such that $\R^2\setminus U$ is convex.  We refer to $\textit{\eref{exterior}}_m$ for problem \eref{exterior} with pinning interval $[Q_{*,m},Q^*_m]$ as defined in \eref{Qm}.
\begin{enumerate}[label=(\roman*)]
\item Let $u_m$ be the minimal supersolution of $\textit{\eref{exterior}}_m$, then $u_m \to \underline{u}$ uniformly where $\underline{u}$ is the minimal supersolution of \eref{exterioraug}.
\item Let $u_m$ be the maximal subsolution of $\textit{\eref{exterior}}_m$, then $u_m \to \underline{u}$ uniformly where $\underline{u}$ is the maximal subsolution of \eref{exterioraug}.
\item Let $u$ be a solution to \eref{exterioraug} with convex support.  Then there exists a sequence of solutions $u_m$ to $\textit{\eref{exterior}}_m$ such that $u_m \to u$ uniformly as $m \to \infty$.
\end{enumerate}
\end{proposition}
\begin{proof}
We show convergence of the minimal supersolution and maximal subsolution.  The last part follows as in \pref{mainprop}.

First the minimal supersolutions. Suppose that $u_m \to u$ uniformly along some subsequence. By \tref{convexcompare} we just need to check the supersolution and weak subsolution property for $u$.  The supersolution property is easy because of the monotonicity $Q^*_m \nearrow Q^*$.  The weak subsolution property is also easy because we only need to test with linear functions, the convergence $Q^*_m \to Q^*$ pointwise on $S^{d-1}$ is enough.

Now the maximal subsolutions, again we just need to check the subsolution and weakened supersolution condition.  Suppose that $u_m \to u$ uniformly along some subsequence. The supersolution property in the limit is, for any $\varphi$ touching $u$ from below at $x\in U \cap \partial \{u>0\}$ either $\Delta \varphi(x) \leq 0$ or
\[ |\grad \varphi|(x) \leq \limsup_{y \to x, m \to \infty} Q_{*,m}(\grad \varphi(y)). \]
Since $Q_{*,cont}$ is upper semicontinuous for any $\ep>0$ there is a neighborhood $N$ of $\grad \varphi(x)$ such that for $m$ sufficiently large and $e \in N$ we have $Q_{*,m}(e) \leq Q_{*,cont}(\grad \varphi) +\ep$.  Thus
\[ |\grad \varphi|(x) \leq Q_{*,cont}(\grad \varphi(x)).\]
Now we consider the weak subsolution condition, this is the interesting part.  Suppose that $\varphi(x) = (p \cdot x)_+$ touches $u$ from above at $0 \in \partial \{u>0\} \cap U$ in some domain $D \subset U$ with $p$ rational.  Then we can find $x_m \to 0$ such that $\varphi(x - x_m) = [p \cdot (x-x_m)]_+$ touches $u_m$ from above at $x_m \in \partial \{u_m>0\}$.  Since $\{u_m>0\}$ is convex $|\grad u_m|$ is defined and concave on the facet $\{p \cdot (x-x_m) = 0\} \cap \partial \{u_m>0\}$.  For $m$ sufficiently large the left and right limits of $Q_{*,m}$ at $e$ are $Q_{*,cont}(e)$.  We argue below that the facet must be a singleton $\{p \cdot (x-x_m) = 0\} \cap \partial \{u_m>0\}= \{x_m\}$.  This means that given $r>0$ small enough that $B_r(x_m) \subset D$ and for $|q - p|$ sufficiently small $u_m> [q \cdot (x-x_m)]_+$ is compactly contained in $B_r(x_m)$ so for an appropriate choice of $x_m$ now $[q \cdot (x-x_m')]_+$ touches $u_m$ from above in $B_r(x_m)$ at $x_m' \in \partial \{u_m>0\}$.  Therefore
\[ |q| \geq  Q_{*,m}(q)\]
and
\[ |p| \geq \lim_{m \to \infty}\lim_{q \to p, \ q \neq p}Q_{*,m}(q) = Q_{*,cont}(p).\]
The case of irrational $p$ is easy because of the correct monotonicity.

Now we argue that if $u_m$ solves \eref{exterior} with convex support and the left and right limits of $Q_{*,m}$ at $p$ agree, with value $Q_{*,cont}$ in this case, then the facet with normal $p$, call it $\Omega_p$, is trivial.  This fact was used above.  The argument is in Caffarelli-Lee~\cite[Lemma 3.5]{CaffarelliLee}, but it is very brief so we repeat it here with more details. Suppose $\Omega_p$ is a non-trivial line segment, without loss $0 \in \Omega_p$. Then $|\grad u_m|$ is concave on the facet and therefore must be identically equal to $Q_{*,cont}(p)$.  Then extend $u$ by reflection through $\Omega_p$ and subtract off the linear function $Q^*(p)x\cdot p$ to obtain a harmonic function $v$ in an open domain $\R^2 \supset V \supset \Omega_p$ with $v = 0$ and $|\grad v| = 0$ on $\Omega_p$.  We identify $\R^2$ with the complex plane $\C$ and after rotation we can assume that $\Omega_p$ is a segment of the real line.  Then
\[ \varphi = v_y - iv_x\]
is analytic in $V$ and $\varphi = 0$ on $V \cap \R$.  Thus $\varphi \equiv 0$ in $V$ and $u$ is linear with slope $p$ in $\Omega_p$, this is not the case.

\end{proof}

  \bibliographystyle{plain}
\bibliography{pinning_articles}
\end{document}